\RequirePackage{ifthen}
\RequirePackage{ifpdf}
\newcommand{\driverOption}{}
\ifthenelse{\boolean{pdf}}{
  \renewcommand{\driverOption}{pdftex}
} { 
  \renewcommand{\driverOption}{dvips}
}

\documentclass[reqno, 11pt, letterpaper, commented, oneside, \driverOption]{amsart}

\newboolean{isCommented}
\usepackage{geometry}
\usepackage{bbm}
\usepackage[final]{graphicx}
\usepackage{upref}
\usepackage{enumerate}
\usepackage{latexsym}
\usepackage{amssymb}
\usepackage[ansinew]{inputenc}
\usepackage[T1]{fontenc}
\usepackage[ruled,vlined]{algorithm2e}
\usepackage{mathtools}
\usepackage{booktabs}
\usepackage{mycommands}

\ifthenelse{\boolean{pdf}}{
  \usepackage[final]{ps4pdf}
} { 
  \usepackage[inactive]{ps4pdf}
}
\PSforPDF{\usepackage{psfrag}}

\newcommand{\hyperrefDriverOption}{}
\ifthenelse{\boolean{pdf}}{
	\renewcommand{\hyperrefDriverOption}{pdftex}
} { 
	\renewcommand{\hyperrefDriverOption}{hypertex}
}
\usepackage[\hyperrefDriverOption,
  colorlinks = false, 
  pdfauthor={Torsten Mütze, Thomas Rast, Reto Spöhel},
  pdftitle={Coloring random graphs online without creating monochromatic subgraphs}]
  {hyperref}
    
\ifthenelse{\boolean{isCommented}} {
	\newcommand{\TM}[1]{\marginpar{\parbox{4cm}{{\small {\bf TM:} #1}}}}
	\newcommand{\RS}[1]{\marginpar{\parbox{4cm}{{\small {\bf RS:} #1}}}}
	\newcommand{\TR}[1]{\marginpar{\parbox{4cm}{{\small {\bf TR:} #1}}}}
} { 
	\newcommand{\TM}[1]{}
	\newcommand{\RS}[1]{}
	\newcommand{\TR}[1]{}
}

\newtheorem{theorem}{Theorem}
\newtheorem{lemma}[theorem]{Lemma}

\newtheorem{proposition}[theorem]{Proposition}
\newtheorem{corollary}[theorem]{Corollary}

\theoremstyle{definition}
\newtheorem{definition}[theorem]{Definition}

\theoremstyle{remark}
\newtheorem{remark}[theorem]{Remark}

\DeclareMathOperator{\CW}{\textsc{ComputeWeights}}
\DeclareMathOperator{\PAINT}{\textsc{Paint}}
\DeclareMathOperator{\ABUILD}{\textsc{AbstractBuild}}
\DeclareMathOperator{\BUILD}{\textsc{Build}}
\DeclareMathOperator{\ASTRAT}{\textsc{AbstractStrategy}}
\DeclareMathOperator{\STRAT}{\textsc{Strategy}}

\long\def\symbolfootnote[#1]#2{\begingroup
\def\thefootnote{\fnsymbol{footnote}}\footnote[#1]{#2}\endgroup}

\ifthenelse{\boolean{isCommented}} {
	\geometry{
	  hmargin={25mm, 50mm},
	  marginparwidth=40mm, 
	  vmargin={25mm, 25mm},
	  headsep=10mm,
	  headheight=5mm,
	  footskip=10mm
	}
} { 
	\geometry{
	  hmargin={25mm, 25mm}, 
	  vmargin={25mm, 25mm},
	  headsep=10mm,
	  headheight=5mm,
	  footskip=10mm
	}
}

\begin{document}

\begin{center}

\vspace*{1.5cm}
\LARGE Coloring random graphs online without creating \\ monochromatic subgraphs\footnote{An extended abstract of this work has appeared in the proceedings of SODA '11.}
\vspace{1cm}

\small

\begingroup
\renewcommand\thefootnote{\fnsymbol{footnote}}%
\begin{tabular}{l@{\hspace{2em}}l}
\newlength\TIwidth
\setlength\TIwidth{\widthof{Institute of Theoretical Computer Science}}%
\hbox to \TIwidth{\Large Torsten Mütze\hfill Thomas Rast} & \Large Reto Spöhel\footnotemark[2] \\[2mm]
  Institute of Theoretical Computer Science & Algorithms and Complexity Group\\
  ETH Zürich                       &  Max-Planck-Institut für Informatik\\
  8092 Zürich, Switzerland                    & 66123 Saarbrücken, Germany \vspace{.05mm}\\
  {\small \{{\tt muetzet|trast}\}{\tt @inf.ethz.ch}}       & {\small {\tt rspoehel@mpi-inf.mpg.de}}
\end{tabular}%
\footnotetext[2]{The author was supported by a fellowship of the Swiss National Science Foundation.}%
\endgroup

\vspace{2cm}

\small

\begin{minipage}{0.85\linewidth}
\textsc{Abstract.}
Consider the following random process: The vertices of a binomial random graph $G_{n,p}$ are revealed one by one, and at each step only the edges induced by the already revealed vertices are visible. Our goal is to assign to each vertex one from a fixed number $r$ of available colors immediately and irrevocably without creating a monochromatic copy of some fixed graph $F$ in the process.

Our first main result is that for any $F$ and $r$, the threshold function for this problem is given by $p_0(F,r,n)=n^{-1/m_1^*(F,r)}$, where $m_1^*(F,r)$ denotes the so-called \emph{online vertex-Ramsey density} of $F$ and $r$. This parameter is defined via a purely deterministic two-player game, in which the random process is replaced by an adversary that is subject to certain restrictions inherited from the random setting. Our second main result states that for any $F$ and $r$, the online vertex-Ramsey density $m_1^*(F,r)$ is a computable rational number.

Our lower bound proof is algorithmic, i.e., we obtain polynomial-time online algorithms that succeed in coloring $G_{n,p}$ as desired with probability $1-o(1)$ for any $p(n) = o(n^{-1/m_1^*(F,r)})$.
\end{minipage}

\end{center}

\vspace{5mm}

\tableofcontents

\section{Introduction} \label{sec:introduction}
The study of colorability properties of random graphs has a rich history and has spurred many important developments in random graph theory. Thanks to the efforts of many researchers (e.g.,~\cite{MR1662274, MR1606020, MR951992, MR2442592, MR0369129, MR1112273, MR1122014, MR918398, MR1138433}), very precise bounds on the chromatic number of the random graph are known by now. More recently, also several related coloring notions and their associated `chromatic numbers' have been investigated for the random graph (e.g.,~\cite{MR1773652, MR2387558, MR1370970, MR2575099, MR2532875, MR2396353, MR1182457, MR1146898}).

In this work we are concerned with the following generalized notion of graph coloring: a coloring of a graph $G$ is \emph{valid} with respect to some given graph $F$ if it contains no monochromatic copy of $F$, i.e., if there is no copy of $F$ in $G$ whose vertices all receive the same color. Note that a proper coloring in the usual sense is a coloring that is valid with respect to a single edge. More generally, a coloring that is valid with respect to the star with $\ell$ rays is a coloring in which each color class induces a graph with maximum degree at most $\ell-1$ (this is sometimes called an $(\ell-1)$-improper coloring, see~\cite{MR2575099} and references therein).

The main motivation for studying this notion of colorability comes from Ramsey theory, where one usually considers similarly defined \emph{edge}-colorings. The threshold for the existence of a valid vertex-coloring of the random graph with respect to some given fixed graph $F$ was determined by {\L}uczak, Ruci\'nski, and Voigt~\cite{MR1182457}. To state their result, we introduce some terminology.

As usual, we denote by $\Gnp$ the random graph on $n$ vertices (labelled from $1,\ldots,n$) obtained by including each of the $\binom{n}{2}$ possible edges with probability $p=p(n)$ independently.
We say that a graph $G$ is \emph{$(F,r)$-vertex-Ramsey} if every $r$-coloring of the vertices of $G$ contains a monochromatic copy of $F$, i.e., if $G$ does \emph{not} allow a valid $r$-coloring with respect to $F$. A graph is a \emph{matching} if its maximum degree is 1. We denote the number of edges and vertices of a graph $H$ by $e(H)$ and $v(H)$, respectively.

\begin{theorem}[\cite{MR1182457}]\label{thm:Ramsey}
Let $r\geq 2$ be a fixed integer and $F$ a fixed graph with at least one edge that in the case $r=2$ is not a matching.
Then there exist positive constants $c=c(F,r)$ and $C=C(F,r)$ such that
\[
  \lim_{n\to\infty}\PP[\text{$\Gnp$ is $(F,r)$-vertex-Ramsey}]
  =
  \begin{cases}
    0   &\text{if\/ $p(n) \leq cn^{ - {1}/{m_1(F)}}$ \enspace,} \\
    1   &\text{if\/ $p(n) \geq Cn^{ - {1}/{m_1(F)}}$ \enspace,}
  \end{cases}
\]
where
\begin{equation} \label{eq:m1}
    m_1(F) := \max_{H\seq F:\,v(H)\geq 2}\; \frac{e(H)}{v(H)-1}
    \enspace.
\end{equation}
\end{theorem}

Note that the parameter $m_1(F)$ does not depend on $r$. It is widely believed (see e.g.~\cite{MR2116574}) that the threshold behaviour is even sharper than stated in Theorem~\ref{thm:Ramsey}. Friedgut and Krivelevich~\cite{MR1768845} proved this conjecture for the class of strictly $1$-balanced graphs, i.e.\ for graphs $F$ for which $e(H)/(v(H)-1)<e(F)/(v(F)-1)$ for all proper subgraphs $H\sneq F$ with $v(H)\geq 2$. 

Implicit in the lower bound proof of Theorem~\ref{thm:Ramsey} is the existence of a polynomial-time algorithm that \aas succeeds in finding a valid coloring of $\Gnp$ for $p(n)\leq cn^{-1/m_1(F)}$, where polynomial here and throughout means polynomial in $n$ for $F$ and $r$ fixed.
(Here and throughout, \aas stands for asymptotically almost surely, i.e., with probability tending to 1 as $n$ tends to infinity.)

In this work, we study the same coloring problem in an \emph{online setting}, and derive results of the same generality as those stated in Theorem~\ref{thm:Ramsey} for the offline case.

\subsection{The online setting}
\label{sec:introduction-online-setting}

We consider the following online problem: The vertices of an initially hidden instance of $\Gnp$ are revealed one by one in increasing order, and at each step of the process only the edges induced by the vertices revealed so far are visible. Alternatively, one can think of the random edges leading from each vertex to previous vertices as being generated at the moment the vertex is revealed (each edge being inserted with probability $p$ independently from all other edges). Each vertex has to be colored immediately and irrevocably with one of $r$ available colors as soon as it is revealed, with the goal of avoiding monochromatic copies of a fixed graph $F$ as before.

It follows from standard arguments (see~\cite[Lemma 7]{org-lb}) that this online problem has a threshold $p_0(F,r,n)$ in the following sense: For any function $p(n) = o(p_0)$ there is a strategy that \aas finds an $r$-coloring of $\Gnp$ that is valid with respect to $F$ online, and for any function~$p(n)=\omega(p_0)$ \emph{any} online strategy will \aas fail to do so. (Observe that no computational restrictions are imposed in this definition, i.e., the coloring strategy is \emph{not} required to be an efficient algorithm.)

Note that this is a weaker threshold behaviour than the one stated in Theorem~\ref{thm:Ramsey}. A closer inspection of the arguments in this paper shows that the online thresholds are indeed coarser than the offline thresholds given by Theorem~\ref{thm:Ramsey}: the limiting probability that a valid coloring can be found online is a constant bounded away from $0$ and $1$ whenever $p(n)$ has the same order of magnitude as the threshold $p_0(F,r,n)$. This is a consequence of the fact that the online thresholds turn out to be determined by \emph{local} substructures (see~\cite[Theorem 3.9]{MR1782847}).
\RS{`Coarser thresholds': Das steht nicht explizit im Paper, ist aber trotzdem wahr: für $p(n)$ auf dem Threshold taucht mit konstanter WSK keiner der Witnessgraphen auf (folgt aus der FKG-Ungleichung, siehe Theorem 3.9 im Purple Book); Painter hat also eine positive Überlebens-WSK. Andererseits geht jeder Schritt der multi-round exposure im upper-bound-Beweis mit positiver Wahrscheinlichkeit gut (um das zu zeigen, ersetze man einfach die SMM mit Janson); Painter hat also eine positive WSK zu verlieren.}

The online problem was first studied in~\cite{ovg-combinatorica}, where the following simple strategy was analyzed. Assuming that the colors are numbered from $1$ to $r$, the \emph{greedy strategy} fixes an appropriate choice of subgraphs $H_1, \dots, H_r\seq F$, and at each step uses the highest-numbered color $i$ that does not complete a monochromatic copy of $H_i$ (or color $1$ if no such color exists). Note that this strategy can easily be implemented in polynomial time.

For any graph $F$ and any integer $r\geq 1$ we define the parameter $\ol{m}_1(F,r)$ recursively by
\begin{equation}\label{eq:mFr}
	\ol{m}_1(F,r):=
	\begin{cases}
	  \displaystyle \max_{H\seq F}\frac{e(H)}{v(H)} &\text{if $r = 1$} 
	  	\enspace, \\
	  \displaystyle \max_{H\seq F}\frac{e(H)+\ol{m}_1(F,r-1)}{v(H)}
	  	&\text{if $r \geq 2$} \enspace.
	\end{cases}
\end{equation}

The results of~\cite{ovg-combinatorica} can be stated as follows.

\begin{theorem}[\cite{ovg-combinatorica}] \label{thm:ovg-combinatorica}
For any fixed graph $F$ with at least one edge and any fixed integer $r\geq 2$, the threshold for finding an $r$-coloring of\/ $\Gnp$ that is valid with respect to $F$ online satisfies
\begin{equation*}
  p_0(F,r,n)\geq n^{-1/\ol{m}_1(F,r)} \enspace,
\end{equation*}
where $\ol{m}_1(F,r)$ is defined in~\eqref{eq:mFr}. A polynomial-time algorithm that succeeds \aas for any $p(n)=o(n^{-1/\ol{m}_1(F,r)})$ is given by the greedy strategy.

If $F$ has an induced subgraph~$F^\circ\sneq F$ on~$v(F)-1$ vertices satisfying
\begin{equation} \label{eq:two-round-condition}
  m_1(F^\circ) \leq \ol{m}_1(F,2) \enspace,
\end{equation}
where $m_1(F^\circ)$ is defined in~\eqref{eq:m1}, the greedy strategy is best possible, i.e., the threshold is
\begin{equation*}
  p_0(F,r,n)=n^{-1/\ol{m}_1(F,r)} \enspace.
\end{equation*}
\end{theorem}

The second part of Theorem~\ref{thm:ovg-combinatorica} applies in particular to the case where $F$ is a clique or a cycle of arbitrary fixed size (in these specific cases, the appropriate choice of subgraphs $H_1, \dots, H_r\seq F$ for the greedy strategy is $H_1=\dots = H_r=F$). Thus for cliques and cycles, explicit threshold functions are known.

It was also pointed out in~\cite{ovg-combinatorica} that the greedy strategy is \emph{not} best possible in general --- as we shall see below, the threshold is significantly higher than what is guaranteed by Theorem~\ref{thm:ovg-combinatorica} already in the innocent-looking case when $F$ is a long path.

Our main result gives a combinatorial characterization of the online threshold that allows us to compute, for any $F$ and $r$, a value $\gamma=\gamma(F,r)$ such that the threshold is given by $p_0(F,r,n)=n^{-\gamma}$. We also obtain polynomial-time coloring algorithms that \aas find valid colorings of $\Gnp$ online in the entire regime below the respective thresholds, i.e., for any $p(n)=o(p_0)$.

\subsection{A general characterization of the online threshold}

Our main result characterizes the general threshold for the online problem in terms of a \emph{deterministic two-player game}, which we describe in the following. The two players are called \emph{Builder} and \emph{Painter}, and the board is a graph that grows in each step of the game. Painter wants to maintain a valid coloring of the board, and her opponent Builder tries to prevent her from doing so by forcing her to create a monochromatic copy of~$F$.

The game starts with an empty board, i.e., no vertices are present at the beginning of the game. In each step, Builder presents a new vertex and a number of edges leading from previous vertices to this new vertex. Painter has to color the new vertex immediately and irrevocably with one of $r$ available colors, and as before she loses as soon as she creates a monochromatic copy of $F$. Note that so far this is the same setting as before, except that we replaced `randomness' by the second player Builder. (Put differently, if Builder presents every possible edge with probability $p$ independently, this is exactly the online process introduced above.)
However, we additionally impose the restriction that Builder is not allowed to present an edge that would create a (not necessarily monochromatic) subgraph $H$ with $e(H)/v(H) > d$ on the board, for some fixed real number $d$ known to both players. In other words, Builder must adhere to the restriction that the evolving board~$B$ satisfies $m(B)\leq d$ at all times, where as usual we define
\begin{equation*}
  m(B):=\max_{H\seq B} \frac{e(H)}{v(H)} \enspace.
\end{equation*}
We will refer to this game as the \emph{deterministic $F$-avoidance game with $r$ colors and density restriction~$d$}.

We say that \emph{Builder has a winning strategy} in this game (for a fixed graph $F$, a fixed number of colors~$r$, and a fixed density restriction $d$) if he can force Painter to create a monochromatic copy of~$F$ within a finite number of steps. Conversely, we say that \emph{Painter has a winning strategy} if she can avoid creating a monochromatic copy of $F$ for an arbitrary number of steps.
Note that if for some fixed $F$ and~$r$, Builder has a winning strategy for some density restriction $d$, then he also has a winning strategy for every density restriction $d'\geq d$. We say that a Painter or Builder strategy is \emph{optimal} if it is a winning strategy simultaneously for all~$d$ for which the respective player has a winning strategy.

For any graph $F$ and any integer $r\geq 2$ we define the \emph{online vertex-Ramsey density} $m_1^*(F,r)$ as
\begin{equation} \label{eq:m1*}
  m_1^*(F,r):=\inf \left\{
    d\in\RR \bigmid
    \leftbox{0.59\displaywidth}{Builder has a winning strategy in the
      deterministic \mbox{$F$-avoidance} game with $r$ colors and density
      restriction $d$}
  \right\} \enspace.
\end{equation}
It is not hard to see that $m_1^*(F,r)$ is indeed well-defined for any $F$ and $r$.
With these definitions in hand, our results can be stated as follows. 

\begin{theorem} \label{thm:main-1}
For any graph $F$ with at least one edge and any integer $r\geq 2$, the online vertex-Ramsey density $m_1^*(F,r)$ is a computable rational number, and the infimum in \eqref{eq:m1*} is attained as a minimum.
\end{theorem}

\begin{theorem} \label{thm:main-2}
For any fixed graph $F$ with at least one edge and any fixed integer $r\geq 2$, the threshold for finding an $r$-coloring of\/ $\Gnp$ that is valid with respect to $F$ online is
\begin{equation} \label{eq:threshold}
  p_0(F,r,n)=n^{-1/{m_1^*(F,r)}} \enspace,
\end{equation}
where $m_1^*(F,r)$ is defined in \eqref{eq:m1*}.
A polynomial-time algorithm that succeeds \aas for any $p(n)=o(p_0)$ can be derived from one of Painter's optimal strategies in the deterministic two-player game. 
\end{theorem}

Theorem~\ref{thm:main-2} reduces the problem of determining the threshold of the original probabilistic problem to the purely deterministic combinatorial problem of computing $m_1^*(F,r)$ or, informally speaking, of `solving' the deterministic two-player game. According to Theorem~\ref{thm:main-1}, the latter is possible by a finite computation; note that in the asymptotic setting of Theorem~\ref{thm:main-2}, this is in fact a constant-size computation. 

It follows from the results of \cite{ovg-combinatorica} that for any graph $F$ we have
\begin{equation*}
  \lim_{r\to\infty} m_1^*(F,r)=m_1(F) \enspace.
\end{equation*}
Thus the online thresholds approach the offline threshold given by Theorem~\ref{thm:Ramsey} as the number of colors increases. It also follows from an example presented in~\cite{ovg-combinatorica} that if $F$ is the disjoint union of two graphs $H_1$ and $H_2$, the parameter $m_1^*(F,r)$ may be strictly higher than both $m_1^*(H_1,r)$ and $m_1^*(H_2,r)$. (Such a behaviour cannot occur for the parameter $m_1(F)$ appearing in Theorem~\ref{thm:Ramsey}.)

\subsection{Remarks on Theorem~\texorpdfstring{\ref{thm:main-1}}{3}}

To put Theorem~\ref{thm:main-1} into perspective, we mention that none of its three statements (computable, rational, infimum attained as minimum) is known to hold for the offline counterpart of $m_1^*(F,r)$, i.e., for the \emph{vertex-Ramsey density}
\begin{equation*}
  m_1^o(F,r):=\inf\big\{m(G) \,\bigmid[\big]\, \text{$G$ is $(F,r)$-vertex-Ramsey} \big\}
\end{equation*}
introduced in~\cite{MR1248487}. It is also not known whether such statements are true for two analogous parameters related to \emph{edge}-colorings (see~\cite{org-new, MR2152058}).
In fact, even the value of $m_1^o(P_3, 2)$ is unknown 
--- the authors of \cite{MR1248487} offer 400,000 z{\l}oty (Polish currency in 1993) for its exact determination (here $P_3$ denotes the path on three vertices).

As is the case for many parameters in Ramsey theory, computing the online vertex-Ramsey density $m_1^*(F,r)$ becomes intractable already for moderately large graphs~$F$. 
We have an implementation that computes $m_1^*(F,2)$ for all graphs $F$ with at most 7 vertices in under 10 minutes on an ordinary desktop computer. Using more computational power, we managed to determine $m_1^*(F,2)$ exactly for all non-forests on at most 9 vertices.  (Our implementation is rather inefficient for forests --- we believe that for this special case an adapted program would be much faster, see the remarks in Section~\ref{sec:forests} below.) The program can be downloaded from the authors' websites~\cite{homepage-mrs}.
There might be some room for improvement here, but it seems unrealistic to compute $m_1^*(F,2)$ for, say, general graphs with 20 vertices with our approach in reasonable time. (Recall that the Ramsey number $R(5)$, i.e., the smallest integer $\ell$ such that every edge-coloring of the complete graph on $\ell$ vertices contains a monochromatic $K_5$, is still unknown!) 

\subsection{Remarks on Theorem~\texorpdfstring{\ref{thm:main-2}}{4}} \label{sec:remarks-main-2}

The intuition behind Theorem~\ref{thm:main-2} is the following: It is well-known \cite{MR620729} (see also~\cite[Section 3]{MR1782847}) that for any fixed graph $G$, \aas the random graph $\Gnp$ with $p(n)=\omega(n^{-1/m(G)})$ contains a copy of $G$. The textbook second moment method proof of this fact can be adapted to show that for $p(n)=\omega(n^{-1/d})$ and any fixed finite Builder strategy for the deterministic two-player game that respects the density restriction $d$, \aas the random process will exactly reproduce the given Builder strategy somewhere on the board. Thus if Builder has a winning strategy for a graph $F$ and some given density restriction $d$, then in the probabilistic process with $p(n)=\omega(n^{-1/d})$, any online algorithm will \aas be forced to create a monochromatic copy of~$F$ somewhere on the board. Consequently, the  threshold of the probabilistic problem satisfies $p_0(F,r,n)\leq n^{-1/d}$. This argument is completely generic in the sense that it does not require any assumptions on the structure of Builder's winning strategy.  The underlying connection between the probabilistic and a deterministic variant of the same problem was first observed in~\cite{org-new} (for the edge-coloring version of the problem studied here), and subsequently applied in~\cite{balogh-butterfield}. The main contribution of the present work is that for the vertex-coloring problem studied here, the best upper bound on the online threshold resulting from this approach is tight (recall~\eqref{eq:m1*} and \eqref{eq:threshold}).

By the argument we just sketched, every winning strategy for Builder in the deterministic game translates to an upper bound on the threshold of the probabilistic problem. It seems to be much harder to prove an equally general statement translating \emph{Painter's} winning strategies in the deterministic game to \emph{lower} bounds on the threshold of the probabilistic problem. The reason for this is that the probabilistic process satisfies the density restriction imposed on Builder only \emph{locally}: Even though the random graph $\Gnp$ with $p(n)=\Theta(n^{-1/d})$ \aas contains no \emph{constant-sized} graphs $G$ with $m(G)>d$, the density of larger subgraphs is unbounded --- in particular, the expected density of the entire random graph is $\binom{n}{2}p/n=\Theta(n^{1-1/d})$, which is unbounded for $d>1$.  Consequently, winning strategies for Painter in the deterministic game do not automatically give rise to successful coloring strategies for the probabilistic problem. In order to nevertheless establish the desired lower bound result we will need a quite detailed understanding of the \emph{structure} of Painter's and Builder's optimal strategies in the deterministic game. 

Theorem~\ref{thm:main-2} establishes a general correspondence between the original probabilistic problem and the deterministic two-player game. We are not aware of any other results that establish a similar correspondence between probabilistic and deterministic variants of the same problem. In particular, such a correspondence does not hold for the offline version of the problem studied here: according to Theorem~\ref{thm:Ramsey}, the threshold for the existence of a valid $r$-coloring w.r.t.\ $F$ in the probabilistic setting is determined by the parameter $m_1(F)$, and not by the parameter $m_1^o(F,r)$ coming from the corresponding deterministic problem (in general we have $m_1(F)\neq m_1^o(F,r)$).

\subsection{Algorithms for efficiently coloring random graphs online} \label{sec:algorithms}

We now describe the structure of the coloring algorithms that arise from our approach, and their relation to Painter's optimal strategies in the deterministic game.

We use the concept of \emph{ordered graphs}. An ordered graph is a graph with an associated ordering of its vertices, where this ordering is interpreted as the order in which these vertices appeared in the probabilistic process or the deterministic game.
We will see that for any graph $F$ and any integer $r$, there exists an optimal Painter strategy (i.e., a strategy that is a winning strategy for any density restriction $d<m_1^*(F,r)$) that can be represented as a \emph{priority list over ordered monochromatic subgraphs of~$F$}. Such a priority list is computed along with $m_1^*(F,r)$ in our approach, and encodes the relative `level of danger' Painter associates with copies of a given ordered subgraph of $F$ in a given color. In the asymptotic setting of the probabilistic process, determining such a priority list is a constant-size computation.

Given this priority list, Painter's strategy is the following: Whenever Builder presents a new vertex, Painter determines for each color the most dangerous ordered graph that would be completed if the new vertex were assigned this color, and then selects the color for which this most dangerous graph is least dangerous among all colors. (Observe that this requires Painter to memorize the order in which the vertices on the board arrived.) Note that this strategy based on a priority list can be easily implemented in polynomial time. 

As we shall see, for any $F$ and $r$ we can compute a priority list such that the strategy represented by it is not only (i) a winning strategy for Painter in the deterministic game with density restriction $d$ for any $d<m_1^*(F,r)$, but also (ii) a (polynomial-time) algorithm that succeeds \aas in finding a valid coloring of $\Gnp$ online for any $p(n)=o(n^{-1/m_1^*(F,r)})$. (Recall from Section~\ref{sec:remarks-main-2} that (i) does not automatically imply (ii)!)

\subsection{Is there an explicit formula for $m_1^*(F,r)$?} \label{sec:forests}

From Theorem~\ref{thm:ovg-combinatorica} and Theorem~\ref{thm:main-2} it follows that for any graph $F$ that has an induced subgraph~$F^\circ\sneq F$ on~$v(F)-1$ vertices satisfying \eqref{eq:two-round-condition},
for any $r\geq 2$ the online vertex-Ramsey density $m_1^*(F,r)$ is given by $\ol{m}_1(F,r)$ as defined in~\eqref{eq:mFr}. (Of course, this can also be proved directly by  considering only the deterministic game.)

The question arises whether also for general graphs $F$ the abstract definition of $m_1^*(F,r)$ in~\eqref{eq:m1*} can be replaced by an explicit formula, perhaps by suitably generalizing the definition \eqref{eq:mFr}. To address this question we point out some of the difficulties involved in the innocent-looking case where $F$ is a long path. The results concerning this special case will be published separately in the companion paper~\cite{path-games}.

We first present a simplified formulation of our results for the case where $F$ is an arbitrary forest. Suppose $d$ is of the form $d=(k-1)/k$ for some integer $k\geq 2$. Then the restriction that Builder is not allowed to create a subgraph of density more than $d$ is equivalent to requiring that Builder creates no cycles and no components (=trees) with more than $k$ vertices. We call this game the \emph{deterministic $F$-avoidance game with $r$ colors and tree size restriction~$k$}.

\begin{corollary}[Forests]\label{cor:forests}
For any fixed forest $F$ with at least one edge and any fixed integer $r\geq 2$, the threshold for finding an $r$-coloring of $\Gnp$ that is valid \wrt $F$ online is 
\begin{equation*}
  p_0(F,r,n)=n^{-1-1/(k^*(F,r)-1)} \enspace,
\end{equation*}
where $k^*(F,r)$ is the smallest integer $k$ such that Builder has a winning strategy in the deterministic $F$-avoidance game with $r$ colors and tree size restriction $k$.
\end{corollary}
(Corollary~\ref{cor:forests} can also be proved directly by a much simpler proof than the general arguments in this work, reusing ideas of~\cite{org-new}.)

It is not hard to see that $k^*(F,r)$ is indeed well-defined for any forest $F$ and any integer $r\geq 2$. 
It follows from the results in~\cite{ovg-combinatorica} that for any tree $F$ and any integer $r\geq 2$ the greedy strategy (with $H_1=\cdots=H_r=F$) is a winning strategy for Painter in the deterministic game with tree size restriction $k=v(F)^r-1$, i.e., guarantees a lower bound of $k^*(F,r)\geq v(F)^r$. 

For the rest of this section we focus on the case where $F=P_\ell$ is the path on $\ell$ vertices, and $r=2$ colors are available. For this case our general procedure for computing $m_1^*(F,r)$ (or, equivalently if $F$ is a forest, for computing $k^*(F,r)$) can be simplified considerably. We were able to compute $k^*(P_\ell,2)$ for all $\ell\leq 45$. The resulting values are stated in Table~\ref{tab:kstar-Pell}, where the bottom row shows the difference $k^*(P_\ell,2)-\ell^2$, i.e., by how much optimal Painter strategies can improve on the greedy lower bound $v(P_\ell)^2=\ell^2$. The values in Table~\ref{tab:kstar-Pell} and the corresponding optimal Painter strategies seem to follow no discernible pattern. In view of this, it does not seem very likely that there exists an explicit formula for $k^*(P_\ell,2)$, let alone for the parameter $m_1^*(F,r)$ in general. 

The values in Table~\ref{tab:kstar-Pell} also raise the question by how much optimal strategies can improve on the greedy lower bound asymptotically as $\ell\to\infty$. We can show that $k^*(P_\ell,2)=\Omega(\ell^{2.01})$, i.e., there exist Painter strategies that improve on the greedy lower bound by a factor polynomial in $\ell$. On the other hand, we can prove an upper bound of $k^*(P_\ell,2)=\cO(\ell^{2.59})$, which shows that no superpolynomial improvement is possible~\cite{path-games}.

\begin{table}
\scriptsize
\setlength{\tabcolsep}{0.815mm}
\begin{center}
\begin{tabular}{lrrrrrrrrrrrrrrrrrrr}\toprule
$\ell$ & $2,\ldots,27$ & 28 & 29 & 30 & 31 & 32 & 33 & 34 & 35 & 36 & 37 & 38 & 39 & 40 & 41 & 42 & 43 & 44 & 45 \\ \midrule
$k^*(P_\ell,2)$ & $2^2,\ldots,27^2$ & 791 & 841 & 902 & 961 & 1040 & 1089 & 1156 & 1225 & 1323 & 1376 & 1449 & 1521 & 1641 & 1699 & 1796 & 1856 & 1991 & 2057 \\
$k^*(P_\ell,2)-\ell^2$ & 0 & 7 & 0 & 2 & 0 & 16 & 0 & 0 & 0 & 27 & 7 & 5 & 0 & 41 & 18 & 32 & 7 & 55 & 32 \\ \bottomrule
\end{tabular}
\end{center}
\caption{Exact values of $k^*(P_\ell,2)$ for $\ell\leq 45$.} \label{tab:kstar-Pell}
\end{table}

\subsection{Further related work}  The question of finding valid \emph{edge}-colorings of random graphs online was first considered by Friedgut \emph{et al.}, who proved a threshold result for the case where $F$ is a triangle and $r=2$ colors are available~\cite{MR2037068}. In~\cite{org-lb, org-ub}, the greedy strategy was analyzed for the edge-coloring setting, and results similar to Theorem~\ref{thm:ovg-combinatorica} were derived for the case of $r=2$ colors. In~\cite{org-new}, we presented the upper bound approach via deterministic two-player games discussed above; this approach was applied by Balogh and Butterfield to derive new upper bounds for the case where $F$ is a triangle and $r=3$ colors are available~\cite{balogh-butterfield}. It would be very interesting to determine whether a general result analogous to Theorem~\ref{thm:main-2} holds for the edge-coloring setting.

Various edge-coloring Builder-Painter games were studied in the context of deterministic Ramsey theory. The smallest number of moves Builder needs to win in the deterministic edge-coloring game without any restrictions is called the \emph{online (size) Ramsey number of $F$} and was studied by many researchers~\cite{MR701171, MR1249704, MR2594965, MR2445473, MR2152058, MR2381412, MR2411444}. Variants of the game where Builder is subject to various restrictions were studied in~\cite{west-et-al, MR2097326, MR2506387}.

\subsection{Organization of this paper}

Before actually proving Theorem~\ref{thm:main-1} and Theorem~\ref{thm:main-2}, we informally present the main ideas behind our proofs in Section~\ref{sec:ideas}.
In Section~\ref{sec:computing-ovrd} we describe our procedure to compute the online vertex-Ramsey density $m_1^*(F,r)$ for any $F$ and $r$. In this section we formulate two central propositions (Proposition~\ref{prop:Lambda-Builder} and Proposition~\ref{prop:Lambda-Painter} below) which together show that an optimal Builder strategy and an optimal Painter strategy for the deterministic game can be derived from this procedure. The proof of Theorem~\ref{thm:main-1} is based on these two propositions and is also presented in Section~\ref{sec:computing-ovrd}.

In Section~\ref{sec:upper-bound} we prove Proposition~\ref{prop:Lambda-Builder} by deriving an explicit Builder strategy from the procedure presented in Section~\ref{sec:computing-ovrd}, and in Section~\ref{sec:lower-bound} we prove Proposition~\ref{prop:Lambda-Painter} by deriving an explicit Painter strategy from the same procedure. These two sections can be read independently from each other.

In Section~\ref{sec:proof-thm-main-2} we finally turn to the original probabilistic problem and present the proof of Theorem~\ref{thm:main-2}. While the upper bound proof is completely self-contained and can be read independently from all other proofs, the lower bound proof relies very much on our analysis of the deterministic game in the preceding sections.

\section{Proof ideas} \label{sec:ideas}

In this section we aim to give an informal description of the main ideas behind our proofs. We will first focus on our procedure for computing the online vertex-Ramsey density, and then briefly comment on the proofs of Theorem~\ref{thm:main-1} and Theorem~\ref{thm:main-2}.

\subsection{Computing the online vertex-Ramsey density} \label{sec:intuition-computing}

Throughout this section, we focus on the deterministic game and sketch the underlying ideas in our procedure for computing the online vertex-Ramsey density $m_1^*(F,r)$ for given~$F$ and $r$. Note that the following is \emph{not} a proof sketch of Theorem~\ref{thm:main-1} --- rather, our goal in this section is to develop some intuition for how one arrives at the key definitions which stand at the very beginning of our formal arguments. 

\subsubsection{Basic observations} \label{sec:basic-observations}

Consider a family $\{G_1, \dots, G_f\}$ of disjoint copies of the same graph~$G$ on the board, and suppose that Builder adds vertices $v_1, \ldots, v_f$ to the board connecting $v_i$ to $G_i$ in exactly the same way for all $1\leq i\leq f$. Then, by the pigeonhole principle, for a $(1/r)$-fraction of the new vertices, Painter's coloring decision will be the same and result in copies of the same $r$-colored graph $G^+$.
By performing this pigeonholing in each step of his strategy, Builder can thus force Painter to always create \emph{many} copies of one of the $r$ graphs $G^+$ that Painter may choose from. Consequently, in the following we may assume \wolog that whenever Builder manages to enforce an $r$-colored graph $G^+$ on the board, he has as many such copies available as he needs in further steps.

As it turns out, the only type of move that is useful for Builder is of the following form: Assume that for each of the colors $s\in[r]$ the board contains a monochromatic copy of some subgraph $H_s$ of $F$ in color~$s$. Then Builder can force Painter to extend one of these copies to a monochromatic copy of a subgraph $H_\sigma^+$ of $F$ with $v(H_\sigma^+)=v(H_\sigma)+1$ for a color $\sigma\in[r]$ by presenting a new vertex $v$ and connecting it appropriately to the already existing monochromatic copies of $H_1, \dots, H_r$ (see Figure~\ref{fig:recursion}). Furthermore, \wolog Builder will always perform such a step using monochromatic copies of the graphs $H_1,\dots,H_r$ that have evolved independently from each other so far, and that are therefore contained in distinct components of the board (playing like this throughout will not increase the density restriction~$d$ for which Builder's strategy is legal). Proceeding in this fashion, Builder step by step enforces larger monochromatic subgraphs of $F$ from smaller ones, and eventually a monochromatic copy of $F$ (if the density restriction allows it).

\begin{figure}
\centering
\PSforPDF{
 \psfrag{v}{$v$}
 \psfrag{H1}{$H_1$}
 \psfrag{Hs}{$H_\sigma$}
 \psfrag{Hsp}{$H_\sigma^+$}
 \psfrag{Hr}{$H_r$}
 \psfrag{G}{\Large $G$}
 \psfrag{G1}{$G_1$}
 \psfrag{Gs}{$G_\sigma$}
 \psfrag{Gr}{$G_r$}
 \psfrag{ldots}{$\ldots$}
 \includegraphics{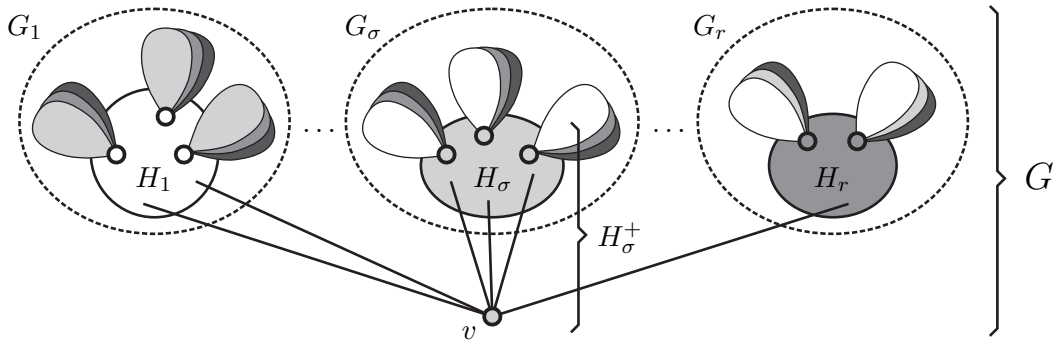}
}
\caption{Builder enforces larger monochromatic subgraphs of $F$ from smaller ones.} \label{fig:recursion}
\end{figure}

Each monochromatic copy of some subgraph $H$ of~$F$ created in this way is contained in a larger `history graph' $G$ that encodes all of Builder's construction steps that lead to the monochromatic copy of $H$. Using the notation from the preceding paragraph, the history graph $G$ of $H_\sigma^+$ arises as the union of the history graphs $G_1,\ldots,G_r$ of the copies of $H_1,\ldots,H_r$ (due to our assumption on how Builder plays, these $r$ history graphs are disjoint from each other), the vertex $v$ and the edges that connect $v$ to the copies of $H_1,\ldots,H_r$ in $G_1,\ldots,G_r$.

\subsubsection{Exploring Builder's options}

The key ingredient in our approach is a systematic exploration from Builder's point of view which monochromatic subgraphs of $F$ he can enforce against a \emph{fixed} Painter strategy. Our final procedure for computing $m_1^*(F,r)$ will have to branch on different coloring decisions of Painter, each branching corresponding to a different Painter strategy, but these branchings do not interfere with the ideas we want to present here. In the following we therefore assume that Painter plays according to a fixed strategy, and explain on an intuitive level how Builder can determine the smallest density restriction $d$ for which he can enforce a monochromatic copy of $F$ against the given Painter strategy.

As a first approach to such a systematic exploration, Builder could maintain for each color $s\in[r]$ a list $\cH_s$ of all subgraphs $H_s$ of $F$ for which he has already enforced a monochromatic copy in color~$s$ against the given Painter strategy, and also record the specific way in which the graph $H_s$ can be enforced by storing the corresponding history graph $G_s$. Builder can then use entries $(H_s,G_s)\in \cH_s$, one for every color $s\in[r]$, to create new entries $(H^+_\sigma,G)\in\cH_\sigma$ in the manner described above (see Figure~\ref{fig:recursion}), and compute for each such step the smallest density restriction $d$ for which this step is legal. (Recall that by appropriate pigeonholing in each step, Builder can create as many copies of each entry as he needs on the board.) There is no obvious termination criterion for this procedure, i.e., without further arguments Builder can never be sure that he found the \emph{smallest} possible density restriction~$d$ for which he can enforce a monochromatic copy of~$F$ against Painter's fixed strategy (it could be that by building larger and larger graphs he discovers new ways to enforce $F$ that are compliant with smaller and smaller density restrictions). In the following we will sketch how this approach can be refined to eventually yield a procedure which is guaranteed to find the smallest such $d$ in a finite number of steps.

\subsubsection{A generalized density restriction} \label{sec:generalized-game}

Note that each new history graph $G$ arising in a given step of Builder has a \emph{recursive} structure. Unfortunately, for computing the smallest admissible density restriction for which this step is legal the recursive structure of~$G$ does not help. However, by suitably generalizing our concept of density restriction the recursive structure of~$G$ can indeed be exploited.

For a fixed real number $\theta>0$ and any graph $H$ we define
\begin{equation} \label{eq:def-mu}
  \mu_\theta(G):=v(H)-e(H)\cdot\theta \enspace,
\end{equation}
and consider the following generalization of the deterministic $F$-avoidance game with $r$ colors and density restriction $d$: For fixed real parameters $\theta>0$ and $\beta$ we require that Builder adheres to the restriction that every subgraph $H$ of the evolving board $B$ with $v(H)\geq 1$ satisfies
\begin{equation} \label{eq:theta-rho-game}
  \mu_\theta(H) \geq \beta \enspace.
\end{equation}
We refer to this game as the \emph{deterministic $F$-avoidance game with $r$ colors and generalized density restriction $(\theta,\beta)$}. For any graph $F$ with at least one edge, any integer $r\geq 2$ and any real number $\theta>0$ we define the parameter
\begin{equation} \label{eq:beta*}
  \beta^*(F,r,\theta):=\sup\left\{\beta\in\RR \bigmid \parbox{0.46\displaywidth}{Builder has a winning strategy in the deterministic $F$-avoidance game with $r$ colors and generalized density restriction $(\theta,\beta)$}\right\} \enspace.
\end{equation}

Before discussing how this generalized game allows us to exploit the recursive structure of Builder's construction steps, let us explain how it relates to the original game with density restriction $d$ that we are actually interested in.

Note that for any $\theta>0$, the game with generalized density restriction $(\theta, 0)$ is equivalent to the game with density restriction $d=1/\theta$. Together with the definition in \eqref{eq:beta*} it follows that if for a given $\theta>0$ we have $\beta^*(F,r,\theta)<0$, then Painter has a winning strategy in the game with density restriction $d=1/\theta$, and if $\beta^*(F,r,\theta)>0$, then Builder has a winning strategy in the game with density restriction $d=1/\theta$. So intuitively speaking, computing the online vertex-Ramsey density $m_1^*(F,r)$ is equivalent to determining the root of $\beta^*(F,r,\theta)$, although it is not clear yet whether such a root exists and whether it is unique.

As it turns out, $\beta^*(F,r,\theta)$ does indeed have a unique root $\theta^*=\theta^*(F,r)$, and the online vertex-Ramsey density $m_1^*(F,r)$ satisfies $m_1^*(F,r)=1/\theta^*$. Furthermore, we can show that the root $\theta^*$ lies in an explicitly given finite set $Q=Q(F,r)$ of rational numbers. Therefore, it is straightforward to compute $m_1^*(F,r)$ provided we can compute $\beta^*(F,r,\theta)$ for given rational values of $\theta$. We describe a procedure that does essentially that, with one major caveat: Observe that if $\beta\geq 0$, then the condition \eqref{eq:theta-rho-game} holds for all subgraphs of the board if and only if it holds for all \emph{connected} subgraphs. Our approach makes crucial use of this observation, and consequently our procedure computes $\beta^*(F,r,\theta)$ exactly for any input parameters $F$, $r$, $\theta$ for which $\beta^*(F,r,\theta)\geq 0$, but returns meaningless negative values on input parameters for which $\beta^*(F,r,\theta)<0$. This makes no difference for our purposes since in order to find the root of $\beta^*(F,r,\theta)$ it suffices to check whether $\beta^*(F,r,\theta)$ equals zero for given values of $\theta\in Q$.

In the following we explain how the generalized density restriction allows us to exploit the recursive structure of the history graphs arising in the game. As before our viewpoint is that we are exploring Builder's options against a fixed strategy of Painter. More precisely, we consider a fixed value of $\theta>0$, and our goal now is to determine the largest value of $\beta$ for which Builder can enforce a monochromatic copy of $F$ in the game with generalized density restriction $(\theta, \beta)$ against the given Painter strategy. Combining this with the already mentioned branching on different strategies of Painter allows us to compute $\beta^*(F,r,\theta)$ as defined in~\eqref{eq:beta*}.

\subsubsection{From history graphs to vertex weights}  \label{sec:intuition-condensing}

We return to considering Builder's construction step in which monochromatic copies of subgraphs $H_1, \dots, H_r$ of $F$ with the corresponding history graphs $G_1, \dots, G_r$ are connected to a new vertex $v$, and Painter's decision to assign color $\sigma$ to $v$ creates a copy of $H^+_\sigma$ in color~$\sigma$ with history graph $G$ (see Figure~\ref{fig:recursion}). 
In order to find the largest $\beta\geq 0$ for which this step is legal in the game with generalized density restriction $(\theta,\beta)$, we need to find the minimal value $\mu_\theta(J)$ among all connected subgraphs~$J$ of $G$ that contain $v$ (recall that we may assume that~$J$ is connected due to the assumption that~$\beta\geq 0$). As $\mu_\theta(J)=v(J)-e(J)\cdot\theta$ as defined in~\eqref{eq:def-mu} is linear in $e(J)$ and $v(J)$, a connected subgraph $J$ of $G$ containing $v$ that minimizes $\mu_\theta(J)$ can be found recursively as follows: determine \emph{independently} for each~$s\in[r]$ the connected subgraph $J_s$ of $G_s+v$ containing $v$ that minimizes $\mu_\theta(J_s)$, where $G_s+v$ denotes the subgraph of $G$ induced by $v$ and all vertices of the copy of $G_s$ in $G$. The graph $J$ we are interested in is then given by the union of the graphs $J_s$ for all $s\in[r]$. (Note that the subgraph $J'$ of $G$ that contains~$v$ and maximizes $e(J')/v(J')$ can \emph{not} be found by independently considering each of the graphs $G_s+v$, $s\in[r]$ --- this is precisely why we introduced the generalized notion of density restriction.) This independence allows us to compute the subgraph $J$ of $G$ that minimizes $\mu_\theta(J)$ \emph{recursively} without remembering the actual structure of the history graphs $G_s$, $s\in[r]$. All the information that is necessary to do the same minimization in future steps (when the copy of $H_\sigma^+$ is extended to form larger subgraphs of $F$ in color $\sigma$) can be stored by assigning the value $\sum_{s\in[r]\setminus\{\sigma\}}(\mu_\theta(J_s)-1)$ to the vertex $v$ in $H_\sigma^+$ (the $-1$ in the sum accounts for the fact that all the graphs $G_s+v$, $s\in[r]$, share the vertex $v$). In other words, we can condense the `history' behind each of the vertices of a monochromatic copy of some subgraph of~$F$ into a single number. (Recall that we consider $\theta>0$ to be fixed --- this is crucial in all of the above.)
 
As a consequence, when maintaining the lists $\cH_s$, $s\in[r]$, Builder no longer needs to store the entire history graph associated with some monochromatic subgraph $H$ of $F$ on one of these lists, but can store all the necessary information as a simple vertex-weighting of $H$. This greatly reduces the amount of information Builder needs to keep track of, but does not yet solve the issue that there is no obvious termination criterion for Builder's exploration (Builder might still keep constructing new non-trivial entries forever).

\subsubsection{Unique vertex weights via vertex orderings}

In general, it may and will happen that the same subgraph $H$ of $F$ appears several times on one of Builder's lists with different vertex-weightings in such a way that none of these entries is redundant --- depending on how $H$ is used in future steps, different vertex-weightings of the same graph might be desirable from Builder's view. In other words, there is no unique best way of enforcing a copy of $H$ in a given color for Builder.

It turns out, however, that different useful vertex-weightings can only arise if Builder presents the vertices of $H$ in different orders (there are $v(H)!$ many different orders). For a fixed such order, there is a well-defined best vertex-weighting that Builder can achieve when enforcing $H$ in that particular order. Thus to explore his options completely Builder only needs to compute \emph{finite} lists $\cH_s$, one for every color $s\in[r]$, which contain one entry for each vertex-ordering of every subgraph $H$ of $F$.

This does not quite solve the issues we mentioned yet --- it could still occur that Builder needs to recompute the vertex-weighting for a given entry many times because he finds better and better ways to enforce a given graph $H$ in a particular order. To prevent this from happening, we need to be quite careful about the order in which we compute the entries of the lists $\cH_s$ --- essentially we start by considering the game with generalized density restriction $(\theta,\beta)$ for the given fixed $\theta>0$ and a very large $\beta$, and then successively lower $\beta$ by the minimal amount that makes new options available to Builder. In each step we compute the weights for all graphs that Builder can create respecting the current generalized density restriction $(\theta,\beta)$. This guarantees that we need to compute the weights for each graph only once, and therefore finally allows Builder to explore his options completely by a \emph{finite} procedure.

\subsubsection{Tying it all together}

Along the lines sketched in the previous sections, we can compute $\beta^*(F,r,\theta)$ by dynamic programming over vertex-ordered subgraphs of $F$ (provided that $\beta^*(F,r,\theta)$ is non-negative for the given $\theta>0$, see the remarks in Section~\ref{sec:generalized-game}), branching on Painter's decisions as appropriate. The online vertex-Ramsey density $m_1^*(F,r)$ can then be derived from $\beta^*(F,r,\theta)$ as explained in Section~\ref{sec:generalized-game}. As this is now a finite procedure, it also follows that the supremum in~\eqref{eq:beta*} is attained as a maximum, which with some further arguments also implies that the infimum in~\eqref{eq:m1*} is attained as a minimum.

\subsection{About the proof of Theorem~\texorpdfstring{\ref{thm:main-1}}{3}} \label{sec:proof-thm1}

For any graph $F$ and any integer $r$, let $\tm(F,r)$ denote the value computed by the procedure outlined in Section~\ref{sec:intuition-computing}. We prove that $\tm(F,r)$ equals $m_1^*(F,r)$ as defined in~\eqref{eq:m1*} by constructing explicit winning strategies for Builder and Painter, for arbitrary density restrictions $d\geq \tm(F,r)$ and $d< \tm(F,r)$, respectively.

For Builder such a strategy follows from the general principles underlying the procedure sketched in Section~\ref{sec:intuition-computing}: all steps of the dynamic program which is at the heart of our approach can be interpreted as actual construction steps on the board of the deterministic game. 
 
For Painter, such a strategy can be recovered from the branching on Painter's decisions performed in our procedure --- we show that the decisions corresponding to a `worst' path in the branching tree (viewed from Builder's perspective) give rise to a Painter strategy that succeeds in avoiding a monochromatic copy of $F$ against \emph{any} Builder strategy. This Painter strategy can be encoded by a priority list as described in Section~\ref{sec:algorithms}.

To prove the success of this strategy, we use a \emph{witness graph argument}: Essentially, we show inductively that whenever a monochromatic copy of some ordered subgraph $H$ of $F$ in some color $s\in[r]$ appears on the board, then this copy is contained in a graph that is at least as dense as indicated by the weights computed for $H$ and the color $s$ by the dynamic program in our procedure. (Recall from Section~\ref{sec:intuition-computing} that these weights basically encode the density of the history graph corresponding to the best way for Builder to enforce a monochromatic copy of $H$ in color $s$.) This invariant holds in particular for all vertex-orderings of the graph~$F$ and all colors $s\in[r]$, and implies that whenever a monochromatic copy of $F$ is completed, the board contains a graph that violates the density restriction imposed on Builder.

The proof of Theorem~\ref{thm:main-1} we just sketched also shows that there exists an integer $a_{\max}=a_{\max}(F,r)$ such that for any given density restriction $d$ Builder never needs more than $a_{\max}$ steps to enforce a monochromatic copy of $F$, if he is able to do so at all. Note that this statement alone directly implies all three assertions of Theorem~\ref{thm:main-1}, as it shows that $m_1^*(F,r)$ can also be computed trivially by exhaustive search over the finitely many possible ways Builder and Painter can play in $a_{\max}$ steps of the game.

\subsection{About the proof of Theorem~\texorpdfstring{\ref{thm:main-2}}{4}} \label{sec:proof-thm2}

We have already discussed the proof of the upper bound part of Theorem~\ref{thm:main-2} in Section~\ref{sec:remarks-main-2}; as mentioned this proof is self-contained and does not depend on the rest of this work. The proof of the lower bound part is much more involved and relies on the same witness graph approach as the argument for Painter's success in the deterministic game described in the previous section. However, there is the additional issue that, as explained in Section~\ref{sec:remarks-main-2}, the random graph $\Gnp$ with $p(n)=o(n^{-1/m_1^*(F,r)})$ satisfies a density restriction of $d=m_1^*(F,r)$ only locally and \emph{not} globally. Consequently, in order to apply the witness graph argument outlined above to the probabilistic setting of Theorem~\ref{thm:main-2}, we also need to show that the size of the witness graphs resulting from our arguments is bounded by some constant $\vmax=\vmax(F,r)$ (and not, say, linear in $n$). Unfortunately, we cannot show this for \emph{all} priority lists that represent optimal strategies for Painter in the deterministic game.
However, by applying a number of further technical refinements to the procedure described in Section~\ref{sec:intuition-computing}, we can guarantee that it only computes priority lists for which a constant $\vmax$ as desired indeed exists. It follows with the same witness graph argument as before that these priority lists represent polynomial-time coloring algorithms that \aas succeed in finding a valid coloring of $\Gnp$ online for any $p(n)=o(n^{-1/m_1^*(F,r)})$.

\section{Computing the online vertex-Ramsey density} \label{sec:computing-ovrd}

\subsection{Proof of Theorem~\texorpdfstring{\ref{thm:main-1}}{3}}

Recall the definition of the deterministic $F$-avoidance game with $r$ colors and generalized density restriction $(\theta,\beta)$ from Section~\ref{sec:generalized-game}, and recall further that, at least intuitively, computing the online vertex-Ramsey density $m_1^*(F,r)$ is equivalent to determining the root of $\beta^*(F,r, \theta)$ as defined in ~\eqref{eq:beta*} (where existence and uniqueness of this root are not clear yet).

As already mentioned, we are going to derive a procedure that returns $\beta^*(F,r,\theta)$ for any $\theta>0$ for which $\beta^*(F,r,\theta)\geq 0$, and a meaningless negative value for any $\theta>0$ for which $\beta^*(F,r,\theta)< 0$. This procedure will be described in Section~\ref{sec:algorithm}, and its output will be denoted by $\Lambda_\theta(F,r)$. We will see that the function $\Lambda_\theta(F,r)$ is well-defined for any real number $\theta>0$, and for rational values of $\theta$ it can be computed using only integer arithmetic. Most of the remainder of this paper will be devoted to the proofs of the following two key statements.

\begin{proposition}[Builder strategy from $\Lambda_\theta(F,r)$] \label{prop:Lambda-Builder}
Let $F$ be a graph with at least one edge and $r\geq 2$ an integer. There is a constant $\amax=\amax(F,r)$ such that the following holds: For any real numbers $\theta>0$ and $\beta\geq 0$ with
\begin{equation} \label{eq:Lambda-geq-beta}
  \Lambda_\theta(F,r)\geq \beta \enspace,
\end{equation}
where $\Lambda_\theta()$ is defined in \eqref{eq:def-Lambda} below, Builder can enforce a monochromatic copy of\/ $F$ in the deterministic $F$-avoidance game with $r$ colors and generalized density restriction $(\theta,\beta)$ in at most $\amax$ steps, regardless of how Painter plays.
\end{proposition}

\begin{proposition}[Painter strategy from $\Lambda_\theta(F,r)$] \label{prop:Lambda-Painter}
Let $F$ be a graph with at least one edge, $r\geq 2$ an integer, and\/ $\theta>0$ and $\beta\geq 0$ real numbers such that
\begin{equation} \label{eq:Lambda-less-beta}
  \Lambda_\theta(F,r)<\beta \enspace,
\end{equation}
where $\Lambda_\theta()$ is defined in \eqref{eq:def-Lambda} below. \\
Then Painter can avoid creating a monochromatic copy of\/ $F$ in the deterministic $F$-avoidance game with $r$ colors and generalized density restriction $(\theta,\beta)$, regardless of how Builder plays.
\end{proposition}

Before going into any details about the procedure that defines $\Lambda_\theta(F,r)$, we show how Proposition~\ref{prop:Lambda-Builder} and Proposition~\ref{prop:Lambda-Painter} imply Theorem~\ref{thm:main-1}. 

For technical reasons, our formal arguments do not rely on the parameter $\beta^*()$ defined in~\eqref{eq:beta*}, but on a related parameter that we introduce now.
For any graph $F$ with at least one edge, any integer $r\geq 2$, any real number $\theta>0$ and any integer $a\geq \amin:=r(v(F)-1)+1$, we define
\begin{equation} \label{eq:beta'}
  \beta'(F,r,\theta,a):=\sup\left\{\beta\in\RR \bigmid \parbox{0.50\displaywidth}{Builder has a winning strategy in the deterministic $F$-avoidance game with $r$ colors and generalized density restriction $(\theta,\beta)$ in at most $a$ steps}\right\} \enspace.
\end{equation}
Here the supremum is over a nonempty set of values because presenting the complete graph on $\amin$ vertices sequentially is a winning strategy for Builder that satisfies the generalized density restriction $(\theta,\beta)$ for any $\beta\leq \min\{ k - \binom{k}{2}\cdot\theta \mid 1\leq k\leq \amin \}$.
Note that for all $F$, $r$, and $\theta$ as before we have
\begin{equation} \label{eq:convergence}
  \beta^*(F,r,\theta)=\sup_{a\geq \amin}\,\beta'(F,r,\theta,a)= \lim_{a\to\infty}\beta'(F,r,\theta,a) \enspace.
\end{equation}
As in the definition of $\beta'()$ in~\eqref{eq:beta'} there is only a finite number of possible Builder strategies to consider, it is not hard to derive the following properties of $\beta'()$.

\begin{lemma}[Properties of $\beta'(F,r,\theta,a)$] \label{lemma:beta*}
For any graph $F$ with at least one edge, any integer $r\geq 2$, any real number $\theta>0$ and any integer $a\geq \amin$, the supremum in~\eqref{eq:beta'} is attained as a maximum. For fixed $F$, $r$, and $a$ as before, $\beta'(F,r,\theta,a)$ viewed as a function of $\theta>0$ is continuous, non-increasing, piecewise linear, and has a unique root, which is contained in the set 
\begin{equation} \label{eq:Qa}
  \textstyle Q(a):=\{\;0<\frac{v}{e}<2 \;\mid \; v,e\in\NN \;\wedge\; 1\leq v\leq a\; \wedge \; 1\leq e\leq \binom{v}{2}\; \} \enspace.
\end{equation}
\end{lemma}

\begin{proof} We identify Builder's strategies in the deterministic two-player game with $r$ colors with finite $r$-ary rooted trees, where each node at depth $k$ of such a tree is an $r$-colored graph on $k$ vertices, representing the board after the $k$-th step of the game. Specifically, the tree $\cT$ representing a given Builder strategy is constructed as follows: The root of $\cT$ is the null graph (the graph whose vertex set is empty). The $r$ children of any node $B$ at depth $k$ of $\cT$ are obtained by adding the $(k+1)$-th vertex of Builder's strategy to $B$ (together with the edges that connect this vertex to previously added vertices according to Builder's strategy) and coloring it with one of the $r$ available colors. Continuing like this, we construct $\cT$, representing any situation in which Builder stops playing by a leaf of $\cT$.

Note that in this formalization, a given tree $\cT$ represents a generic strategy for Builder (in the deterministic game with $r$ colors) that may or may not satisfy a given generalized density restriction $(\theta, \beta)$, and that can be thought of as a strategy for the~`$F$-avoidance' game for any given graph $F$. We say that $\cT$ is a \emph{winning strategy} for Builder in a specific $F$-avoidance game if and only if every leaf of~$\cT$ contains a monochromatic copy of $F$. We say that a Builder strategy $\cT$ is a \emph{legal strategy} in the game with generalized density restriction $(\theta,\beta)$ if and only if~\eqref{eq:theta-rho-game} is satisfied for every subgraph $H$ with $v(H)\geq 1$ of every node $B$ in $\cT$.

Let $F$, $r$ and $a\geq \amin$ be given. As the number of steps of the game is bounded by $a$, there is only a finite family $\fT=\fT(r,a)$ of different Builder strategies, obtained by exhaustive enumeration of all possible ways to add a new vertex to the board. Let $\fW= \fW(F,r,a)\seq \fT$ denote the set of winning strategies for Builder for the given $F$, and recall that for $a\geq \amin$ the family $\fW$ is nonempty.

Note that for any winning strategy $\cT\in\fW$ and for any fixed $\theta>0$,
\begin{equation} \label{eq:fT}
  f_\cT(\theta):=\min_{\substack{B\in\cT \\ H\seq B: v(H)\geq 1}} \mu_\theta(H)
\end{equation}
is the maximal value of $\beta$ such that $\cT$ is a legal strategy in the game with generalized density restriction $(\theta,\beta)$. Optimizing over the (finite and nonempty) set of winning strategies, we obtain $\beta'(F,r,\theta, a)$ as defined in~\eqref{eq:beta'} as
\begin{equation} \label{eq:beta'-max-fT}
  \beta'(F,r,\theta,a)=\max_{\cT\in\fW}f_\cT(\theta) \enspace.
\end{equation}
We conclude that the supremum in \eqref{eq:beta'} is attained as a maximum.
In the following we derive the claimed properties of $\beta'(F,r,\theta,a)$ as a function of $\theta>0$ by considering the functions $f_\cT(\theta)$, $\cT\in\fW$.

Using \eqref{eq:fT} and combining the properties of the linear functions $\mu_\theta(H)$ for all $H\seq B$ with $v(H)\geq 1$ and all $B\in\cT$ it is not hard to see that for any $\cT\in\fW$ the function $f_\cT(\theta)$ satisfies the following properties:
\begin{itemize}
\item $f_\cT(\theta)$ is continuous and piecewise linear.
\item There is an $\eps=\eps(\cT)>0$ such that $f_\cT(\theta)=1$ for all $0<\theta\leq\eps$ and $f_\cT(\theta)$ is strictly decreasing for all $\theta\geq \eps$.
\item $f_\cT(\theta)$ has a unique root in the set $\{\frac{v}{e} \;\mid\; v,e\in\NN \;\wedge\; 1\leq v\leq a\; \wedge \; 1\leq e\leq \binom{v}{2}\}$.
\end{itemize}

Note that the root of $f_\cT(\theta)$ is strictly smaller than 2: For any winning strategy $\cT\in\fW$, there is a leaf $B$ in $\cT$ that contains a (not necessarily monochromatic) copy of $P_3$ as a subgraph. This is trivially true if $P_3\seq F$ (as every leaf of $\cT$ contains a monochromatic copy of $F$). If $P_3\sneq F$, then $F$ is a matching, and any strategy where Painter colors endpoints of isolated edges on the board with different colors corresponds to a root-leaf path in $\cT$ that does not end with a matching (as otherwise $\cT$ would not be a winning strategy for Builder). Thus in either case the graph $H=P_3$ is a subgraph of some node $B$ of $\cT$, and consequently the minimization in \eqref{eq:fT} includes the function $\mu_\theta(H)=3-2\cdot\theta$, whose root is strictly smaller than 2.

It follows with~\eqref{eq:beta'-max-fT} that also $\beta'(F,r,\theta,a)$ satisfies the three properties listed above, and that its root is strictly smaller than 2. Combining those properties shows that $\beta'(F,r,\theta,a)$ satisfies the conditions claimed in the lemma.
\end{proof}

Using Proposition~\ref{prop:Lambda-Builder}, Proposition~\ref{prop:Lambda-Painter} and Lemma~\ref{lemma:beta*}, we will prove the following explicit version of Theorem~\ref{thm:main-1}. 

\begin{theorem}[Explicit Version of Theorem~\ref{thm:main-1}] \label{thm:main-1'}
For any graph $F$ with at least one edge and any integer $r\geq 2$, the online vertex-Ramsey density $m_1^*(F,r)$ defined in \eqref{eq:m1*} satisfies
\begin{equation} \label{eq:m1*-via-theta*}
  m_1^*(F,r)=1/\theta^* \enspace,
\end{equation}
where $\theta^*=\theta^*(F,r)$ is the unique solution of
\begin{equation} \label{eq:root-of-Lambda}
  \Lambda_\theta(F,r) \stackrel{!}{=} 0
\end{equation}
and $\Lambda_\theta()$ is defined in \eqref{eq:def-Lambda} below. \\
Moreover, $\theta^*$ is a rational number from the set $Q(\amax)$, where $Q()$ is defined in \eqref{eq:Qa} and $\amax$ is the constant guaranteed by Proposition~\ref{prop:Lambda-Builder}. Furthermore, the infimum in \eqref{eq:m1*} is attained as a minimum.
\end{theorem}

Theorem~\ref{thm:main-1} is an immediate consequence of Theorem~\ref{thm:main-1'}, observing that the solution of the equation~\eqref{eq:root-of-Lambda} can be computed by evaluating $\Lambda_\theta(F,r)$ for all (finitely many) rational $\theta\in Q(\amax)$ (the constant $\amax$ is given explicitly in the proof of Proposition~\ref{prop:Lambda-Builder}, see~\eqref{eq:amax} below).

\begin{proof}[Proof of Theorem~\ref{thm:main-1'}] Throughout the proof we consider $F$ and $r$ fixed and let $\amax=\amax(F,r)$ denote the constant guaranteed by Proposition~\ref{prop:Lambda-Builder}.

Proposition~\ref{prop:Lambda-Builder} and Proposition~\ref{prop:Lambda-Painter} imply that for any given $\theta>0$ for which $\Lambda_\theta(F,r)\geq 0$, the parameter $\Lambda_\theta(F,r)$ is the maximal value of $\beta$ for which Builder can win the deterministic game with generalized density restriction $(\theta,\beta)$, and if he can win then he needs at most $\amax$ steps to enforce a monochromatic copy of $F$, i.e., $\Lambda_\theta(F,r)$ coincides with $\beta'(F,r,\theta,\amax)$ as defined in~\eqref{eq:beta'}.

Recall that according to Lemma~\ref{lemma:beta*} the supremum in \eqref{eq:beta'} is always attained as a maximum, i.e., for any $\theta>0$ Builder has a winning strategy in the game with generalized density restriction $(\theta,\beta'(F,r,\theta,\amax))$.
Thus if $\beta'(F,r,\theta,\amax)\geq 0$ we must have that $\Lambda_\theta(F,r)\geq \beta'(F,r,\theta,\amax)$ as otherwise we could apply Proposition~\ref{prop:Lambda-Painter} with $\beta=\beta'$ to obtain a contradiction. Hence also $\Lambda_\theta(F,r)$ is non-negative in that case.

It follows that for any $\theta>0$ the two functions $\Lambda_\theta(F,r)$ and $\beta'(F,r,\theta,\amax)$ either coincide or are both negative. Thus in particular they have the same set of roots, which by Lemma~\ref{lemma:beta*} consists of a single rational number $\theta^*=\theta^*(F,r)$ from the set $Q(\amax)$.

Applying Proposition~\ref{prop:Lambda-Builder} with $\theta=\theta^*$ and $\beta=0$ yields that Builder has a winning strategy in the game with generalized density restriction $(\theta^*,0)$ (in at most $\amax$ steps). Conversely, for any $\theta>\theta^*$ we obtain with Lemma~\ref{lemma:beta*} that $\beta'(F,r,\theta,\amax)$ is negative which, as discussed above, implies that also $\Lambda_\theta(F,r)$ is negative. Consequently we may apply Proposition~\ref{prop:Lambda-Painter} with $\beta=0$ to infer that Painter has a winning strategy in the game with generalized density restriction $(\theta,0)$.

Recalling that for any $\theta>0$ the game with generalized density restriction $(\theta,0)$ is equivalent to the original deterministic game with density restriction $d=1/\theta$, we may restate our findings as follows: Builder has a winning strategy in the game with density restriction $d=1/\theta^*$ (in at most $\amax$ steps), and for any $d<1/\theta^*$ Painter has a winning strategy in the game with density restriction $d$. I.e., the online vertex-Ramsey density defined in~\eqref{eq:m1*} satisfies $m^*_1(F,r)=1/\theta^*$, and the infimum in \eqref{eq:m1*} is attained as a minimum.
\end{proof}

\begin{remark}
Analogously to the second paragraph of the preceding proof it follows that for any given $\theta>0$ for which $\Lambda_\theta(F,r)\geq 0$, also $\beta^*(F,r, \theta)$ as defined in~\eqref{eq:beta*} coincides with $\Lambda_\theta(F,r)=\beta'(F,r,\theta,\amax)$. Thus the unique root $\theta^*$ of $\Lambda_\theta(F,r)=\beta'(F,r,\theta,\amax)$ is also a root of $\beta^*(F,r, \theta)$.

Furthermore, the observation that the non-increasing functions $\beta'(F,r,\theta,a)$, $a\geq \amin$, have a slope of at most $-1$ around their respective roots implies with~\eqref{eq:convergence} that the pointwise limit $\beta^*(F,r, \theta)$ has at most one root. Thus $\theta^*$ is indeed also the unique root of $\beta^*(F,r,\theta)$, as claimed in Section \ref{sec:generalized-game}.
\end{remark}

\subsection{Definitions and notations}

In order to present our procedure for computing the values $\Lambda_\theta(F,r)$ satisfying Proposition~\ref{prop:Lambda-Builder} and Proposition~\ref{prop:Lambda-Painter}, we need to introduce a number of definitions and notations. Along with the definitions we give some intuition how those formal objects implement the ideas outlined in Section~\ref{sec:intuition-computing}.

To simplify notation, for a graph $H$ and any vertex $v$ of $H$ we abbreviate $v\in V(H)$ to $v\in H$. For a graph $H$ and any set of vertices $U\seq V(H)$, we denote by $H\setminus U$ the graph obtained from $H$ by removing all vertices in $U$ and all edges incident to them. To indicate removal of a single vertex $v\in H$ we abbreviate $H\setminus\{v\}$ to $H\setminus v$.

\subsubsection{Weighted graphs}

A \emph{vertex-weighted graph} is a graph $H$ with a weight function $w:V(H)\rightarrow\RR$. We refer to the values $w(u)$, $u\in H$, as \emph{vertex weights}. Throughout this work, these vertex weights represent contributions to the linear function $\mu_\theta()$ defined in \eqref{eq:def-mu} that are obtained from `condensing' history graphs as outlined in Section~\ref{sec:intuition-condensing}. They will always be non-positive.

For a fixed real number $\theta>0$, any graph $H$, any vertex $v\in H$ and any weight function $w:V(H)\setminus\{v\}\rightarrow\RR$ we define the value
\begin{equation} \label{eq:def-d}
  d_\theta(H,v,w):=\min_{J\seq H: v\in J} \Big(\sum_{u\in J\setminus v} \big(1+w(u)\big)-e(J)\cdot\theta\Big) \enspace,
\end{equation}
where the minimization is over all subgraphs $J$ of $H$ that contain the vertex $v$. As this minimization includes the graph $J$ that consists only of the isolated vertex $v$, we always have $d_\theta(H,v,w)\leq 0$. Note that the minimum in \eqref{eq:def-d} is always attained by an \emph{induced} subgraph $J\seq H$. For convenience we will also use this notation for weight functions $w$ whose domain is strictly larger than the set $V(H)\setminus\{v\}$. Of course, for the value of $d_\theta(H,v,w)$ only the values $w(u)$ of vertices $u\in H\setminus v$ are relevant.

The intuition behind the value $d_\theta(H,v,w)$ is the following: Assume that a copy of $H\setminus v$ is used as one of the graphs $H_s$ in Figure~\ref{fig:recursion}, and Painter selects a color $\sigma\in[r]$ such that a copy of some other graph $H_\sigma$ is extended to a copy of $H_\sigma^+$. Then $H$ becomes part of the history graph $G$ of $H_\sigma^+$, and the recursive contribution to the value $\mu_\theta(J)$ (as defined in~\eqref{eq:def-mu}) of a subgraph $J\seq G$ minimizing $\mu_\theta(J)$ is exactly $d_\theta(H,v,w)$ if $v$ is included in $J$. In our dynamic program, this will be recorded by adding a term of $d_\theta(H,v,w)$ to the vertex weight of $v$ in $H_\sigma^+$ (and this is also how the vertex weights $w$ of $H\setminus v$ were computed in earlier steps).

For a fixed real number $\theta>0$, any graph $H$ and any weight function $w:V(H)\rightarrow\RR\cup\{-\infty\}$ we define
\begin{equation} \label{eq:def-lambda}
  \lambda_\theta(H,w):=\sum_{u\in H} \big(1+w(u)\big)-e(H)\cdot\theta \enspace.
\end{equation}
As it is the case for the definition of $d_\theta()$ in \eqref{eq:def-d}, it is also convenient here to allow weight functions $w$ whose domain is strictly larger than the set $V(H)$. Of course, for the value of $\lambda_\theta(H,w)$ only the values $w(u)$ of vertices $u\in H$ are relevant. Observe that $\lambda_\theta(H,w)$ defined in \eqref{eq:def-lambda} can be written recursively for every vertex $v\in H$ as
\begin{equation} \label{eq:lambda-recursive}
 \lambda_\theta(H,w)=\lambda_\theta(H\setminus v,w)+1+w(v)-\deg_H(v)\cdot\theta \enspace,
\end{equation}
where $\deg_H(v)$ denotes the degree of $v$ in $H$. This will be used several times in our arguments.

Using \eqref{eq:def-mu}, $\lambda_\theta(H,w)$ defined in \eqref{eq:def-lambda} can also be written as $\lambda_\theta(H,w)= \mu_\theta(H) + \sum_{u\in H} w(u)$, which intuitively means the following: If we imagine $H$ to be at the center of a large history graph $G$, the parameter $\lambda_\theta(H,w)$ corresponds to the value $\mu_\theta(J)$ of the graph $J$ obtained by attaching to each vertex $v\in H$ the $r-1$ subgraphs that minimize $\mu_\theta(J_s)$ among all subgraphs $J_s$ containing $v$ in each of the $r-1$ branches of the history graph $G$. 

\subsubsection{Ordered graphs}
For any graph $H$, $h:=v(H)$, a \emph{vertex ordering} is a bijective mapping $\pi:V(H)\rightarrow \{1,\ldots,h\}$, conveniently denoted by its preimages, $\pi=(\pi^{-1}(1),\ldots,\pi^{-1}(h))$. An \emph{ordered graph} is a pair $(H,\pi)$, where $H$ is a graph and $\pi$ is an ordering of its vertices. In the context of the $F$-avoidance game we interpret the ordering $\pi=(v_1,\ldots,v_h)$ as the order in which the vertices of $H$ appeared in the game, where $v_h$ is the vertex that appeared first (we refer to it as the \emph{oldest} vertex) and $v_1$ is the vertex that appeared last (we refer to it as the \emph{youngest} vertex). We use $\Pi(V(H))$ to denote the set of all vertex orderings $\pi$ of $H$.

For an ordered graph $(H,\pi)$ and any subgraph $J\seq H$, we denote by $\pi|_J$ the order on the vertices of $J$ induced by $\pi$. For any set $U\seq V(H)$ we use $\pi\setminus U$ as a shorthand notation for $\pi|_{H\setminus U}$. To indicate removal of a single vertex $v\in H$ we abbreviate $\pi\setminus\{v\}$ to $\pi\setminus v$.

Moreover, we define
\begin{equation} \label{eq:def-S}
  \cS(F):=\big\{[(H,\pi)]_\sim\bigmid[\big] H\seq F \text{ with } v(H)\geq 1 \text{ and } \pi\in\Pi(V(H))\big\}
\end{equation}
as the family of all isomorphism classes of ordered subgraphs of $F$, where we write $(H,\pi)\sim(H',\pi')$ if $(H,\pi)$ and $(H',\pi')$ are isomorphic as ordered graphs. For simplicity we refer to the elements $[(H,\pi)]_\sim\in\cS(F)$ in the following always as graphs $(H,\pi)\in\cS(F)$.
It is convenient to think of the graphs in $\cS(F)$ as nodes of a rooted tree $\cT(F)$ with root node $(K_1,(v_1))$ (an isolated vertex), where for each node $(H,\pi)\in\cS(F)$, $\pi=(v_1,\ldots,v_h)$, with $v(H)\geq 2$ the parent node is given by $(H\setminus v_1,\pi\setminus v_1)$.
For any subset $\cH\seq\cS(F)$ we define the set $\cC(\cH,F)\seq\cS(F)$ as
\begin{equation} \label{eq:def-C}
  \cC(\cH,F):=\begin{cases} \big\{(K_1,(v_1))\big\} & \text{if $\cH=\emptyset$} \\
                            \big\{(H,\pi=(v_1,\ldots,v_h))\in\cS(F)\setminus \cH \,\bigmid[\big]\, (H\setminus v_1,\pi\setminus v_1)\in\cH\big\} & \text{otherwise} \enspace.
              \end{cases}
\end{equation}
Note that $\cC(\cH,F)$ is exactly the set of nodes of $\cT(F)$ that are children of some node in $\cH$, but that are not contained in $\cH$. Figure~\ref{fig:tk3} shows the tree $\cT(F)$ for $F=K_3$ and illustrates the definition in \eqref{eq:def-C}.

\begin{figure}
\centering
\PSforPDF{
 \psfrag{h}{$\cH$}
 \psfrag{th}{$\cC(\cH,K_3)$}
 \psfrag{tk3}{\Large $\cT(K_3)$}
 \psfrag{1}{\small 1}
 \psfrag{2}{\small 2}
 \psfrag{3}{\small 3}
 \includegraphics{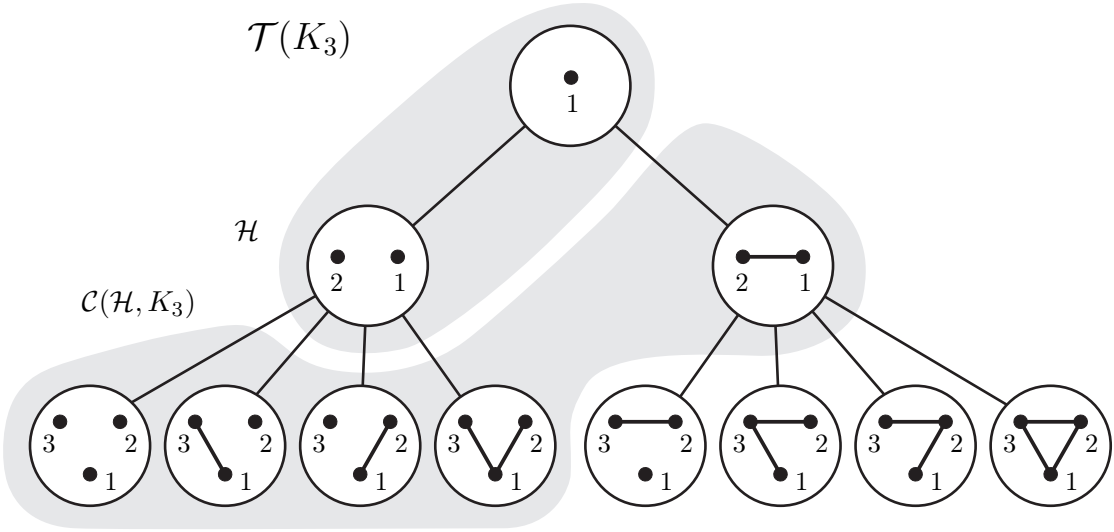}
}
\caption{Illustration of the tree $\cT(K_3)$ and the definition in \eqref{eq:def-C} (the shaded regions represent subsets of nodes of $\cT(K_3)$).} \label{fig:tk3}
\end{figure}

\begin{remark} \label{remark:induced-subgraphs}
Note that for two graphs $H_1\subsetneq H_2$ with $v(H_1)=v(H_2)$, a monochromatic copy of $H_1$ on the board can never evolve into a copy of $H_2$ later in the game, as new edges appear only incident to newly added vertices. As a consequence, we could restrict our attention to \emph{induced} subgraphs of~$F$ in all of our arguments. While changing the definition of $\cS(F)$ in \eqref{eq:def-S} accordingly would indeed lead to some algorithmic savings (see Section~\ref{sec:implementation}), for our formal arguments we find it more convenient to include \emph{all} subgraphs of $F$ in the definition~\eqref{eq:def-S}. Otherwise, unnecessary distraction would arise everytime an \emph{induced} subgraph is mentioned in a proof.
\end{remark}

\subsection{The algorithm} \label{sec:algorithm}

In the following we present an algorithm $\CW()$, whose output is then used to define the function $\Lambda_\theta(F,r)$ that is referred to in Proposition~\ref{prop:Lambda-Builder} and Proposition~\ref{prop:Lambda-Painter}.

Beside the graph $F$ and the number of colors $r$, the algorithm has two more input parameters: the parameter $\theta$ from the generalized density restriction (see Section~\ref{sec:generalized-game}), and a finite sequence $\alpha\in[r]^{r\cdot|\cS(F)|}$ with the following interpretation:
As indicated in Section~\ref{sec:intuition-computing}, the underlying idea of the algorithm is to explore systematically from Builder's point of view which monochromatic ordered subgraphs of $F$ he can enforce if Painter plays according to a fixed strategy.
Step by step Builder enforces larger monochromatic subgraphs from smaller ones, and the appropriate vertex weights for these graphs are computed by dynamic programming. The sequence $\alpha$ encodes Painter's coloring decisions in the order they occur in the course of the algorithm (i.e., it represents a fixed Painter strategy), where an entry of this sequence may correspond to several coloring decisions of Painter for which she uses the same color. (Our proofs show that she would not gain anything by using different colors for these decisions.)

The algorithm maintains for each color $s\in[r]$ a family $\cH_s\seq\cS(F)$ of ordered subgraphs of $F$ and a function $w_s:\cH_s\rightarrow\RR$. The families $\cH_s$ correspond to the ordered subgraphs of $F$ for which Builder has already enforced a monochromatic copy in color $s$. In the course of the algorithm, the families $\cH_s$ are successively enlarged.
Initially, we have $\cH_s=\emptyset$ for all $s\in[r]$, and at each step the candidate graphs to be added to the families $\cH_s$ are given by the sets $\cC(\cH_s,F)$ defined in~\eqref{eq:def-C}; these correspond to the graphs that Builder can construct by adding a single vertex to a graph he has already enforced. Consequently, throughout the algorithm the families $\cH_s$, $s\in[r]$, viewed as subsets of nodes of the tree $\cT(F)$ defined after~\eqref{eq:def-S}, grow downwards from the root.

For each $s\in[r]$ and each ordered graph $(H,\pi)\in\cH_s$, $\pi=(v_1,\ldots,v_h)$, the function $w_s:\cH_s\rightarrow\RR$ maintained by the algorithm induces a weight function $w_{(H,\pi,s)}: V(H) \rightarrow \RR$ as follows: The weight $w_{(H,\pi,s)}(v_1)$ of the youngest vertex $v_1$ is given directly by $w_s(H,\pi)$; the weight $w_{(H,\pi,s)}(v_2)$ of the second-youngest vertex $v_2$ is given by $w_s(H\setminus v_1,\pi\setminus v_1)$, i.e., by the value of $w_s$ for the parent of $(H,\pi)$ in $\cT(F)$; and so on. The full weight function $w_{(H,\pi,s)}: V(H) \rightarrow \RR$ is therefore obtained by considering the value of $w_s$ for all graphs on the path from $(H,\pi)$ to the root $(K_1,(v_1))$ of the tree $\cT(F)$, and each graph $(H,\pi)\in\cH_s$ inherits all vertex weights except that of the youngest vertex from his ancestors in $\cT(F)$.

More formally, and extending this construction to all graphs $(H,\pi)\in\cS(F)$, we define for each $s\in[r]$ and each $(H,\pi)\in\cS(F)$, $\pi=(v_1,\ldots,v_h)$, the weight function
\begin{equation} \label{eq:def-w}
  w_{(H,\pi,s)}(v_i) :=
  \left\{
    \begin{alignedat}{3}
      &w_s(H\setminus\{v_1,\ldots&&,v_{i-1}\},\pi\setminus\{v_1,\ldots,v_{i-1}\}) \\
      &&& \text{if } (H\setminus\{v_1,\ldots,v_{i-1}\},\pi\setminus\{v_1,\ldots,v_{i-1}\})\in\cH_s \enspace, \\
      &-\infty
      && \text{otherwise} \mathstrut \enspace.
    \end{alignedat}
  \right.
\end{equation}
This notation will also be used in the formulation of~$\CW()$ in Algorithm~\ref{algo:cw} below. (We shall see that the algorithm never encounters the value $-\infty$ during its execution.) Note that an ordered graph $(H,\pi)\in\cS(F)$ has vertices of weight $-\infty$ if and only if $(H,\pi)\in \cS(F)\setminus \cH_s$ for the corresponding $s\in[r]$, which intuitively means that Builder has not yet enforced a monochromatic copy of $(H,\pi)$ in color $s$. 

The families $\cH_s\seq \cS(F)$ and the functions $w_s:\cH_s\rightarrow \RR$, $s\in[r]$, are extended step by step in the course of the algorithm, and their final values are returned when the algorithm terminates.

\begin{algorithm} \label{algo:cw}
\DontPrintSemicolon
\SetEndCharOfAlgoLine{}
\LinesNumbered
\SetNlSty{}{}{}
\SetArgSty{}
\caption{$\CW(F,r,\theta,\alpha)$}
\vspace{.2em}
\SetKw{Continue}{continue}
\KwIn{a graph $F$ with at least one edge, an integer $r\geq 2$, a real number $\theta>0$, a sequence $\alpha\in[r]^{r\cdot |\cS(F)|}$}
\KwOut{an $r$-tuple $\big((\cH_s,w_s)\big)_{s\in[r]}$, where $\cH_s\seq\cS(F)$ and $w_s:\cH_s\rightarrow\RR$ for all $s\in[r]$}
\vspace{.2em}
\ForEach{$s\in[r]$}{
  $\cH_s:=\emptyset$ \label{cw:H-init} \;
  $\forall d\in\RR: \cC_s(d):=\emptyset$ \label{cw:Cd-init} \;
}
$i:=0$ \;
\Repeat({(*)}){$\cH_s=\cS(F)$ for some $s\in[r]$ \label{cw:terminate} }{
  $i:=i+1$ \;
  \ForEach{$s\in[r]$ \label{cw:least-dangerous-threat-loop}}{
    $\displaystyle d_s^i:=\max_{(H,\pi=(v_1,\ldots,v_h))\in \cC(\cH_s,F)} d_\theta(H,v_1,w_{(H,\pi,s)})$ \label{cw:least-dangerous-threat} \;
  }
  $\sigma:=\alpha_i$ \label{cw:select-color} \;
  $w^i:=\sum_{s\in[r]\setminus\{\sigma\}} d_s^i$ \label{cw:primary-weight} \;
  $j:=0$ \;
  \Repeat({(**)}){all $(H,\pi)\in \cC(\cH_\sigma,F)$, $\pi=(v_1,\ldots,v_h)$, satisfy $d_\theta(H,v_1,w_{(H,\pi,\sigma)})<d_\sigma^i$ \label{cw:this-round-complete} }{
    $j:=j+1$ \;
    $\cC^{i,j}:=\big\{(H,\pi=(v_1,\ldots,v_h))\in \cC(\cH_\sigma,F)\bigmid[\big] d_\theta(H,v_1,w_{(H,\pi,\sigma)})=d_\sigma^i\big\}$ \label{cw:primary-threats} \;
    \If{$j=1$}{
      $\cC_\sigma(d_\sigma^i):=\cC^{i,1}$ \label{cw:forward-threats} \;
    }
    \ForEach{$(H,\pi)\in\cC^{i,j}$ \label{cw:set-primary-weight-loop}}{
      $w_\sigma(H,\pi):=w^i$ \label{cw:set-primary-weight} \;
    }
    $\cH_\sigma:=\cH_\sigma\cup\cC^{i,j}$ \label{cw:complete-primary-threats} \;
    $k:=0$ \;
    \Repeat({(***)}){$\cC^{i,j,k}=\emptyset$ \label{cw:secondary-threats-complete} }{
      $k:=k+1$ \;
      $\cT^{i,j,k}:=\big\{(H,\pi=(v_1,\ldots,v_h))\in \cC(\cH_\sigma,F)\bigmid[\big] d_\theta(H,v_1,w_{(H,\pi,\sigma)})\geq d_\sigma^i\big\}$ \label{cw:potential-secondary-threats} \;
      $\cC^{i,j,k}:=\emptyset$ \label{cw:secondary-threats-init} \;
      \ForEach{$(H,\pi)\in\cT^{i,j,k}$, $\pi=(v_1,\ldots,v_h)$, \label{cw:through-potential-secondary-threats} }{
        \If{$d_\theta(H,v_1,w_{(H,\pi,\sigma)})>d_\sigma^i$ or $\nexists J\seq H:v_1\in J\wedge (J,\pi|_J)\in\cC_\sigma(d_\sigma^i)$ \label{cw:secondary-threat-condition} }{
          \If{$\nexists J\seq H:v_1\in J\wedge (J,\pi|_J)\in\cC_\sigma(d_\theta(H,v_1,w_{(H,\pi,\sigma)}))$ \label{cw:condition-lower-index} }{
            $\ihat:=\max\{1\leq\ibar\leq i \mid \alpha_\ibar=\sigma \wedge d_\theta(H,v_1,w_{(H,\pi,\sigma)})<d_\sigma^\ibar \big\}$ \label{cw:set-weight-index-lower} \;
          }
          \Else{
            $\ihat:=\max\{1\leq\ibar\leq i \mid \alpha_\ibar=\sigma \wedge d_\theta(H,v_1,w_{(H,\pi,\sigma)})\leq d_\sigma^\ibar \big\}$ \label{cw:set-weight-index-higher} \;
          }
          $w_\sigma(H,\pi):=w^\ihat$ \label{cw:set-secondary-weight} \;
          $\cC^{i,j,k}:=\cC^{i,j,k}\cup\{(H,\pi)\}$ \label{cw:secondary-threats} \;
        }
      }
      $\cH_\sigma:=\cH_\sigma\cup\cC^{i,j,k}$ \label{cw:complete-secondary-threats} \;
    }
  }
}
\Return $\big((\cH_s,w_s)\big)_{s\in[r]}$ \;
\end{algorithm}

Consider now the algorithm $\CW()$ as given in Algorithm~\ref{algo:cw}. In the following we will try to convey an intuitive understanding of its operation, building on the informal remarks given in Section~\ref{sec:intuition-computing}.

The algorithm works in rounds, and each round corresponds to relaxing the generalized density restriction $(\theta, \beta)$ by slightly lowering $\beta$, and then fully exploring Builder's options that become available as a consequence. Each iteration of the repeat-loop~(*) is one such round. 

At the beginning of the $i$-th round, for every color $s\in[r]$ the maximal $d_\theta()$-value among all graphs in $\cC(\cH_s,F)$, denoted by $d_s^i$, is determined (lines~\ref{cw:least-dangerous-threat-loop}--\ref{cw:least-dangerous-threat}). Here the sets $\cC(\cH_s,F)$ correspond to all graphs in color $s$ that Builder could try to enforce next, and considering for each color a graph that \emph{maximizes} $d_\theta()$ yields a new construction step for which $\beta$ needs to be lowered \emph{least} in order for that step to be compliant with the generalized density restriction $(\theta,\beta)$. (Specifically, $\beta$ needs to be lowered to $\beta_i:=1+\sum_{s\in[r]}d_s^i$; note however that this successive lowering of $\beta$ ist not done explicitly in the algorithm.) 
 
The $i$-th entry of the sequence $\alpha$ is then used to determine Painter's coloring decision $\sigma:=\alpha_i$ for this construction step (line~\ref{cw:select-color}), and the rest of the round consists of updating the families $\cH_s$ and the functions $w_s$ with all the information that can be extracted from that decision. In fact, only the family $\cH_\sigma$ grows; the families $\cH_s$, $s\in[r]\setminus\{\sigma\}$, do not change. 

The value $w^i:=\sum_{s\in[r]\setminus\{\sigma\}} d_s^i$ defined in line~\ref{cw:primary-weight} corresponds to the weight that needs to be assigned to the youngest vertex of every graph that is completed in color $\sigma$ as a direct consequence of Painter's coloring decision. When the repeat-loop~(**) is executed for the first time, those graphs are added to $\cH_\sigma$ via the set $\cC^{i,1}\seq\cS(F)$, and the function $w_\sigma: \cH_\sigma \to \RR$ is updated by assigning the value $w^i$ to the newly created graphs (lines~\ref{cw:primary-threats}--\ref{cw:complete-primary-threats}). For technical reasons, these graphs are also stored separately in a set $\cC_\sigma(d^i_\sigma)$ that will be relevant later in the algorithm.

The remainder of the $i$-th round explores options that became available to Builder as a result of the graphs in $\cC^{i,1}$ being added to $\cH_\sigma$. These graphs can now be used themselves for further construction steps, and the graphs created in those construction steps can be used even further, etc. Some of these new potential construction steps are not legal for the current generalized density restriction $(\theta, \beta_i)$, and will therefore only be explored in later rounds when (intuitively) $\beta$ is lowered further. However, some of these are indeed legal for the current value of $\beta$, and it turns out that the previous decisions of Painter already imply which colors Painter should use in each of those construction steps (!). These indirect consequences of Painter's decision to use color $\sigma=\alpha_i$ in round $i$ are explored in the repeat-loop~(***), and in the repeat-loop~(**) when it is executed for $j\geq 2$. The resulting graphs are added to $\cH_\sigma$ via the sets $\cT^{i,j,k},\cC^{i,j,k}\seq\cS(F)$ in the repeat-loop~(***), and via the sets $\cC^{i,j}\seq\cS(F)$, $j\geq 2$, in the repeat-loop~(**). This exploration of indirect consequences involves some technicalities for which we cannot give much intuition; see however the remarks in the first three paragraphs of Section~\ref{sec:implementation} below. Note that the sets $\cC_\sigma(d^i_\sigma)$ defined in line~\ref{cw:forward-threats} (in this or an earlier round) come back into play in lines~\ref{cw:secondary-threat-condition}--\ref{cw:condition-lower-index}.

The $i$-th round terminates as soon as all ordered graphs $(H,\pi)\in\cC(\cH_\sigma,F)$, $\pi=(v_1,\ldots,v_h)$, satisfy $d_\theta(H,v_1,w_{(H,\pi,\sigma)})<d_\sigma^i$ (line~\ref{cw:this-round-complete}). This corresponds to Builder having exhausted all his legal options in the game with generalized density restriction $(\theta, \beta_i)$ (recall that $\beta_i=1+\sum_{s\in[r]}d_s^i$). The $(i+1)$-th round of the algorithm will then consider the game with generalized density restriction $(\theta, \beta_{i+1})$ for some $\beta_{i+1}<\beta_i$.

The algorithm terminates as soon as one of the families $\cH_s$, $s\in[r]$, contains all ordered subgraphs of~$F$, i.e., $\cH_s=\cS(F)$ (line~\ref{cw:terminate}). This corresponds to Builder having enforced copies of all ordered subgraphs of~$F$ in color $s$ (in particular, monochromatic copies of $F$ in all possible vertex orderings).

We defer the formal arguments that $\CW()$ is a well-defined algorithm and terminates correctly to Section~\ref{sec:basic-properties}, where we will prove the following claim.

\begin{lemma}[Well-definedness and termination of algorithm] \label{lemma:algo-well-defined} 
All expressions that occur in the algorithm $\CW()$ are well-defined, all numerical values and all sets that occur are finite, and on any input as specified the algorithm terminates correctly after at most $r\cdot|\cS(F)|$ iterations of the repeat-loop~(*).
\end{lemma}

With Algorithm~\ref{algo:cw} in hand, we now define the parameter $\Lambda_\theta(F,r)$ for which we will prove Proposition~\ref{prop:Lambda-Builder} and Proposition~\ref{prop:Lambda-Painter}.

For a fixed real number $\theta>0$, any graph $F$ with at least one edge and any integer $r\geq 2$ we define
\begin{equation} \label{eq:def-Lambda}
  \newcommand\mystrut{\vphantom{]^\{}}
  \Lambda_\theta(F,r):=
  \min_{\alpha\in[r]^{r\cdot|\cS(F)|}}
  \max_{\substack{s\in[r]\mystrut \\ \pi\in\Pi(V(F))}}
  \min_{H\seq F:v(H)\geq 1\mystrut} \lambda_\theta(H,w_{(H,\pi|_H,s)}) \enspace,
\end{equation}
where $\lambda_\theta()$ is defined in \eqref{eq:def-lambda}, and $w_{(H,\pi,s)}()$ is defined for all $(H,\pi)\in\cS(F)$ and all $s\in[r]$ in \eqref{eq:def-w} using the results $\big((\cH_s,w_s)\big)_{s\in[r]}:=\CW(F,r,\theta,\alpha)$ of Algorithm~\ref{algo:cw}.

We defer the formal arguments that $\Lambda_\theta(F,r)$ is well-defined to Section~\ref{sec:basic-properties}, where we will prove the following claim.

\begin{lemma}[Well-definedness of $\Lambda_\theta(F,r)$] \label{lemma:Lambda-well-defined}
For any real number $\theta>0$, any graph $F$ with at least one edge and any integer $r\geq 2$, the parameter $\Lambda_\theta(F,r)$ defined in~\eqref{eq:def-Lambda} is a well-defined finite value.
\end{lemma}

Note that for rational values of $\theta>0$, the parameter $\Lambda_\theta(F,r)$ can be computed using only integer arithmetic.

Before we begin with the technical analysis of the algorithm $\CW()$ in Sections~\ref{sec:basic-properties}---\ref{sec:further-properties}, we give a few remarks about its implementation in the next section.

\subsection{Simplifications and implementation of the algorithm} \label{sec:implementation}

By Theorem~\ref{thm:main-1'}, we can compute the online vertex-Ramsey density $m_1^*(F,r)$ as the inverse of the root of the parameter $\Lambda_\theta(F,r)$ defined in \eqref{eq:def-Lambda}, where this definition involves the return values of the algorithm $\CW()$. As it turns out, the algorithm $\CW()$ can be simplified considerably if one is only interested in computing the online vertex-Ramsey density $m_1^*(F,r)$ for given $F$ and $r$ (and not in proving Theorem~\ref{thm:main-1} or Theorem~\ref{thm:main-2}, or in computing explicit winning strategies for Builder and Painter). The program to compute $m_1^*(F,r)$ that is available from the authors' websites \cite{homepage-mrs} uses such a simplified version of the pseudocode above.
In the following we outline the most important steps in this simplification.

First of all, the sets $\cC_s(d)$ defined in line~\ref{cw:forward-threats} and the case distinctions inside the repeat-loop~(***) whether certain subgraphs are contained in those sets or not can be omitted, as they are only used for proving the lower bound part of Theorem~\ref{thm:main-2}, our result for the probabilistic problem. Specifically, these extra technicalities are needed to bound the size of the witness graphs for certain coloring strategies that are derived from the algorithm $\CW()$ --- recall from Section~\ref{sec:proof-thm1} and Section~\ref{sec:proof-thm2} that such a bound is unimportant for the deterministic game, but crucial for the original probabilistic problem (see also Lemma~\ref{lemma:witness-graphs} and Remark~\ref{remark:vmax} below).

In a second step the algorithm can be simplified even further: As it turns out, the entire repeat-loop~(***) can be omitted; i.e.,  we do not need to compute any vertex weights for graphs $(H,\pi)\in\cC(\cH_\sigma,F)$, $\pi=(v_1,\ldots,v_h)$, that satisfy $d_\theta(H,v_1,w_{(H,\pi,\sigma)})>d_\sigma^i$ for the current value of $d_\sigma^i$, and we do not need to add such graphs to the corresponding family $\cH_\sigma$ (thus for the color $\sigma$ the algorithm will ignore the entire subtree of $\cT(F)$ rooted at $(H,\pi)$). The reason for this is that such graphs are essentially useless for Builder, and therefore we do not need to consider them in our systematic exploration of Builder's options (see Lemma~\ref{lemma:partner} and Algorithm~\ref{algo:builder} below).

Yet another simplification follows from Lemma~\ref{lemma:Lambda-dsi-sum} below: Combining \eqref{eq:def-Lambda} and \eqref{eq:further-properties} shows that we can change the return value of the algorithm $\CW()$ to the sum on the right hand side of \eqref{eq:further-properties} (note that this sum is exactly $\beta_{\icheck}$, as used in our informal description of the algorithm). Thus we may stop the algorithm as soon as for some $\pi\in\Pi(V(F))$ the graph $(F,\pi)$ is added to one of the families $\cH_s$, $s\in[r]$, which may happen considerably earlier than the termination condition in line~\ref{cw:terminate}.

Further major savings are achieved by considering only induced subgraphs in the definition~\eqref{eq:def-S}, as pointed out in~Remark~\ref{remark:induced-subgraphs}.
 
Some of these modifications might change the values returned by the algorithm $\CW(F,r,\theta,\alpha)$ for a specific sequence $\alpha$ (as the families $\cH_s$, $s\in[r]$, may evolve differently in the course of the algorithm, the entries of $\alpha$ get a different semantic), but not the value of $\Lambda_\theta(F,r)$ as defined in~\eqref{eq:def-Lambda}.

We conclude this section by sketching some ideas to further speed up the computation of $m_1^*(F,r)$ that do not directly relate to the pseudocode given in Algorithm~\ref{algo:cw}.

When evaluating $\Lambda_\theta(F,r)$ for a given $\theta\in(0,2)$, rather than calling the algorithm $\CW()$ for each possible input sequence $\alpha\in[r]^{r\cdot|\cS(F)|}$ separately, we call it only once and, in each iteration, branch on all $r$ values the variable $\sigma$ can assume in line~\ref{cw:select-color}. Since most of the branches of the resulting recursion tree end after much fewer than $r\cdot|\cS(F)|$ iterations, this allows us to evaluate the minimization in \eqref{eq:def-Lambda} and hence the value of $\Lambda_\theta(F,r)$ much more efficiently. 

By Theorem~\ref{thm:main-1'} we have $m_1^*(F,r)=1/\theta^*$, where $\theta^*=\theta^*(F,r)$ is the unique root of $\Lambda_\theta(F,r)$ defined in \eqref{eq:def-Lambda}, which is guaranteed to be in the finite set $Q(\amax)$.
In order to efficiently search for $\theta^*$ in $Q(\amax)$, we can exploit that the function $\Lambda_\theta(F,r)$ changes its sign from positive to negative at $\theta^*$. Specifically, in order to compute $\theta^*$, we alternate between shrinking the possible interval for the root $\theta^*$ by binary search (starting with the interval $(0,2)$), and evaluating $\Lambda_\theta(F,r)$ for all rational values of $\theta$ inside the current interval up to a certain size of the denominator.

\subsection{Basic properties of the algorithm} \label{sec:basic-properties}

In this section we establish a number of basic properties of the algorithm $\CW()$, including several important monotonicity properties. We also provide the proofs for Lemma~\ref{lemma:algo-well-defined} and Lemma~\ref{lemma:Lambda-well-defined}.

We begin by proving that the families $\cH_s$ grow downward from the root in the tree $\cT(F)$ throughout the algorithm, as already mentioned.

\begin{lemma}[Closure property of families $\cH_s$] \label{lemma:closure-Hs}
Throughout the algorithm $\CW()$ and for each $s\in[r]$ we have that if $(H,\pi)$, $\pi=(v_1,\ldots,v_h)$, $h\geq 2$, is in $\cH_s$, then $(H\setminus v_1,\pi\setminus v_1)$ is also in $\cH_s$. In particular, if $\cH_s\neq\emptyset$ then $(K_1,(v_1))\in\cH_s$. 
\end{lemma}

\begin{proof}
Observe that graphs are only added to $\cH_s$ in lines~\ref{cw:complete-primary-threats} and \ref{cw:complete-secondary-threats} of iterations for which $\alpha_i=s$, via the sets $\cC^{i,j}$ and $\cC^{i,j,k}\seq\cT^{i,j,k}$. Thus by the definition of $\cC^{i,j}$ in line~\ref{cw:primary-threats} and of $\cT^{i,j,k}$ in line~\ref{cw:potential-secondary-threats}, only graphs that are currently in~$\cC(\cH_s,F)$ are added to $\cH_s$. The claim thus follows from the definition of $\cC(\cH_s,F)$ in \eqref{eq:def-C}.
\end{proof}

Next we prove the first part of Lemma~\ref{lemma:algo-well-defined}, which states that all expressions that occur in the algorithm $\CW()$ are well-defined and that all numerical values and all sets that occur are finite. (We ignore the assignment $\sigma:=\alpha_i$ in line~\ref{cw:select-color} for the time being --- well-definedness of that assignment is immediate once we have proven the second part of Lemma~\ref{lemma:algo-well-defined}, namely that the algorithm $\CW()$ terminates correctly after at most $r\cdot|\cS(F)|$ rounds.)

\begin{proof}[Proof of Lemma~\ref{lemma:algo-well-defined} (Well-definedness of algorithm)] 
We need to argue that the expression $d_\theta(H,v_1,w_{(H,\pi,s)})$ (wherever it occurs) is well-defined, and that the maximum in line~\ref{cw:least-dangerous-threat}, line~\ref{cw:set-weight-index-lower} and line~\ref{cw:set-weight-index-higher} is always over a nonempty set.

Note that $w_\sigma(H,\pi)$ is defined in line~\ref{cw:set-primary-weight} or line~\ref{cw:set-secondary-weight}, just before $(H,\pi)$ is added to $\cH_\sigma$ in line~\ref{cw:complete-primary-threats} or line~\ref{cw:complete-secondary-threats}, respectively. Combining this with Lemma~\ref{lemma:closure-Hs} and using the definition in \eqref{eq:def-w}, it follows that the function $w_{(H,\pi,s)}$ defines finite vertex weights for all vertices of every graph $(H,\pi)\in\cH_s$. Consequently, for every graph $(H,\pi)\in \cC(\cH_s,F)$, $\pi=(v_1,\ldots,v_h)$, the function $w_{(H,\pi,s)}$ defines finite weights for \emph{all but the youngest vertex $v_1$}. As throughout the algorithm the function $d_\theta(H,v_1,w_{(H,\pi,s)})$ is evaluated only for graphs from the corresponding set $\cC(\cH_s,F)$, it follows that this expression (wherever it occurs) is indeed well-defined (recall the definition of $d_\theta()$ in \eqref{eq:def-d}).

As long as none of the families $\cH_s$, $s\in[r]$, contains all graphs from $\cS(F)$, the sets $\cC(\cH_s,F)$ are nonempty. Thus the termination condition in line~\ref{cw:terminate} ensures that the maximum in line~\ref{cw:least-dangerous-threat} is always over a nonempty set and therefore well-defined.

\label{page:K1}
For any $\sigma\in[r]$, let $\icheck$ be the smallest integer $i\geq 1$ for which $\alpha_i=\sigma$ and note that the $\icheck$-th iteration of the repeat-loop~(*) is the first one where graphs are added to the family $\cH_\sigma$ (initially, we have $\cH_\sigma=\emptyset$). As $\cC(\emptyset,F)=\{(K_1,(v_1))\}$ and $d_\theta(K_1,v_1,w_{(K_1,(v_1),\sigma)})=0$, by the definitions in lines~\ref{cw:least-dangerous-threat}, \ref{cw:primary-threats}, and \ref{cw:forward-threats} we have  $d_\sigma^\icheck=0$, and $(K_1,(v_1))$ is contained in $\cC^{\icheck,1}=\cC_\sigma(0)$. By the definition of $d_\theta()$ in \eqref{eq:def-d}, $d_\theta(H,v_1,w_{(H,\pi,\sigma)})$ on the right hand side of line~\ref{cw:set-weight-index-higher} is always non-positive, i.e.\ less than or equal to $d_\sigma^\icheck=0$, implying that the maximum in line~\ref{cw:set-weight-index-higher} is always over a nonempty set and therefore well-defined (this set contains at least the integer $\icheck$). We can argue analogously for the set on the right hand side of line~\ref{cw:set-weight-index-lower} if $d_\theta(H,v_1,w_{(H,\pi,\sigma)})<0$. On the other hand, if $d_\theta(H,v_1,w_{(H,\pi,\sigma)})=0$ then by what we said before the subgraph $J\seq H$ that consists only of the vertex $v_1$ is contained in $\cC_\sigma(0)=\cC_\sigma(d_\theta(H,v_1,w_{(H,\pi,\sigma)}))$, and thus the condition in line~\ref{cw:condition-lower-index} is violated.

For this last argument we used that no graphs are ever removed from the sets $\cC_s(d)$, $d\in\RR$. We have not shown this yet formally; however, it follows from Lemma~\ref{lemma:di-wi-monotonicity} below that after the initialization in line~\ref{cw:Cd-init}, each set $\cC_s(d)$ is modified at most once, namely in line~\ref{cw:forward-threats} of the unique iteration $i$ for which $\alpha_i=s$ and $d^i_s=d$. (If the reader is worried about this forward reference, he is welcome to substitute line~\ref{cw:forward-threats} by $\cC_\sigma(d_\sigma^i):=\cC_\sigma(d_\sigma^i)\cup\cC^{i,1}$ for the time being, i.e., until Lemma~\ref{lemma:di-wi-monotonicity} is proven.)
\end{proof}

Having established that all numerical values assigned in the algorithm $\CW()$ are finite, we state the following observations for further reference.

\begin{lemma}[Finite and non-positive weights] \label{lemma:finite-weights}
Throughout the algorithm $\CW()$, for each $s\in[r]$ we have that for each $(H,\pi)\in\cH_s$, $\pi=(v_1,\ldots,v_h)$, all vertex weights $w_{(H,\pi,s)}(v_i)$ defined in~\eqref{eq:def-w} are finite non-positive values, while for each $(H,\pi)\in\cS(F)\setminus\cH_s$, at least one of the vertex weights is $-\infty$.

Consequently, for all $(H,\pi)\in\cH_s$, $\lambda_\theta(H,w_{(H,\pi,s)})$ defined in~\eqref{eq:def-lambda} is a finite value bounded by $v(F)$, while for all $(H,\pi)\in\cS(F)\setminus\cH_s$, we have $\lambda_\theta(H,w_{(H,\pi,s)})=-\infty$.
\end{lemma}

\begin{proof}
By the definition of $d_\theta()$ in \eqref{eq:def-d}, $d_\theta(H,v_1,w_{(H,\pi,s)})$ on the right hand side of line~\ref{cw:least-dangerous-threat} is always non-positive and finite, from which we conclude, using the definitions in line~\ref{cw:primary-weight}, \ref{cw:set-primary-weight} and \ref{cw:set-secondary-weight}, that $w_s(H,\pi)$ is a non-positive finite value for all $s\in[r]$ and all $(H,\pi)\in\cH_s$. The first part of the statement now follows from the definition in~\eqref{eq:def-w} and Lemma~\ref{lemma:closure-Hs}. The second part follows from the first part using the definition in~\eqref{eq:def-lambda}.
\end{proof}

The next lemma establishes two important monotonicity properties, which will be used in many of the upcoming proofs.

\begin{lemma}[Monotonicity of $d_s^i$ and $w^i$ in $i$] \label{lemma:di-wi-monotonicity}
Let $\sigma\in[r]$ and $\alpha\in[r]^{r\cdot|\cS(F)|}$ be the input sequence of the algorithm $\CW()$.
Throughout the algorithm, if $\alpha_i=\sigma$, then the variables $d_s^i, d_s^{i+1}$, $s\in[r]$, defined in line~\ref{cw:least-dangerous-threat} satisfy
\begin{equation*}
  d_\sigma^{i+1}<d_\sigma^i \quad\text{and}\quad d_s^{i+1}=d_s^i \quad \text{for all $s\in[r]\setminus\{\sigma\}$} \enspace.
\end{equation*}
Moreover, if $\alpha_\ibar=\alpha_i=\sigma$ for some $\ibar< i$, then the variables $w^\ibar,w^i$ defined in line~\ref{cw:primary-weight} satisfy
\begin{equation*}
  w^i\leq w^\ibar \enspace,
\end{equation*}
with equality if and only if $\alpha_\ibar=\alpha_{\ibar+1}=\cdots=\alpha_i=\sigma$. 
\end{lemma}

Of course, the variables $d_s^i$ and $d_s^{i+1}$ referred to in Lemma~\ref{lemma:di-wi-monotonicity} are defined only if the number of iterations of the repeat-loop~(*) is at least $i+1$. Similarly, $w^i$ and $w^\ibar$ are defined only if the number of iterations is at least $i$. Otherwise the statement of the lemma is void.

\begin{proof}
The first part of the lemma follows from the definition of $d_\sigma^i$ in line~\ref{cw:least-dangerous-threat}, the termination condition in line~\ref{cw:this-round-complete}, and the fact that none of the families $\cH_s$, $s\in[r]\setminus\{\sigma\}$, is modified within the $i$-th iteration of the repeat-loop~(*).
The second part of the lemma follows from the first part and the definition of $w^i$ in line~\ref{cw:primary-weight}.
\end{proof}

For the proof that $\CW()$ terminates correctly after at most $r\cdot|\cS(F)|$ iterations of the repeat-loop~(*) we will need the following auxiliary statement.

\begin{lemma}[End of repeat-loop~(**)] \label{lemma:end-**}
Throughout the algorithm $\CW()$, at the end of each iteration of the repeat-loop~(**), all graphs $(H,\pi)\in\cC(\cH_\sigma,F)$, $\pi=(v_1,\ldots,v_h)$, satisfy
\begin{equation} \label{eq:end-**}
  d_\theta(H,v_1,w_{(H,\pi,\sigma)})\leq d_\sigma^i \enspace.
\end{equation}
For those graphs satisfying \eqref{eq:end-**} with equality there is a subgraph $J\seq H$ with $v_1\in J$ and $(J,\pi|_J)\in\cC_\sigma(d_\sigma^i)$.
\end{lemma}

\begin{proof} The $j$-th iteration of the repeat-loop~(**) ends as soon as the repeat-loop~(***) terminates. Due to the condition in line~\ref{cw:secondary-threats-complete} this happens in the first iteration $k$ for which the set $\cC^{i,j,k}\seq \cT^{i,j,k}$ is empty, which means that all graphs currently in~$\cC(\cH_\sigma,F)$ violate the condition in the definition of $\cT^{i,j,k}$ in line~\ref{cw:potential-secondary-threats}, or the condition in line~\ref{cw:secondary-threat-condition}. This implies the claim.
\end{proof}

We now prove the second part of Lemma~\ref{lemma:algo-well-defined}, namely that the algorithm $\CW()$ terminates correctly after at most $r\cdot|\cS(F)|$ rounds.

\begin{proof} [Proof of Lemma~\ref{lemma:algo-well-defined} (Termination of algorithm)]
Before bounding the number of iterations of the repeat-loop~(*) we need to argue that the inner two repeat-loops always terminate.

Let $\sigma\in[r]$ and suppose that we have $\alpha_i=\sigma$ in the current iteration $i$ of the repeat-loop~(*). Note that in each iteration of the repeat-loop~(***) except the last one, at least one element is added to the family $H_\sigma$ via the set $\cC^{i,j,k}$ in~line~\ref{cw:complete-secondary-threats}. Since throughout the algorithm, $\cH_\sigma$ is a subfamily of $\cS(F)$, and since no graphs are ever deleted from $\cH_\sigma$, the repeat-loop~(***) terminates after at most $|\cS(F)|+1$ iterations.

It follows directly from the definition of $d_\sigma^i$ in line~\ref{cw:least-dangerous-threat} that in the first iteration $j=1$ of the repeat-loop~(**), the set $\cC^{i,1}$ defined in line~\ref{cw:primary-threats} is nonempty. As a consequence of the first part of Lemma~\ref{lemma:end-**} and the termination condition in line~\ref{cw:this-round-complete}, the set $\cC^{i,j}$ is also nonempty in all later iterations $j>1$. Therefore, in each iteration of the repeat-loop~(**), at least one element is added to the family $\cH_\sigma$ via the set $\cC^{i,j}$ in line~\ref{cw:complete-primary-threats}, and thus similarly to above the repeat-loop~(**) terminates after at most $|\cS(F)|$ iterations. \label{page:j-loop}

The above also implies that in each iteration of the repeat-loop~(*), the size of exactly one of the families $\cH_s$, $s\in[r]$, increases by at least one. Considering the condition in line~\ref{cw:terminate}, we conclude that the algorithm terminates after at most $r\cdot|\cS(F)|$ iterations of the repeat-loop~(*).
\end{proof}

We proceed by proving Lemma~\ref{lemma:Lambda-well-defined}, which states that $\Lambda_\theta()$ defined in~\eqref{eq:def-Lambda} is a well-defined finite value.

\begin{proof}[Proof of Lemma~\ref{lemma:Lambda-well-defined}]
Due to the termination condition in line~\ref{cw:terminate} of $\CW()$, for each possible input sequence $\alpha\in[r]^{r\cdot |\cS(F)|}$ there is an $s\in[r]$ for which the family $\cH_s$ returned by $\CW()$ equals $\cS(F)$. By Lemma~\ref{lemma:finite-weights}, for this $s$ the parameter $\lambda_\theta(H,w_{(H,\pi,s)})$ is a finite value for all $(H,\pi)\in\cS(F)$. Thus for any fixed $\alpha\in[r]^{r\cdot |\cS(F)|}$ the maximization over all $s\in[r]$ in~\eqref{eq:def-Lambda} yields a finite value, and consequently also the outer minimization over all (finitely many) sequences $\alpha\in[r]^{r\cdot |\cS(F)|}$ in~\eqref{eq:def-Lambda} yields a finite value.
\end{proof}

We continue with a simple invariant that holds at the beginning of each iteration of the main loop of~$\CW()$.

\begin{lemma}[Beginning of repeat-loop~(*)] \label{lemma:beginning-loop-*}
Throughout the algorithm $\CW()$, at the beginning of the $i$-th iteration of the repeat-loop~(*), for every $s\in[r]$ all graphs $(H,\pi)\in\cH_s$, $\pi=(v_1,\ldots,v_h)$, satisfy $d_\theta(H,v_1,w_{(H,\pi,s)})>d_s^i$.
\end{lemma}

\begin{proof}
Fix $i$, $s\in[r]$, and $(H,\pi)\in \cH_s$ as in the lemma, and let $\ibar<i$ denote the iteration in which $(H,\pi)$ was added to $\cH_s$. We must have $\alpha_\ibar=s$, and $(H,\pi)$ was added to $\cH_s$ either via one of the sets $\cC^{\ibar,j}$ in line~\ref{cw:complete-primary-threats}, or via one of the sets $\cC^{\ibar,j,k}\seq\cT^{\ibar,j,k}$ in line~\ref{cw:complete-secondary-threats}. By the conditions in the definition of $\cC^{\ibar,j}$ in line~\ref{cw:primary-threats} and of $\cT^{\ibar,j,k}$ in line~\ref{cw:potential-secondary-threats} we have $d_\theta(H,v_1,w_{(H,\pi,s)})\geq d_s^\ibar > d_s^{\ibar+1} \geq  d_s^i$, where the last two inequalities follow from the first part of Lemma~\ref{lemma:di-wi-monotonicity}.
\end{proof}

The next lemma takes a closer look at the $d_\theta()$-value of graphs that are added to the families $\cH_s$ via one of the sets $\cC^{i,j,k}$ in the repeat-loop (***).

\begin{lemma}[Sandwiched $d_\theta()$-values for graphs in $\cC^{i,j,k}$] \label{lemma:d-sandwiched}
Let $\sigma\in[r]$ and $\alpha\in[r]^{r\cdot|\cS(F)|}$ be the input sequence of the algorithm $\CW()$.
Throughout the algorithm, if $\alpha_i=\sigma$, then for any graph $(H,\pi)$, $\pi=(v_1,\ldots,v_h)$, that is added to the set $\cC^{i,j,k}$ in line~\ref{cw:secondary-threats} (for this graph $w_\sigma(H,\pi)$ is defined in line~\ref{cw:set-secondary-weight} by setting it to $w^\ihat$) the following holds: If $\ihat$ is defined in line~\ref{cw:set-weight-index-lower}, we have
\begin{subequations} \label{eq:d-between}
\begin{equation} \label{eq:d-between-leq-l}
  d_\sigma^{\ihat+1} \leq d_\theta(H,v_1,w_{(H,\pi,\sigma)}) < d_\sigma^\ihat \enspace.
\end{equation}
Otherwise, i.e.\ if $\ihat$ is defined in line~\ref{cw:set-weight-index-higher}, we have
\begin{equation} \label{eq:d-between-l-leq}
  d_\sigma^{\ihat+1} < d_\theta(H,v_1,w_{(H,\pi,\sigma)}) \leq d_\sigma^\ihat \enspace.
\end{equation}
\end{subequations}
\end{lemma}

\begin{proof}
Clearly, the definition of $\ihat$ in line~\ref{cw:set-weight-index-lower} implies
\begin{equation} \label{eq:d-sec-threat-ub}
  d_\theta(H,v_1,w_{(H,\pi,\sigma)}) < d_\sigma^\ihat \enspace.
\end{equation}
As $(H,\pi)$ is in $\cC^{i,j,k}\seq \cT^{i,j,k}$, by the definition in line~\ref{cw:potential-secondary-threats} we have $d_\theta(H,v_1,w_{(H,\pi,\sigma)})\geq d_\sigma^i$, and consequently $\ihat<i$. Consider the smallest integer $\nu\geq 1$ for which $\alpha_{\ihat+\nu}=\alpha_\ihat=\sigma$, and note that $\ihat+\nu\leq i$.
Again by the definition of $\ihat$ in line~\ref{cw:set-weight-index-lower} we have
\begin{equation} \label{eq:d-sec-threat-lb}
  d_\sigma^{\ihat+\nu} \leq d_\theta(H,v_1,w_{(H,\pi,\sigma)}) \enspace.
\end{equation}
Moreover, by the choice of $\nu$ and the first part of Lemma~\ref{lemma:di-wi-monotonicity} we have
\begin{equation} \label{eq:d-no-change}
  d_\sigma^{\ihat+\nu}=d_\sigma^{\ihat+1} \enspace.
\end{equation}
Combining \eqref{eq:d-sec-threat-lb} and \eqref{eq:d-no-change} yields the first inequality in \eqref{eq:d-between-leq-l}, and together with ~\eqref{eq:d-sec-threat-ub} proves the first part of the lemma. The second part follows similarly by using the definition of $\ihat$ in line~\ref{cw:set-weight-index-higher} and by interchanging $<$ with $\leq$ in \eqref{eq:d-sec-threat-ub} and \eqref{eq:d-sec-threat-lb} (also in this case we must have $\ihat<i$, as otherwise we would have $d_\theta(H,v_1,w_{(H,\pi,\sigma)})=d_\sigma^i$, a contradiction to the condition in line~\ref{cw:secondary-threat-condition} and the negation of the condition in line~\ref{cw:condition-lower-index}).
\end{proof}

The next lemma captures another important monotonicity condition: the lower the $d_\theta()$-value of an ordered graph in a particular color is, the smaller is the weight assigned to its youngest vertex.

\begin{lemma}[$d_\theta()$-value vs.\ weight monotonicity] \label{lemma:danger-weight-monotonicity}
Throughout the algorithm $\CW()$, for every $s\in[r]$ and any two graphs $(H,\pi), (J,\tau)\in\cH_s$, $\pi=(v_1,\ldots,v_h)$, $\tau=(u_1,\ldots,u_c)$, the following two properties hold:
\begin{itemize}
\item
If
\begin{equation*}
  d_\theta(H,v_1,w_{(H,\pi,s)})<d_\theta(J,u_1,w_{(J,\tau,s)})
\end{equation*}
then we have
\begin{equation*}
  w_{(H,\pi,s)}(v_1)\leq w_{(J,\tau,s)}(u_1) \enspace.
\end{equation*}

\item
If
\begin{equation*}
  d_\theta(H,v_1,w_{(H,\pi,s)})=d_\theta(J,u_1,w_{(J,\tau,s)})
  \quad \text{and} \quad
  w_{(H,\pi,s)}(v_1)<w_{(J,\tau,s)}(u_1)
\end{equation*}
then $w_{(H,\pi,s)}(v_1)\eqBy{eq:def-w} w_s(H,\pi)$ is defined either in line~\ref{cw:set-primary-weight} or in line~\ref{cw:set-secondary-weight} with $\ihat$ defined in line~\ref{cw:set-weight-index-higher}, and $w_{(J,\tau,s)}(u_1)=w_s(J,\tau)$ is defined in line~\ref{cw:set-secondary-weight} with $\ihat$ defined in line~\ref{cw:set-weight-index-lower}.
\end{itemize}
\end{lemma}

For the second part of Lemma~\ref{lemma:danger-weight-monotonicity} note that $w_s(H,\pi)$ and $w_s(J,\tau)$ are defined exactly once in the course of the algorithm (each in some iteration of the various repeat-loops), just before the corresponding graph $(H,\pi)$ or $(J,\tau)$ is added to the family $\cH_s$.
 
\begin{proof}
Let $\imax$ denote the total number of iterations of the repeat-loop~(*).
Fix some $\sigma\in[r]$ and consider the set $R_\sigma\in\mathbb{R}^2$ defined by
\begin{equation} \label{eq:R-sigma}
  R_\sigma:=\bigcup_{1\leq i\leq \imax:\alpha_i=\sigma} \Big\{ \big((1-t)\cdot d_\sigma^i+t\cdot d_\sigma^{i+1},w^i\big) \bigmid[\big] t\in[0,1] \Big\} \enspace,
\end{equation}
where we use the convention $d_\sigma^{\imax+1}:=d_\sigma^\imax$ if $\alpha_\imax=\sigma$.
By Lemma~\ref{lemma:di-wi-monotonicity} we have for any two pairs $(x,y),(x',y')\in R_\sigma$
\begin{equation} \label{eq:R-monotonicity}
   x<x' \;\Longrightarrow\; y\leq y' \enspace.
\end{equation}
Now fix some graph $(H,\pi)\in\cH_\sigma$, $\pi=(v_1,\ldots,v_h)$, and consider the iteration $i$ of the repeat-loop~(*) where $\alpha_i=\sigma$ and where $(H,\pi)$ is added to the family $\cH_\sigma$. If $(H,\pi)$ is added to $\cH_\sigma$ in line~\ref{cw:complete-primary-threats} then we have
\begin{equation} \label{eq:d-w-pt}
  d_\theta(H,v_1,w_{(H,\pi,\sigma)})=d_\sigma^i \quad \text{and} \quad w_{(H,\pi,\sigma)}(v_1)\eqBy{eq:def-w} w_\sigma(H,\pi)=w^i
\end{equation}
by the definitions in line~\ref{cw:primary-threats} and line~\ref{cw:set-primary-weight}.
If on the other hand $(H,\pi)$ is added to $\cH_\sigma$ in line~\ref{cw:complete-secondary-threats} then we obtain with Lemma~\ref{lemma:d-sandwiched} that
\begin{equation} \label{eq:d-w-st}
  d_\sigma^{\ihat+1}\leq d_\theta(H,v_1,w_{(H,\pi,\sigma)}) \leq d_\sigma^\ihat \quad \text{and} \quad w_{(H,\pi,\sigma)}(v_1)\eqBy{eq:def-w} w_\sigma(H,\pi)=w^\ihat
\end{equation}
for some $\ihat\leq i$ with $\alpha_\ihat=\sigma$ (defined either in line~\ref{cw:set-weight-index-lower} or in line~\ref{cw:set-weight-index-higher}).

Combining \eqref{eq:R-sigma}, \eqref{eq:d-w-pt} and \eqref{eq:d-w-st} shows that any graph $(H,\pi)\in\cH_\sigma$, $\pi=(v_1,\ldots,v_h)$, satisfies
\begin{equation} \label{eq:d-w-in-R}
  \big(d_\theta(H,v_1,w_{(H,\pi,\sigma)}),w_{(H,\pi,\sigma)}(v_1)\big)\in R_\sigma \enspace.
\end{equation}
The first part of the claim now follows from \eqref{eq:R-monotonicity} and \eqref{eq:d-w-in-R}.

The second part of the claim follows from \eqref{eq:R-sigma}, \eqref{eq:d-w-in-R} and Lemma~\ref{lemma:di-wi-monotonicity} by examining where in the set $R_\sigma$ the point $\big(d_\theta(H,v_1,w_{(H,\pi,\sigma)}),w_{(H,\pi,\sigma)}(v_1)\big)$ can possibly be located, depending on whether $(H,\pi)$ is added to $\cH_\sigma$ in line~\ref{cw:complete-primary-threats} (then $w_\sigma(H,\pi)$ is defined in line~\ref{cw:set-primary-weight}) or in line~\ref{cw:complete-secondary-threats} (then $w_\sigma(H,\pi)$ is defined in line~\ref{cw:set-secondary-weight}), where the cases whether $\ihat$ is defined in line~\ref{cw:set-weight-index-lower} or in line~\ref{cw:set-weight-index-higher} have to be distinguished using \eqref{eq:d-between-leq-l} and \eqref{eq:d-between-l-leq} from Lemma~\ref{lemma:d-sandwiched}.
\end{proof}

In the following we will repeatedly use the following auxiliary statement, which is an immediate consequence of the definition in \eqref{eq:def-d}.

\begin{lemma}[Weights vs.\ $d_\theta()$-value monotonicity] \label{lemma:d-subgraph-monotonicity}
Let\/ $H$ be a graph, $v\in H$ and $J\seq H$ with $v\in J$. Moreover, let $w_H:V(H)\setminus\{v\}\rightarrow\RR$ and $w_J:V(J)\setminus\{v\}\rightarrow\RR$ with $w_H(u)\leq w_J(u)$ for all $u\in J\setminus v$. Then we have
 $d_\theta(H,v,w_H)\leq d_\theta(J,v,w_J)$.
\end{lemma}

The next lemma establishes an important monotonicity condition for the vertex weights with respect to taking (ordered) subgraphs: the weights of a subgraph are always at least as high as the weights of the entire graph.

\begin{lemma}[Subgraph weight monotonicity] \label{lemma:subgraph-monotonicity}
Throughout the algorithm $\CW()$, for every $s\in[r]$, if $(H,\pi)\in\cH_s$, then for every subgraph $J\seq H$ we have $(J,\pi|_J)\in\cH_s$ and $w_{(H,\pi,s)}(u)\leq w_{(J,\pi|_J,s)}(u)$ for all $u\in J$.
\end{lemma}

\begin{proof}
We will prove the following auxiliary claim: for every $s\in[r]$, if $(H,\pi)\in\cH_s$, $\pi=(v_1,\ldots,v_h)$, then for every subgraph $J\seq H$ with $v_1\in J$ we have $(J,\pi|_J)\in\cH_s$ and $w_{(H,\pi,s)}(u)\leq w_{(J,\pi|_J,s)}(u)$ for all $u\in J$. This implies the original claim, where the subgraphs $J\seq H$ are not required to contain the youngest vertex $v_1$, as follows: if $(H,\pi)\in\cH_s$, $\pi=(v_1,\ldots,v_h)$, and $J\seq H$ is \emph{any} subgraph of $H$, then by Lemma~\ref{lemma:closure-Hs} we also have $(H^{-c},\pi^{-c})\in\cH_s$ where $c:=\min\{i\mid v_i\in J\}-1$ and $(H^{-c},\pi^{-c}):=(H\setminus\{v_1,\ldots,v_c\},\pi\setminus\{v_1,\ldots,v_c\})$. Moreover, $(J,\pi|_J)$ contains the youngest vertex $v_{c+1}$ of $(H^{-c},\pi^{-c})$. Therefore, applying the auxiliary claim to $(H^{-c},\pi^{-c})$ and $(J,\pi|_J)$, together with the observation that $w_{(H,\pi,s)}(u)=w_{(H^{-c},\pi^{-c},s)}(u)$ for all $u\in H^{-c}$ completes the argument.

To prove the auxiliary claim we argue by induction over the number of vertices of $H$. The claim clearly holds if $H$ consists only of a single vertex, as then $J=H$ is the only subgraph of $H$. For the induction step let $\sigma\in[r]$ and consider a graph $(H,\pi)\in\cH_\sigma$, $\pi=(v_1,\ldots,v_h)$, with at least two vertices. We consider the iteration $i$ of the repeat-loop~(*) where $\alpha_i=\sigma$ and where $(H,\pi)$ is added to the family $\cH_\sigma$.
Let $J$ be a subgraph of $H$ with $v_1\in J$. By Lemma~\ref{lemma:closure-Hs} we have that $(H\setminus v_1,\pi\setminus v_1)\in\cH_\sigma$ at this point, and thus we know by induction that
\begin{equation} \label{eq:J-minus-Hs}
  (J\setminus v_1,\pi|_{J\setminus v_1})\in\cH_\sigma
\end{equation}
and that
\begin{equation} \label{eq:ineq-vertex-weights-ind}
  w_{(H,\pi,\sigma)}(u)\leq w_{(J,\pi|_J,\sigma)}(u) \quad \text{for all $u\in J\setminus v_1$} \enspace.
\end{equation}
To complete the proof we only need to show two things: Firstly, that $(J,\pi|_J)$ is either already contained in $\cH_\sigma$ or added to this set together with $(H,\pi)$ at the latest, and secondly, that $w_{(H,\pi,\sigma)}(v_1)\leq w_{(J,\pi|_J,\sigma)}(v_1)$ for the last vertex $v_1$.

Recall that graphs are only added to $\cH_\sigma$ via one of the sets $\cC^{i,j}$ in line~\ref{cw:complete-primary-threats}, or via one of the sets $\cC^{i,j,k}\seq \cT^{i,j,k}$ in line~\ref{cw:complete-secondary-threats}. If $(H,\pi)$ is contained in one of the sets $\cC^{i,j}$, then by the definition in line~\ref{cw:primary-threats} we have
\begin{subequations}
\begin{equation} \label{eq:d-H-eq-d-i}
  d_\theta(H,v_1,w_{(H,\pi,\sigma)})=d_\sigma^i \enspace,
\end{equation}
whereas if $(H,\pi)$ is contained in one of the sets $\cC^{i,j,k}$, then by the definition in line~\ref{cw:potential-secondary-threats} we have
\begin{equation} \label{eq:d-H-ge-d-i}
  d_\theta(H,v_1,w_{(H,\pi,\sigma)})\geq d_\sigma^i \enspace.
\end{equation}
\end{subequations}

Applying Lemma~\ref{lemma:d-subgraph-monotonicity} using~\eqref{eq:ineq-vertex-weights-ind} shows that
\begin{equation} \label{eq:ineq-d-values}
  d_\theta(H,v_1,w_{(H,\pi,\sigma)})\leq d_\theta(J,v_1,w_{(J,\pi|_J,\sigma)}) \enspace.
\end{equation}
We will distinguish the cases where the inequality \eqref{eq:ineq-d-values} is strict,
\begin{subequations}
\begin{equation} \label{eq:ineq-d-values-strict}
  d_\theta(H,v_1,w_{(H,\pi,\sigma)})<d_\theta(J,v_1,w_{(J,\pi|_J,\sigma)}) \enspace,
\end{equation}
and where it is tight,
\begin{equation} \label{eq:ineq-d-values-tight}
  d_\theta(H,v_1,w_{(H,\pi,\sigma)})=d_\theta(J,v_1,w_{(J,\pi|_J,\sigma)}) \enspace.
\end{equation}
\end{subequations}
Altogether we distinguish four cases: whether $(H,\pi)$ is contained in one of the sets $\cC^{i,j}$ or $\cC^{i,j,k}$, and whether the inequality \eqref{eq:ineq-d-values} is strict or tight.

\begin{itemize}
\item $(H,\pi)\in\cC^{i,j}$ and inequality \eqref{eq:ineq-d-values} is strict. Combining \eqref{eq:d-H-eq-d-i} and \eqref{eq:ineq-d-values-strict} yields
\begin{equation} \label{eq:d-J-gt-d-sigma-i}
  d_\theta(J,v_1,w_{(J,\pi|_J,\sigma)})>d_\sigma^i \enspace.
\end{equation}
By the definition of $d_\sigma^i$ in line~\ref{cw:least-dangerous-threat}, it follows from \eqref{eq:d-J-gt-d-sigma-i} that if $(J,\pi|_J)$ was not already contained in $\cH_\sigma$ at the beginning of the repeat-loop~(**) (in the $i$-th iteration of the repeat-loop~(*)), then at this point $(J\setminus v_1,\pi|_{J\setminus v_1})$ was not contained in $\cH_\sigma$ either. By \eqref{eq:J-minus-Hs} there must then be some $j'<j$ such that $(J\setminus v_1,\pi|_{J\setminus v_1})$ was added to $\cH_\sigma$ in the $j'$-th iteration of the repeat-loop~(**). Combining the first part of Lemma~\ref{lemma:end-**} and \eqref{eq:d-J-gt-d-sigma-i} shows that also $(J,\pi|_J)$ was added to $\cH_\sigma$ in the $j'$-th iteration of the repeat-loop~(**). In any case $(J,\pi|_J)$ is already contained in $\cH_\sigma$ when $(H,\pi)$ is added to this set. Applying the first part of Lemma~\ref{lemma:danger-weight-monotonicity} using \eqref{eq:ineq-d-values-strict} yields that $w_{(H,\pi,\sigma)}(v_1)\leq w_{(J,\pi|_J,\sigma)}(v_1)$, completing the inductive step in this case.

\item $(H,\pi)\in\cC^{i,j,k}$ and inequality \eqref{eq:ineq-d-values} is strict. Combining \eqref{eq:d-H-ge-d-i} and \eqref{eq:ineq-d-values-strict} yields $d_\theta(J,v_1,w_{(J,\pi|_J,\sigma)})>d_\sigma^i$. Therefore, if $(J,\pi|_J)$ is not already contained in $\cH_\sigma$ when $(H,\pi)$ is added to this set via the set $\cC^{i,j,k}$, then $(J,\pi|_J)$ is contained in $\cC^{i,j,k}$ as well (recall \eqref{eq:J-minus-Hs} and the definition in line~\ref{cw:potential-secondary-threats} and note that $(J,\pi|_J)$ satisfies the first condition in line~\ref{cw:secondary-threat-condition}), and added to $\cH_\sigma$ together with $(H,\pi)$. To complete the inductive step apply again the first part of Lemma~\ref{lemma:danger-weight-monotonicity} using \eqref{eq:ineq-d-values-strict}.

\item $(H,\pi)\in\cC^{i,j}$ and inequality \eqref{eq:ineq-d-values} is tight. Combining \eqref{eq:d-H-eq-d-i} and \eqref{eq:ineq-d-values-tight} yields $d_\theta(J,v_1,w_{(J,\pi|_J,\sigma)})=d_\sigma^i$. Therefore, if $(J,\pi|_J)$ is not already contained in $\cH_\sigma$ when $(H,\pi)$ is added to this set via the set $\cC^{i,j}$, then $(J,\pi|_J)$ is contained in $\cC^{i,j}$ as well (recall \eqref{eq:J-minus-Hs} and the definition in line~\ref{cw:primary-threats}), and added to $\cH_\sigma$ together with $(H,\pi)$. The definitions in line~\ref{cw:set-primary-weight}, \ref{cw:set-weight-index-lower}, \ref{cw:set-weight-index-higher} and \ref{cw:set-secondary-weight} show that in any case
\begin{equation} \label{eq:w-v1-H-J}
  w_{(H,\pi,\sigma)}(v_1)\eqBy{eq:def-w} w_\sigma(H,\pi) =w^i \quad \text{and} \quad
  w_{(J,\pi|_J,\sigma)}(v_1)\eqBy{eq:def-w} w_\sigma(J,\pi|_J)=w^\ibar 
\end{equation}
for some $\ibar\leq i$ with $\alpha_\ibar=\sigma$. Applying the second part of Lemma~\ref{lemma:di-wi-monotonicity} using \eqref{eq:w-v1-H-J} yields that $w_{(H,\pi,\sigma)}(v_1)\leq w_{(J,\pi|_J,\sigma)}(v_1)$, completing the inductive step in this case.

\item $(H,\pi)\in\cC^{i,j,k}$ and inequality \eqref{eq:ineq-d-values} is tight. Combining \eqref{eq:d-H-ge-d-i} and \eqref{eq:ineq-d-values-tight} yields
\begin{equation} \label{eq:d-J-geq-d-sigma-i}
  d_\theta(J,v_1,w_{(J,\pi|_J,\sigma)})\geq d_\sigma^i \enspace.
\end{equation}
Therefore, if $(J,\pi|_J)$ is not already contained in $\cH_\sigma$ when $(H,\pi)$ is added to this set via the set $\cC^{i,j,k}$, then $(J,\pi|_J)$ is contained in $\cT^{i,j,k}$ as well (recall \eqref{eq:J-minus-Hs} and the definition in line~\ref{cw:potential-secondary-threats}). Suppose for the sake of contradiction that $(J,\pi|_J)$ was not transferred from $\cT^{i,j,k}$ to $\cC^{i,j,k}$, i.e., that it violated the condition in line~\ref{cw:secondary-threat-condition}. By~\eqref{eq:d-J-geq-d-sigma-i} and the first condition in line~\ref{cw:secondary-threat-condition} we have
\begin{equation} \label{eq:d-J-d-sigma}
  d_\theta(J,v_1,w_{(J,\pi|_J,\sigma)})=d_\sigma^i\enspace,
\end{equation}
 and by the second condition in line~\ref{cw:secondary-threat-condition} there is a subgraph $\Jbar\seq J$ with $v_1\in\Jbar$ and
\begin{equation} \label{eq:Jbar-Tsigma}
  (\Jbar,\pi|_\Jbar)\in\cC_\sigma(d_\sigma^i)\enspace.
\end{equation}
Combining \eqref{eq:d-H-ge-d-i}, \eqref{eq:ineq-d-values} and \eqref{eq:d-J-d-sigma} shows that
\begin{equation} \label{eq:d-H-sandwiched}
  d_\theta(H,v_1,w_{(H,\pi,\sigma)})=d_\sigma^i \enspace.
\end{equation}
Clearly $\Jbar$ is a subgraph of $H$ that contains the youngest vertex $v_1$, which combined with \eqref{eq:Jbar-Tsigma} and \eqref{eq:d-H-sandwiched} contradicts the fact that $(H,\pi)$ satisfies the condition in line~\ref{cw:secondary-threat-condition} (only graphs satisfying this condition are transferred from $\cT^{i,j,k}$ to $\cC^{i,j,k}$). Hence the graph $(J,\pi|_J)$ is either already contained in $\cH_\sigma$ or added to this set together with $(H,\pi)$ via the set $\cC^{i,j,k}$ at the latest.

It remains to show that $w_{(H,\pi,\sigma)}(v_1)\leq w_{(J,\pi|_J,\sigma)}(v_1)$. Suppose for the sake of contradiction that this inequality is violated. Using \eqref{eq:ineq-d-values-tight} and the second part of Lemma~\ref{lemma:danger-weight-monotonicity} this implies that $w_\sigma(H,\pi)$ is defined in line~\ref{cw:set-secondary-weight} with $\ihat$ defined in line~\ref{cw:set-weight-index-lower}, and that $w_\sigma(J,\pi|_J)$ is defined either in line~\ref{cw:set-primary-weight} or in line~\ref{cw:set-secondary-weight} with $\ihat$ defined in line~\ref{cw:set-weight-index-higher}. In the following we show that none of those cases can occur, as in each case, similarly to above, the existence of a graph $\Jbar\seq H$ with $v_1\in\Jbar$ and $(\Jbar,\pi|_\Jbar)\in\cC_\sigma(d_\theta(H,v_1,w_{(H,\pi,\sigma)}))$ causes $(H,\pi)$ to violate the condition in line~\ref{cw:condition-lower-index} (thus causing a contradiction).
First consider the case that $w_\sigma(J,\pi|_J)$ is defined in line~\ref{cw:set-primary-weight}, i.e., $(J,\pi|_J)$ is an element of one of the sets $\cC^{\ibar,\jbar}$ defined in some iteration $\ibar\leq i$.
From the definition in line~\ref{cw:primary-threats} we know that
\begin{equation} \label{eq:d-J-d-H-d-sigma-ibar} d_\sigma^\ibar=d_\theta(J,v_1,w_{(J,\pi|_J,\sigma)})\eqBy{eq:ineq-d-values-tight} d_\theta(H,v_1,w_{(H,\pi,\sigma)}) \enspace.
\end{equation}
We distinguish the subcases $\jbar=1$ and $\jbar\geq 2$. If $\jbar=1$ then $(J,\pi|_J)$ was added to the set $\cC_\sigma(d_\sigma^\ibar)$ in line~\ref{cw:forward-threats}. Using \eqref{eq:d-J-d-H-d-sigma-ibar} it follows that $(J,\pi|_J)\in\cC_\sigma(d_\theta(H,v_1,w_{(H,\pi,\sigma)}))$, showing that the graph $\Jbar=J$ itself causes $(H,\pi)$ to violate the condition in line~\ref{cw:condition-lower-index}. Similary, if $\jbar\geq 2$, then by the second part of Lemma~\ref{lemma:end-**} there is a subgraph $\Jbar\seq J$ with $v_1\in \Jbar$ and $(\Jbar,\pi|_\Jbar)\in\cC_\sigma(d_\sigma^\ibar)$. Using again \eqref{eq:d-J-d-H-d-sigma-ibar} it follows that $(\Jbar,\pi|_\Jbar)\in\cC_\sigma(d_\theta(H,v_1,w_{(H,\pi,\sigma)}))$, causing a contradiction also in this case.
Now consider the case that $w_\sigma(J,\pi|_J)$ is defined in line~\ref{cw:set-secondary-weight} with $\ihat$ defined in line~\ref{cw:set-weight-index-higher}. Then by the condition in line~\ref{cw:condition-lower-index} there is a subgraph $\Jbar\seq J$ with $v_1\in\Jbar$ and $(\Jbar,\pi|_\Jbar)\in\cC_\sigma(d_\theta(J,v_1,w_{(J,\pi|_J,\sigma)}))$. Using \eqref{eq:ineq-d-values-tight} it follows that $(\Jbar,\pi|_\Jbar)\in\cC_\sigma(d_\theta(H,v_1,w_{(H,\pi,\sigma)}))$, causing $(H,\pi)$ to violate the condition in line~\ref{cw:condition-lower-index}. This completes the proof.
\end{itemize}
\end{proof}

\subsection{Graphs in $\cC^{i,j,k}$ irrelevant for Builder} \label{sec:partner}

When proving Proposition~\ref{prop:Lambda-Builder} in Section~\ref{sec:upper-bound} we develop a Builder strategy along the lines of the algorithm $\CW()$. The following lemma will be crucial in this: It shows that Builder does not need to enforce any ordered graph $(H,\pi)\in\cS(F)$ that is added to one of the families $\cH_s$ via one of the sets $\cC^{i,j,k}$, as for each such graph there is an alternative ordering $\pi'\in\Pi(V(H))$ of its vertices such that $(H,\pi')$ has already been added to $\cH_s$ via the set $\cC^{i,j}$ with weights that are at least as good for Builder.

\begin{lemma}[Partners between $\cC^{i,j}$ and $\cC^{i,j,k}$] \label{lemma:partner}
Let $\sigma\in[r]$ and consider some iteration $i$ of the repeat-loop~(*) with $\alpha_i=\sigma$ of the algorithm $\CW()$. In every iteration $j\geq 1$ of the repeat-loop~(**), the set $\cC^{i,j}$ defined in line~\ref{cw:primary-threats} and the sets $\cC^{i,j,k}$, $k\geq 1$,  defined in line~\ref{cw:secondary-threats-init} and \ref{cw:secondary-threats} during each iteration of the repeat-loop (***), satisfy the following: For any graph $(H,\pi)\in\cC^{i,j,k}$, $\pi=(v_1,\ldots,v_h)$, the graph $(H,\pi')$, defined by $\pi':=(v_{k+1},v_1,v_2,\ldots,v_k,v_{k+2},\ldots,v_h)$, is contained in $\cC^{i,j}$ and satisfies $w_{(H,\pi,\sigma)}(u)\leq w_{(H,\pi',\sigma)}(u)$ for all $u\in H$.
\end{lemma}

\begin{proof}
For the reader's convenience, Figure~\ref{fig:partner} illustrates the notations used throughout the proof.

\begin{figure}
\centering
\PSforPDF{
 \psfrag{tf}{\Large $\cT(F)$}
 \psfrag{ib}{Induction base}
 \psfrag{ih}{Induction hypothesis}
 \psfrag{is}{Induction step}
 \psfrag{k0}{$k=0$}
 \psfrag{k1}{$k=1$}
 \psfrag{k2}{$k=2$}
 \psfrag{tij}{$\cC^{i,j}$}
 \psfrag{tij1}{$\in\cC^{i,j,1}$}
 \psfrag{tij2}{$\in\cC^{i,j,2}$}
 \psfrag{tijkm1}{$\in\cC^{i,j,k-1}$}
 \psfrag{tijk}{$\in\cC^{i,j,k}$}
 \psfrag{v1}{$v_1$}
 \psfrag{v2}{$v_2$}
 \psfrag{vkp1}{$v_{k+1}$}
 \psfrag{wi}{$w^i$}
 \psfrag{wih}{$w^\ibar$}
 \psfrag{pi}{$\pi=(v_1,\ldots,v_h)$}
 \psfrag{pis}{$\pi^*=(v_{k+1},v_2,\ldots,v_k,v_{k+2},\ldots,v_h)$}
 \psfrag{pip}{$\pi'=(v_{k+1},v_1,\ldots,v_k,v_{k+2},\ldots,v_h)$} 
 \psfrag{h}{$(H,\pi)$}
 \psfrag{hp}{$(H,\pi')$}
 \psfrag{hmv1}{$(H\setminus v_1,\pi\setminus v_1)$}
 \psfrag{hmv1p}{$(H\setminus v_1,\pi^*)$}
 \psfrag{hmm}{\parbox{6cm}{$\;\;\;\;(H\setminus\{v_1,v_{k+1}\},\pi^*\setminus v_{k+1})$ \\
                           $=(H\setminus\{v_1,v_{k+1}\},\pi\setminus\{v_1,v_{k+1}\})$ \\
                           $=(H\setminus\{v_1,v_{k+1}\},\pi'\setminus\{v_1,v_{k+1}\})$}}
 \psfrag{hmvkp1}{$(H\setminus v_{k+1},\pi\setminus v_{k+1})$}
 \includegraphics{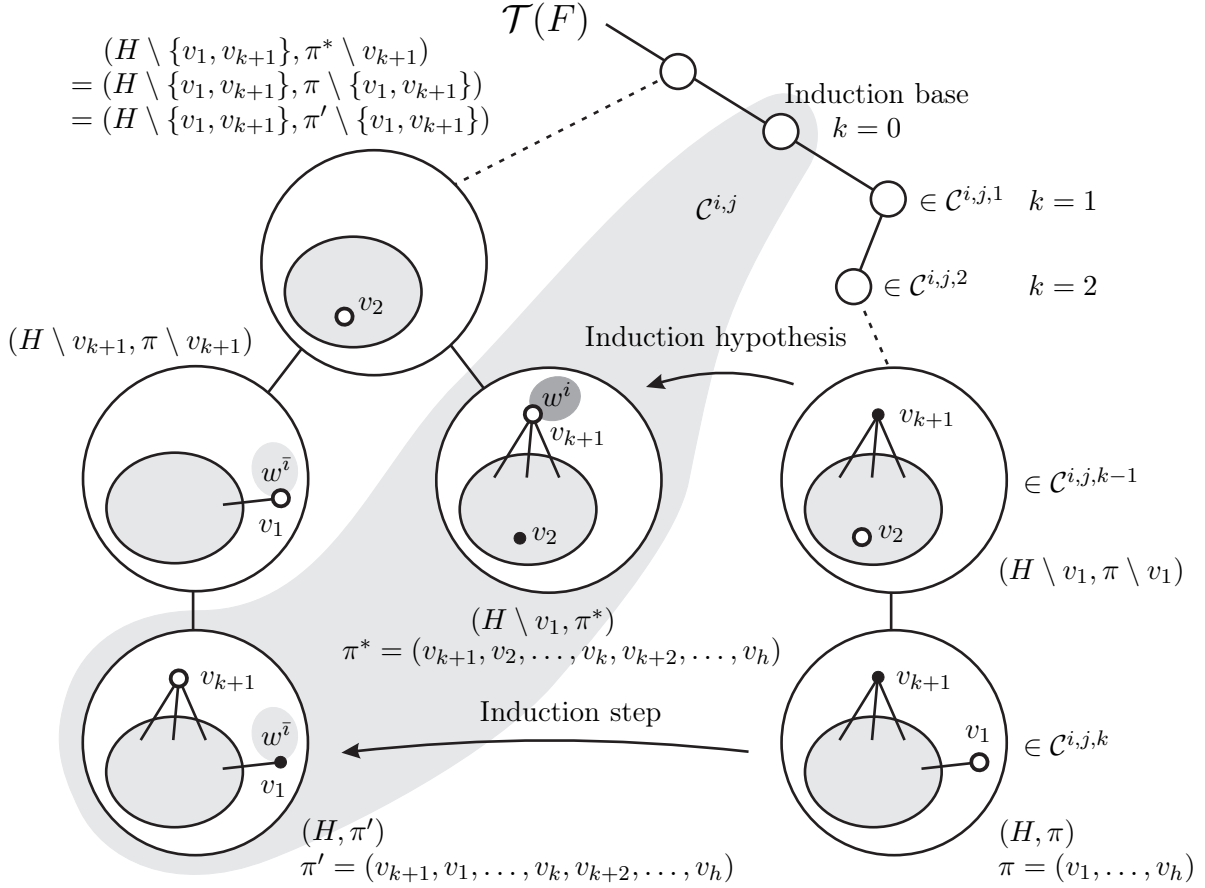}
}
\caption{Notations used in the proof of Lemma~\ref{lemma:partner}. For each ordered graph, the respective youngest vertex is emphasized.} \label{fig:partner}
\end{figure}

We shall prove the following more technical claim: Let $\sigma\in[r]$ and consider some iteration $i$ of the repeat-loop~(*) with $\alpha_i=\sigma$ and some iteration $j$ of the repeat-loop~(**). For $k=0$ and any graph $(H,\pi)\in\cC^{i,j}$, $\pi=(v_1,\ldots,v_h)$, and for $k\geq 1$ and any graph $(H,\pi)\in\cC^{i,j,k}$, $\pi=(v_1,\ldots,v_h)$, the graph $(H,\pi')$, defined by $\pi':=(v_{k+1},v_1,v_2,\ldots,v_k,v_{k+2},\ldots,v_h)$, is contained in $\cC^{i,j}$ and satisfies $w_{(H,\pi,\sigma)}(u)\leq w_{(H,\pi',\sigma)}(u)$ for all $u\in H$.

We will argue at the end of the proof that any graph $(H,\pi)$ contained in one of the sets $\cC^{i,j,k}$ has at least three vertices, ensuring that all subgraphs used in the following arguments have at least one vertex.

To prove the auxiliary claim we consider a fixed iteration $j\geq 1$ of the repeat-loop (**) and argue by induction over $k$,  the number of iterations of the repeat-loop~(***). We choose the state before the beginning of the first iteration $(k=0)$ as our induction base. In this case $\pi'=\pi$ and the claim is trivially true.

For the induction step consider a graph $(H,\pi)\in\cC^{i,j,k}$, $\pi=(v_1,\ldots,v_h)$, that is added to $\cH_\sigma$ in the $k$-th iteration of the repeat-loop~(***) ($k\geq 1$). By the definition of $\cT^{i,j,k}$ in line~\ref{cw:potential-secondary-threats} (recall that $\cC^{i,j,k}\seq \cT^{i,j,k}$) we clearly have
\begin{equation} \label{eq:danger-lower}
  d_\theta(H,v_1,w_{(H,\pi,\sigma)})\geq d_\sigma^i \enspace.
\end{equation}
As $(H,\pi)$ is obtained from $(H\setminus v_1,\pi\setminus v_1)$ by adding $v_1$ as the youngest vertex and all edges incident to it, we have
\begin{equation} \label{eq:H-H-minus-weights}
  w_{(H,\pi,\sigma)}(u)=w_{(H\setminus v_1,\pi\setminus v_1,\sigma)}(u) \quad \text{for all $u\in H\setminus v_1$} \enspace.
\end{equation}
If $j=1$, the definition of $d_\sigma^i$ in line~\ref{cw:least-dangerous-threat} ensures that at the beginning of the $j$-th iteration of the repeat-loop~(**) we have $d_\theta(J,u_1,w_{(J,\tau,\sigma)})\leq d_\sigma^i$ for all $(J,\tau)\in \cC(\cH_\sigma,F)$, $\tau=(u_1,\ldots,u_c)$. If $j>1$, the same statement is true by the first part of Lemma~\ref{lemma:end-**}. Thus if the inequality \eqref{eq:danger-lower} is strict, then $(H,\pi)$ was not in $\cC(\cH_\sigma,F)$ at the beginning of the $j$-th iteration of the repeat-loop~(**), and thus $(H\setminus v_1,\pi\setminus v_1)$ was not in $\cH_\sigma$ at this point.
If $k=1$ this means that $(H\setminus v_1,\pi\setminus v_1)$ must have been added to $\cH_\sigma$ via the set $\cC^{i,j}$. The same conclusion holds if $k=1$ and the inequality \eqref{eq:danger-lower} is tight, as otherwise $(H,\pi)$ would have qualified for inclusion in $\cC^{i,j}$.

Note that the conditions for inclusion into $\cT^{i,j,k}$ and $\cC^{i,j,k}$ in lines \ref{cw:potential-secondary-threats} and~\ref{cw:secondary-threat-condition} do not change during the entire repeat-loop~(***) except for the requirement that $(H,\pi)$ is in the current set $\cC(\cH_\sigma,F)$. Thus if $k\geq 2$ then $(H\setminus v_1,\pi\setminus v_1)$ must have been added to $\cH_\sigma$ via $\cC^{i,j,k-1}$ in the $(k-1)$-th iteration of the repeat-loop~(***).

In all cases we can apply the induction hypothesis and conclude that the graph $(H\setminus v_1,\pi^*)$, defined by $\pi^*:=(v_{k+1},v_2,v_3\ldots,v_k,v_{k+2},\ldots,v_h)$ (to be understood as $\pi\setminus v_1$ if $k=1$) is contained in $\cC^{i,j}$ and satisfies
\begin{equation} \label{eq:ineq-weight-ind}
  w_{(H\setminus v_1,\pi\setminus v_1,\sigma)}(u)\leq w_{(H\setminus v_1,\pi^*,\sigma)}(u) \quad \text{for all $u\in H\setminus v_1$} \enspace.
\end{equation}
By the definition of $\cC^{i,j}$ in line~\ref{cw:primary-threats}, $(H\setminus v_1,\pi^*)$ satisfies
\begin{equation} \label{eq:d-H-pi-star}
  d_\theta(H\setminus v_1,v_{k+1},w_{(H\setminus v_1,\pi^*,\sigma)})=d_\sigma^i\enspace,
\end{equation}
and the weight of its youngest vertex $v_{k+1}$ is set to
\begin{equation} \label{eq:weight-H-pi-star}
  w_{(H\setminus v_1,\pi^*,\sigma)}(v_{k+1})\eqBy{eq:def-w} w_\sigma(H\setminus v_1,\pi^*)=w^i
\end{equation}
in line~\ref{cw:set-primary-weight}.

Consider now the graph $(H\setminus v_{k+1},\pi\setminus v_{k+1})$ and observe that it is a subgraph of $(H,\pi)$. Applying Lemma~\ref{lemma:d-subgraph-monotonicity} and Lemma~\ref{lemma:subgraph-monotonicity} hence yields
\begin{equation} \label{eq:d-H-minus-vs}
  d_\theta(H\setminus v_{k+1},v_1,w_{(H\setminus v_{k+1},\pi\setminus v_{k+1},\sigma)}) \geq d_\theta(H,v_1,w_{(H,\pi,\sigma)}) \geBy{eq:danger-lower} d_\sigma^i \enspace.
\end{equation}
As $(H\setminus v_1,\pi^*)\in\cC^{i,j}$ we must have had
\begin{equation} \label{eq:H-minus-v1-vkp1-Hsigma}
  (H\setminus\{v_1,v_{k+1}\},\pi^*\setminus v_{k+1})=(H\setminus\{v_1,v_{k+1}\},\pi\setminus\{v_1,v_{k+1}\})\in\cH_\sigma
\end{equation}
at the beginning of the $j$-th iteration of the repeat-loop~(**).

As our next intermediate step we will show that the graph $(H\setminus v_{k+1},\pi\setminus v_{k+1})$ (which is obtained from the graph in \eqref{eq:H-minus-v1-vkp1-Hsigma} by adding $v_1$ as the youngest vertex and all edges incident to it) is contained in $\cH_\sigma$ at the beginning of the $j$-th iteration of the repeat-loop~(**) as well. We first consider the case that one of the inequalities in \eqref{eq:d-H-minus-vs} is strict, i.e., we have
\begin{equation} \label{eq:d-H-minus-vkp1-strict}
  d_\theta(H\setminus v_{k+1},v_1,w_{(H\setminus v_{k+1},\pi\setminus v_{k+1},\sigma)}) > d_\sigma^i \enspace.
\end{equation}
If $j=1$, it follows from \eqref{eq:H-minus-v1-vkp1-Hsigma} and \eqref{eq:d-H-minus-vkp1-strict} that $(H\setminus v_{k+1},\pi\setminus v_{k+1})$ is indeed contained in $\cH_\sigma$ at the beginning of the $j$-th iteration of the repeat-loop~(**), as otherwise we would obtain a contradiction to the definition of $d_\sigma^i$ in line~\ref{cw:least-dangerous-threat}. The same conclusion holds for $j>1$ by using \eqref{eq:H-minus-v1-vkp1-Hsigma}, \eqref{eq:d-H-minus-vkp1-strict} and the first part of Lemma~\ref{lemma:end-**}. Now consider the case that all inequalities in \eqref{eq:d-H-minus-vs} are tight, i.e., we have
\begin{equation} \label{eq:d-H-minus-vkp1-tight}
  d_\theta(H\setminus v_{k+1},v_1,w_{(H\setminus v_{k+1},\pi\setminus v_{k+1},\sigma)})
  = d_\theta(H,v_1,w_{(H,\pi,\sigma)}) = d_\sigma^i \enspace.
\end{equation}
Suppose that $j=1$ and that $(H\setminus v_{k+1},\pi\setminus v_{k+1})$ was not already contained in $\cH_\sigma$ at the beginning of the first iteration of the repeat-loop~(**). Then by \eqref{eq:H-minus-v1-vkp1-Hsigma} and \eqref{eq:d-H-minus-vkp1-tight}, this graph would be added to $\cC^{i,1}=\cC_\sigma(d_\sigma^i)$ in line~\ref{cw:primary-threats}. As $(H\setminus v_{k+1},\pi\setminus v_{k+1})$ is a subgraph of $(H,\pi)$ that contains the vertex $v_1$, this observation together with the second equality in \eqref{eq:d-H-minus-vkp1-tight} contradicts the fact $(H,\pi)$ satisfies the condition in line~\ref{cw:secondary-threat-condition}. The remaining subcase $j>1$ can be proven analogously by using the second part of Lemma~\ref{lemma:end-**}.

Therefore, we indeed have
\begin{equation} \label{eq:H-minus-vkps-H-sigma}
  (H\setminus v_{k+1},\pi\setminus v_{k+1})\in\cH_\sigma
\end{equation}
at the beginning of the $j$-th iteration of the repeat-loop~(**). The weight assigned to the youngest vertex $v_1$ of this graph is
\begin{equation} \label{eq:weight-v1-1}
  w_{(H\setminus v_{k+1},\pi\setminus v_{k+1},\sigma)}(v_1) = w^\ibar \geq w^i
\end{equation}
for some $\ibar\leq i$ with $\alpha_\ibar=\sigma$, where the last inequality follows from the second part of Lemma~\ref{lemma:di-wi-monotonicity}.

We will show that the graph $(H,\pi')$, defined by $\pi':=(v_{k+1},v_1,\ldots,v_k,v_{k+2},\ldots,v_h)$ (which is obtained from $(H\setminus v_{k+1},\pi\setminus v_{k+1})$ by adding $v_{k+1}$ as the youngest vertex and all edges incident to it), satisfies the inductive claim: We first demonstrate that this graph is contained in $\cC^{i,j}$ and then prove the claimed inequality between the vertex weights of $(H,\pi)$ and $(H,\pi')$.

As a first step towards this goal we show that $d_\theta(H,v_{k+1},w_{(H,\pi',\sigma)})=d_\sigma^i$.

Clearly we have $(H\setminus\{v_1,v_{k+1}\},\pi^*\setminus v_{k+1})=(H\setminus\{v_1,v_{k+1}\},\pi'\setminus\{v_1,v_{k+1}\})$, implying that
\begin{equation} \label{eq:H-minus-H-weights}
  w_{(H\setminus v_1,\pi^*,\sigma)}(u)=w_{(H,\pi',\sigma)}(u) \quad \text{for all $u\in H\setminus\{v_1,v_{k+1}\}$} \enspace.
\end{equation}
As $(H,\pi')$ is obtained from $(H\setminus v_{k+1},\pi\setminus v_{k+1})$ by adding $v_{k+1}$ as the youngest vertex and all edges incident to it, we have
\begin{equation} \label{eq:weight-v1-2}
  w_{(H,\pi',\sigma)}(v_1)=w_{(H\setminus v_{k+1},\pi\setminus v_{k+1},\sigma)}(v_1) \geBy{eq:weight-v1-1} w^i \enspace.
\end{equation}

Observe that $(H,\pi')$ is a supergraph of $(H\setminus v_1,\pi^*)$, so applying Lemma~\ref{lemma:d-subgraph-monotonicity} and Lemma~\ref{lemma:subgraph-monotonicity} yields
\begin{equation} \label{eq:d-H-vkp1-ineq}
  d_\theta(H,v_{k+1},w_{(H,\pi',\sigma)}) \leq d_\theta(H\setminus v_1,v_{k+1},w_{(H\setminus v_1,\pi^*,\sigma)}) \eqBy{eq:d-H-pi-star} d_\sigma^i \enspace.
\end{equation}
Now suppose for the sake of contradiction that the first inequality in \eqref{eq:d-H-vkp1-ineq} is strict, i.e., we have
\begin{equation} \label{eq:d-H-vkp1-ineq-strict}
  d_\theta(H,v_{k+1},w_{(H,\pi',\sigma)}) < d_\sigma^i \enspace.
\end{equation}
By \eqref{eq:def-d} and \eqref{eq:H-minus-H-weights} this implies that some
\begin{equation} \label{eq:argmin}
  J'\in\argmin_{J\seq H:v_{k+1}\in J}\Big(\sum_{u\in J\setminus v_{k+1}}\big(1+w_{(H,\pi',\sigma)}(u)\big)-e(J)\cdot\theta\Big)
\end{equation}
includes the vertex $v_1$. But then we would have
\begin{align*}
  d_\theta(H,v_1,w_{(H,\pi,\sigma)})
    &\leBy{eq:def-d}
      \sum_{u\in J'\setminus v_1} \big(1+w_{(H,\pi,\sigma)}(u)\big)-e(J')\cdot\theta \\
    &\leByM{\eqref{eq:H-H-minus-weights},\eqref{eq:ineq-weight-ind}}
      \sum_{u\in J'\setminus v_1} \big(1+w_{(H\setminus v_1,\pi^*,\sigma)}(u)\big)-e(J')\cdot\theta \\
    &= \sum_{u\in J'\setminus \{v_1,v_{k+1}\}} \big(1+w_{(H\setminus v_1,\pi^*,\sigma)}(u)\big)+\big(1+\underbrace{w_{(H\setminus v_1,\pi^*,\sigma)}(v_{k+1})}_{\eqBy{eq:weight-H-pi-star} w^i}\big)-e(J')\cdot\theta \\
    &\leByM{\eqref{eq:H-minus-H-weights},\eqref{eq:weight-v1-2}}
      \sum_{u\in J'\setminus v_{k+1}} \big(1+w_{(H,\pi',\sigma)}(u)\big)-e(J')\cdot\theta \\
    &\eqByM{\eqref{eq:def-d},\eqref{eq:argmin}} d_\theta(H,v_{k+1},w_{(H,\pi',\sigma)}) \lBy{eq:d-H-vkp1-ineq-strict} d_\sigma^i \enspace,
\end{align*}
contradicting \eqref{eq:danger-lower}. Hence \eqref{eq:d-H-vkp1-ineq} holds with equality and we have indeed
\begin{equation} \label{eq:d-H-pip}
  d_\theta(H,v_{k+1},w_{(H,\pi',\sigma)})=d_\sigma^i \enspace.
\end{equation}
Clearly, as $(H\setminus v_1,\pi^*)$ is contained in $\cC^{i,j}$, it follows that this graph is not contained in $\cH_\sigma$ at the beginning of the $j$-th iteration of the repeat-loop~(*). As $(H,\pi')$ is a supergraph of $(H\setminus v_1,\pi^*)$, it follows from Lemma~\ref{lemma:subgraph-monotonicity} that $(H,\pi')$ is not contained in $\cH_\sigma$ at this point either. Hence, by \eqref{eq:H-minus-vkps-H-sigma}, \eqref{eq:d-H-pip} and the definition of $\cC^{i,j}$ in line~\ref{cw:primary-threats} we have $(H,\pi')\in\cC^{i,j}$.

It remains to check the claimed inequality between the vertex weights of $(H,\pi)$ and $(H,\pi')$. Note that $(H\setminus v_{k+1},\pi\setminus v_{k+1})$ is a subgraph of $(H,\pi)$. We can hence apply Lemma~\ref{lemma:subgraph-monotonicity} and, observing that $(H,\pi')$ is obtained from $(H\setminus v_{k+1},\pi\setminus v_{k+1})$ by adding $v_{k+1}$ as the youngest vertex and all edges incident to it, obtain the desired inequality for all vertices in the set $\{v_1,\ldots,v_h\}\setminus\{v_{k+1}\}$. For the vertex $v_{k+1}$, note that \eqref{eq:H-H-minus-weights}, \eqref{eq:ineq-weight-ind} and \eqref{eq:weight-H-pi-star} together yield $w_{(H,\pi,\sigma)}(v_{k+1})\leq w^i$ and that $w_{(H,\pi',\sigma)}(v_{k+1})\eqBy{eq:def-w} w_\sigma(H,\pi')=w^i$ by the definition in line~\ref{cw:set-primary-weight}.

This completes the inductive proof of Lemma~\ref{lemma:partner}.

It remains to show that every graph $(H,\pi)\in\cC^{i,j,k}$ has at least three vertices, and that therefore all graphs used in the above arguments are well-defined and have at least one vertex. We show that all graphs $(H,\pi)\in\cS(F)$ on at most two vertices that are ever added to the family $\cH_\sigma$ in the course of the algorithm are added to it via one of the sets $\cC^{i,j}$ defined in line~\ref{cw:primary-threats}: As argued in the proof of Lemma~\ref{lemma:algo-well-defined} on page~\pageref{page:K1}, this is true for the graph $(K_1,(v_1))$ (an isolated vertex), which is added via the set $\cC^{\icheck,1}$ in the first iteration $\icheck$ for which $\alpha_\icheck=\sigma$.
Once we have $\cH_\sigma=\{(K_1,(v_1))\}$, the only two ordered subgraphs of $F$ on two vertices, a single edge and two isolated vertices, are contained in $\cC(\cH_\sigma,F)$. Using this fact together with the observation that $(K_1,(v_1))$ is contained in $\cC_\sigma(0)$ and is a subgraph of both of them, it is easy to check that each of those two graphs can only be added to $\cH_\sigma$ via one of the sets $\cC^{i,j}$: If the $d_\theta()$-value of one of these graphs is equal to 0, then it is added via $\cC^{\icheck,2}$. Otherwise its $d_\theta()$-value is strictly smaller than 0 and it is added via $\cC^{i,1}$ for some $i>\icheck$.
\end{proof}

\subsection{Further properties of the algorithm} \label{sec:further-properties}

The next lemma implies in particular that for a graph $(H,\pi)\in\cH_s$, $\pi=(v_1,\ldots,v_h)$ and a subgraph $J\seq H$, $v_1\in J$, minimizing the right hand side of the definition of $d_\theta(H,v_1,w_{(H,\pi,s)})$ in~\eqref{eq:def-d}, the inequality stated in Lemma~\ref{lemma:subgraph-monotonicity} is in fact an equality. As a consequence, in all situations where the vertex weights of a subgraph $J\seq H$ are relevant, these weights only depend on $(J,\pi|_J)$ and not on the `context' $H$. This is far from clear a priori, and in fact not true for arbitrary subgraphs $J\seq H$.

\begin{lemma}[Irrelevant context of $d_\theta()$-minimizing subgraphs] \label{lemma:irrelevant-context-d}
Throughout the algorithm $\textsc{Compute}\-\textsc{Weights}()$, for every $s\in[r]$, any graph $(H,\pi)\in\cH_s\cup\cC(\cH_s,F)$, $\pi=(v_1,\ldots,v_h)$, and any graph $\Jhat$ from the family
\begin{equation} \label{eq:Jhat-argmin2}
  \argmin_{J\seq H:v_1\in J} \Big(\sum_{u\in J\setminus v_1} \big(1+w_{(H,\pi,s)}(u)\big)-e(J)\cdot\theta\Big)
\end{equation}
we have
\begin{subequations} \label{eq:argmin-weights}
\begin{equation} \label{eq:w-H-w-Jhat-v2-vj}
  w_{(H,\pi,s)}(u)=w_{(\Jhat,\pi|_\Jhat,s)}(u) \quad \text{for all $u\in \Jhat\setminus v_1$} \enspace.
\end{equation}
Moreover, if $(H,\pi)\in\cH_s$, then we have
\begin{equation} \label{eq:w-H-w-Jhat-v1}
  w_{(H,\pi,s)}(v_1)=w_{(\Jhat,\pi|_\Jhat,s)}(v_1) \enspace.
\end{equation}
\end{subequations}
\end{lemma}

Note that by Lemma~\ref{lemma:finite-weights} and Lemma~\ref{lemma:subgraph-monotonicity}, all vertex weights referred to in the formulation of Lemma~\ref{lemma:irrelevant-context-d} are finite values. We will not mention this explicitly again in the following.

The following two auxiliary statements are only used for proving Lemma~\ref{lemma:irrelevant-context-d}.

\begin{lemma} \label{lemma:argmin-intersection}
Let $H$ be a graph with $V(H)=\{v_1,\ldots,v_h\}$, $w:V(H)\setminus\{v_1\}\rightarrow\RR$ an arbitrary weight function, $\theta>0$ a real number, and let $\Jhat$ be a graph from the family
\begin{equation} \label{eq:Jhat-argmin}
  \argmin_{J\seq H:v_1\in J} \Big(\sum_{u\in J\setminus v_1} \big(1+w(u)\big)-e(J)\cdot\theta\Big) \enspace.
\end{equation}
Moreover, let $v_k$ be a vertex contained in $\Jhat$ and $\tJ$ a graph from the family
\begin{equation} \label{eq:tJ-argmin}
  \argmin_{J\seq H\setminus\{v_1,\ldots,v_{k-1}\}:v_k\in J} \Big(\sum_{u\in J\setminus v_k} \big(1+w(u)\big)-e(J)\cdot\theta\Big) \enspace.
\end{equation}
Then the graph $\tJ\cap\Jhat$ is also contained in the family \eqref{eq:tJ-argmin}.
In particular, for two graphs $J',J''$ from the family \eqref{eq:Jhat-argmin} the graph $J'\cap J''$ is also contained in \eqref{eq:Jhat-argmin}.
\end{lemma}

\begin{proof}
In order to simplify notation, we introduce for a real number $\theta>0$, any graph $H$, any vertex $v\in H$ and any weight function $w:V(H)\setminus\{v\}\rightarrow\RR$, the function
\begin{equation} \label{eq:def-lambda-minus}
  \lambda_\theta^-(H,v,w):=\sum_{u\in H\setminus v} \big(1+w(u)\big)-e(H)\cdot\theta \enspace.
\end{equation}
As for the definition of $d_\theta()$ in \eqref{eq:def-d}, it is also convenient here to allow functions $w$ in the third argument of $\lambda_\theta^-()$ whose domain is strictly larger than the set $V(H)\setminus\{v\}$. Of course, for the value of $\lambda_\theta^-(H,v,w)$ only the values $w(u)$ for all $u\in H\setminus v$ are relevant.

By the choice of $\tJ$ in \eqref{eq:tJ-argmin} and by \eqref{eq:def-lambda-minus}, we have
\begin{equation} \label{eq:lm-tJ}
  \lambda_\theta^-(\tJ,v_k,w)\leq \lambda_\theta^-(\tJ\cap\Jhat,v_k,w) \enspace.
\end{equation}
This inequality, however, must be tight, as otherwise the second inequality in
\begin{equation*}
  \lambda_\theta^-(\Jhat\cup\tJ,v_1,w) \eqBy{eq:def-lambda-minus}
  \lambda_\theta^-(\Jhat,v_1,w)+\lambda_\theta^-(\tJ,v_k,w)-\lambda_\theta^-(\tJ\cap\Jhat,v_k,w) \leBy{eq:lm-tJ} \lambda_\theta^-(\Jhat,v_1,w) 
\end{equation*}
would be strict, contradicting the choice of $\Jhat$ in \eqref{eq:Jhat-argmin}. This proves the lemma.
\end{proof}

The next auxiliary statement will be used to prove Lemma~\ref{lemma:irrelevant-context-d} by induction.

\begin{lemma} \label{lemma:argmin-youngest-weights}
The following invariant holds throughout the algorithm $\CW()$. Let $s\in[r]$ and let $(H,\pi)$, $\pi=(v_1,\ldots,v_h)$, be a graph in $\cH_s$, and suppose that every graph $J'$ from the family
\begin{equation} \label{eq:J'-argmin}
  \argmin_{J\seq H:v_1\in J} \Big(\sum_{u\in J\setminus v_1} \big(1+w_{(H,\pi,s)}(u)\big)-e(J)\cdot\theta\Big)
\end{equation}
satisfies
\begin{equation} \label{eq:w-H-w-J'-v2-vj}
  w_{(H,\pi,s)}(u)=w_{(J',\pi|_{J'},s)}(u) \quad \text{for all $u\in J'\setminus v_1$} \enspace.
\end{equation}
Then every such graph $J'$ satisfies
\begin{equation} \label{eq:w-H-w-J'-v1}
  w_{(H,\pi,s)}(v_1)=w_{(J',\pi|_{J'},s)}(v_1) \enspace.
\end{equation}
\end{lemma}

\begin{proof}
Let $J'$ be a graph from the family \eqref{eq:J'-argmin} and note that
\begin{equation} \label{eq:d-H-d-J'}
  d_\theta(H,v_1,w_{(H,\pi,s)})
  \eqByM{\eqref{eq:def-d},\eqref{eq:J'-argmin}} d_\theta(J',v_1,w_{(H,\pi,s)})
  \eqBy{eq:w-H-w-J'-v2-vj} d_\theta(J',v_1,w_{(J',\pi|_{J'},s)}) \enspace.
\end{equation}
By Lemma~\ref{lemma:subgraph-monotonicity} we clearly have $w_{(H,\pi,s)}(v_1)\leq w_{(J',\pi|_{J'},s)}(v_1)$. Consequently, using \eqref{eq:d-H-d-J'} and applying the second part of Lemma~\ref{lemma:danger-weight-monotonicity}, the only way that $w_{(H,\pi,s)}(v_1)$ can be different from  $w_{(J',\pi|_{J'},s)}(v_1)$ is if $w_s(H,\pi)$ is defined either in line~\ref{cw:set-primary-weight} or in line~\ref{cw:set-secondary-weight} with $\ihat$ defined in line~\ref{cw:set-weight-index-higher}, and $w_s(J',\pi|_{J'})$ is defined in line~\ref{cw:set-secondary-weight} with $\ihat$ defined in line~\ref{cw:set-weight-index-lower}. We will show that none of those cases can occur.

\begin{itemize}
\item
First consider the case that $w_s(H,\pi)$ is defined in line~\ref{cw:set-primary-weight} in some iteration $i$ of the repeat-loop~(*) (for which $\alpha_i=s$) and the first iteration $j=1$ of the repeat-loop~(**), i.e., $(H,\pi)$ is contained in $\cC^{i,1}$ and satisfies $d_\theta(H,v_1,w_{(H,\pi,s)})=d_s^i$ and $w_s(H,\pi)=w^i$. Then by \eqref{eq:d-H-d-J'} and by Lemma~\ref{lemma:beginning-loop-*} and Lemma~\ref{lemma:subgraph-monotonicity} the graph $(J',\pi|_{J'})$ must be contained in $\cC^{i,1}$ as well and is added to $\cH_\sigma$ together with $(H,\pi)$. Hence we have $w_s(J',\pi|_{J'})=w^i$ by the definition in line~\ref{cw:set-primary-weight}, proving \eqref{eq:w-H-w-J'-v1} in this case.

\item
Now consider the case that $w_s(H,\pi)$ is defined in line~\ref{cw:set-primary-weight} in some iteration $i$ of the repeat-loop~(*) (for which $\alpha_i=s$) and some iteration $j>1$ of the repeat-loop~(**), i.e., $(H,\pi)$ is contained in $\cC^{i,j}$ and satisfies
\begin{equation} \label{eq:d-H-dsi}
  d_\theta(H,v_1,w_{(H,\pi,s)})=d_s^i\enspace.
\end{equation} 
Then $(H,\pi)$ must have been in $\cC(\cH_s,F)$ at the end of the previous iteration of the repeat-loop (**), and by the second part of Lemma~\ref{lemma:end-**} there is a subgraph $\Jbar\seq H$ with $v_1\in \Jbar$ and $(\Jbar,\pi|_\Jbar)\in\cC_s(d_s^i)$. From the definitions in line~\ref{cw:primary-threats} and line~\ref{cw:forward-threats} it follows that
\begin{equation} \label{eq:d-Jbar-dsi}
  d_\theta(\Jbar,v_1,w_{(\Jbar,\pi|_\Jbar,s)})=d_s^i \enspace.
\end{equation}
Fix some graph $J''$ from the family
\begin{equation} \label{eq:J''-argmin}
  \argmin_{J\seq \Jbar:v_1\in J} \Big(\sum_{u\in J\setminus v_1} \big(1+w_{(\Jbar,\pi|_\Jbar,s)}(u)\big)-e(J)\cdot\theta\Big)
\end{equation}
and note that
\begin{equation} \label{eq:d-Jbar-sum-J''}
  d_\theta(\Jbar,v_1,w_{(\Jbar,\pi|_\Jbar,s)}) \eqByM{\eqref{eq:def-d},\eqref{eq:J''-argmin}}
  \sum_{u\in J''\setminus v_1} \big(1+w_{(\Jbar,\pi|_\Jbar,s)}(u)\big)-e(J'')\cdot\theta \enspace.
\end{equation}
By Lemma~\ref{lemma:subgraph-monotonicity} we have
\begin{equation} \label{eq:w-H-w-Jbar}
  w_{(H,\pi,s)}(u)\leq w_{(\Jbar,\pi|_\Jbar,s)}(u) \quad \text{for all $u\in \Jbar$} \enspace.
\end{equation}
We hence have
\begin{equation*}
  d_\theta(H,v_1,w_{(H,\pi,s)})
  \leBy{eq:def-d} \sum_{u\in J''\setminus v_1} \big(1+w_{(H,\pi,s)}(u)\big)-e(J'')\cdot\theta
  \leByM{\eqref{eq:d-Jbar-sum-J''},\eqref{eq:w-H-w-Jbar}} d_\theta(\Jbar,v_1,w_{(\Jbar,\pi|_\Jbar,s)}) \enspace,
\end{equation*}
which combined with \eqref{eq:d-H-dsi} and \eqref{eq:d-Jbar-dsi} shows that the graph $(J'',\pi|_{J''})$ is contained in the family \eqref{eq:J'-argmin}. Applying Lemma~\ref{lemma:argmin-intersection} yields that the graph $(J'\cap J'',\pi|_{J'\cap J''})$ is also contained in the family \eqref{eq:J'-argmin}. Analogously to \eqref{eq:d-H-d-J'} we have
\begin{equation*}
  d_\theta(H,v_1,w_{(H,\pi,s)})=d_\theta(J'\cap J'',v_1,w_{(J'\cap J'',\pi|_{J'\cap J''},s)}) \enspace,
\end{equation*}
which combined with \eqref{eq:d-H-dsi} shows that
\begin{equation} \label{eq:d-J''-dsi}
  d_\theta(J'\cap J'',v_1,w_{(J'\cap J'',\pi|_{J'\cap J''},s)})=d_s^i \enspace.
\end{equation}
Clearly $J'\cap J''$ is a subgraph of $\Jbar$ that contains the youngest vertex $v_1$. Using this observation and \eqref{eq:d-J''-dsi} and applying Lemma~\ref{lemma:beginning-loop-*} and Lemma~\ref{lemma:subgraph-monotonicity}, the fact that $(\Jbar,\pi|_\Jbar)$ is contained in $\cC_s(d_s^i)$ implies that $(J'\cap J'',\pi|_{J'\cap J''})$ is contained in $\cC_s(d_s^i)$ as well.
But as $J'\cap J''$ is also a subgraph of $J'$ (that contains the youngest vertex $v_1$), this implies with~\eqref{eq:d-H-d-J'} and~\eqref{eq:d-H-dsi} that the graph $(J',\pi|_{J'})$ violates the condition in line~\ref{cw:condition-lower-index}, which yields the desired contradiction.

\item
Finally consider the case that $w_s(H,\pi)$ is defined in line~\ref{cw:set-secondary-weight} (in some iteration $i$ of the repeat-loop~(*)) with $\ihat$ defined in line~\ref{cw:set-weight-index-higher}. By the conditions in line~\ref{cw:secondary-threat-condition} and line~\ref{cw:condition-lower-index} we have $d_\theta(H,v_1,w_{(H,\pi,s)})>d_s^i$ and there is a graph $\Jbar\seq H$ with $v_1\in \Jbar$ and $(\Jbar,\pi|_\Jbar)\in\cC_s(d_\theta(H,v_1,w_{(H,\pi,s)}))$.
Using the definitions in line~\ref{cw:primary-threats} and line~\ref{cw:forward-threats}, as well as the first part of Lemma~\ref{lemma:di-wi-monotonicity}, it follows that
\begin{equation*}
  d_\theta(H,v_1,w_{(H,\pi,s)})=d_s^\ibar
\end{equation*}
and
\begin{equation*}
  d_\theta(\Jbar,v_1,w_{(\Jbar,\pi|_\Jbar,s)})=d_s^\ibar
\end{equation*}
for some $\ibar<i$. From here the proof continues analogously to the second case (where $d_s^i$ needs to be replaced by $d_s^\ibar$), concluding that $(J',\pi|_{J'})$ must violate the condition in line~\ref{cw:condition-lower-index}, which again yields the desired contradiction.
\end{itemize}
\end{proof}

\begin{proof}[Proof of Lemma~\ref{lemma:irrelevant-context-d}]
We argue by induction over the number of vertices of $H$. The claim clearly holds if $H$ consists only of a single vertex, as then $\Jhat=H$ is the only graph contained in the family \eqref{eq:Jhat-argmin2}. This settles the base of the induction.

For the induction step suppose that $H$ has at least two vertices, and let $\Jhat$ be a graph from the family \eqref{eq:Jhat-argmin2}. Clearly, $(H\setminus v_1,\pi\setminus v_1)$ is contained in $\cH_s$ (recall the definition of $\cC()$ in \eqref{eq:def-C}). Applying Lemma~\ref{lemma:subgraph-monotonicity} we obtain that $(\Jhat\setminus v_1,\pi|_{\Jhat\setminus v_1})$ is contained in $\cH_s$ as well and that
\begin{equation} \label{eq:w-H-leq-w-Jhat}
  w_{(H,\pi,s)}(v_i)\leq w_{(\Jhat,\pi|_\Jhat,s)}(v_i) \quad \text{for all $v_i\in \Jhat\setminus v_1$} \enspace.
\end{equation}
We will first show that this inequality is tight for all $v_i\in\Jhat\setminus v_1$ (which is exactly the statement of \eqref{eq:w-H-w-Jhat-v2-vj}).
Suppose for the sake of contradiction that the inequality in \eqref{eq:w-H-leq-w-Jhat} is strict for some $v_i\in\Jhat\setminus v_1$, and choose the largest index $k$ for which this is the case, i.e.\
\begin{equation} \label{eq:w-vk-strict}
  w_{(H,\pi,s)}(v_k)<w_{(\Jhat,\pi|_\Jhat,s)}(v_k)
\end{equation}
and
\begin{equation} \label{eq:older-weights-equal}
  w_{(H,\pi,s)}(v_i)=w_{(\Jhat,\pi|_\Jhat,s)}(v_i) \quad \text{for all $v_i\in \Jhat\setminus\{v_1,\ldots,v_k\}$}
\end{equation}
(we clearly have $k\geq 2$).
Fix a graph $\tJ$ from the family
\begin{equation} \label{eq:tJ-argmin2}
  \argmin_{J\seq H\setminus\{v_1,\ldots,v_{k-1}\}:v_k\in J} \Big(\sum_{u\in J\setminus v_k} \big(1+w_{(H,\pi,s)}(u)\big)-e(J)\cdot\theta\Big)
\end{equation}
and observe that by Lemma~\ref{lemma:argmin-intersection}, also the graph $\tJ\cap\Jhat$ is contained in the family \eqref{eq:tJ-argmin2}. By \eqref{eq:older-weights-equal} the same graph $\tJ\cap\Jhat$ is also contained in the family
\begin{equation*}
  \argmin_{J\seq \Jhat\setminus\{v_1,\ldots,v_{k-1}\}:v_k\in J} \Big(\sum_{u\in J\setminus v_k} \big(1+w_{(\Jhat,\pi|_\Jhat,s)}(u)\big)-e(J)\cdot\theta\Big) \enspace.
\end{equation*}
By induction, we therefore have $w_{(\tJ\cap\Jhat,\pi|_{\tJ\cap\Jhat},s)}(v_k)=w_{(H,\pi,s)}(v_k)$ and $w_{(\tJ\cap\Jhat,\pi|_{\tJ\cap\Jhat},s)}(v_k)=w_{(\Jhat,\pi|_\Jhat,s)}(v_k)$, which together contradicts \eqref{eq:w-vk-strict} and shows that \eqref{eq:w-H-leq-w-Jhat} holds with equality for all $v_i\in \Jhat\setminus v_1$, thus proving \eqref{eq:w-H-w-Jhat-v2-vj}.

The relation \eqref{eq:w-H-w-Jhat-v1} follows from \eqref{eq:w-H-w-Jhat-v2-vj} by applying Lemma~\ref{lemma:argmin-youngest-weights}.
\end{proof}

Lemma~\ref{lemma:irrelevant-context-d} allows us to derive the next statement, which is similar in spirit but considers $\lambda_\theta()$-values instead of $d_\theta()$-values.

\begin{lemma}[Irrelevant context of $\lambda_\theta()$-minimizing subgraphs] \label{lemma:irrelevant-context-lambda}
For every $s\in[r]$ and any graph $(H,\pi)\in\cH_s$, $\pi=(v_1,\ldots,v_h)$, we have
\begin{equation*}
  \min_{J\seq H} \lambda_\theta(J,w_{(H,\pi,s)}) = \min_{J\seq H} \lambda_\theta(J,w_{(J,\pi|_J,s)}) \enspace.
\end{equation*}
\end{lemma}

\begin{proof}
Using the definition of $\lambda_\theta()$ in \eqref{eq:def-lambda}, we obtain from Lemma~\ref{lemma:subgraph-monotonicity} that
\begin{equation} \label{eq:lambdas-w-H-w-J}
  \min_{J\seq H:v_1\in J} \lambda_\theta(J,w_{(H,\pi,s)}) \leq \min_{J\seq H:v_1\in J} \lambda_\theta(J,w_{(J,\pi|_J,s)}) \enspace.
\end{equation}
From
\begin{equation} \label{eq:lambda-J-w-H}
  \min_{J\seq H:v_1\in J} \lambda_\theta(J,w_{(H,\pi,s)}) \eqBy{eq:def-lambda} \min_{J\seq H:v_1\in J} \Big(\sum_{u\in J\setminus v_1} \big(1+w_{(H,\pi,s)}(u)\big)-e(J)\cdot\theta\Big) + 1+w_{(H,\pi,s)}(v_1)
\end{equation}
it follows that the minimum on the left hand side of \eqref{eq:lambdas-w-H-w-J} is attained for some graph
\begin{equation} \label{eq:J-hat-argmin}
  \Jhat\in\argmin_{J\seq H:v_1\in J} \Big(\sum_{u\in J\setminus v_1} \big(1+w_{(H,\pi,s)}(u)\big)-e(J)\cdot\theta\Big) \enspace.
\end{equation}
We can hence apply Lemma~\ref{lemma:irrelevant-context-d} and obtain
\begin{equation*}
  \min_{J\seq H:v_1\in J} \lambda_\theta(J,w_{(H,\pi,s)})
  \eqByM{\eqref{eq:lambda-J-w-H},\eqref{eq:J-hat-argmin}} \sum_{u\in\Jhat} (1+w_{(H,\pi,s)}(u))-e(\Jhat)\cdot\theta
  \eqByM{\eqref{eq:def-lambda},\eqref{eq:argmin-weights}} \lambda_\theta(\Jhat,w_{(\Jhat,\pi|_\Jhat,s)}) \enspace,
\end{equation*}
which shows that the inequality in \eqref{eq:lambdas-w-H-w-J} is tight.
The lemma now follows by combining the resulting identity with the identity
\begin{equation*}
  \min_{J\seq H} \lambda_\theta(J,w_{(H,\pi,s)})=
  \min_{\substack{1\leq i\leq h \\ J\seq H\setminus\{v_1,\ldots,v_{i-1}\}:v_i\in J}} \lambda_\theta(J,w_{(H,\pi,s)}) \enspace.
\end{equation*}
\end{proof}

We are now ready to state and prove the relation between $\Lambda_\theta(F,r)$ as defined in~\eqref{eq:def-Lambda} and the parameter $\beta_i=1+\sum_{s\in[r]}d_s^i$ used in our informal explanation of the algorithm $\CW()$ in Section~\ref{sec:algorithm}.

\begin{lemma}[Relation between $\Lambda_\theta(F,r)$ and $d_s^i$] \label{lemma:Lambda-dsi-sum}
For any input sequence $\alpha\in[r]^{r\cdot|\cS(F)|}$ of the algorithm $\CW()$ we have
\begin{equation} \label{eq:further-properties}
  \max_{\substack{s\in[r] \\ \pi\in\Pi(V(F))}}
  \min_{H\seq F\mathstrut} \lambda_\theta(H,w_{(H,\pi|_H,s)})
  = 1+\sum_{s\in[r]} d_s^\icheck \enspace,
\end{equation}
where $\icheck$ is the smallest integer $i$ for which $(F,\pi)\in\cC^{i,j}$ for some $\pi\in\Pi(V(F))$ and some integer $j\geq 1$, and $d_s^i$ and\/ $\cC^{i,j}$ are defined in line~\ref{cw:least-dangerous-threat} and line~\ref{cw:primary-threats}.
\end{lemma}

\begin{proof}
Throughout the proof, we will repeatedly use that, as a consequence of the first part of Lemma~\ref{lemma:di-wi-monotonicity}, each of the values $d_t^i$, $t\in[r]$, is non-increasing with $i$, and that the sum
\begin{equation} \label{eq:sum-d}
  1+\sum_{t\in[r]}d_t^i
\end{equation}
is decreasing with $i$.

For every $s\in[r]$ and any graph $(H,\pi)\in\cH_s$, $\pi=(v_1,\ldots,v_h)$, we have
\begin{equation} \label{eq:min-lambda-d-w-v1}
  \min_{J\seq H:v_1\in J} \lambda_\theta(J,w_{(H,\pi,s)}) \eqByM{\eqref{eq:def-d},\eqref{eq:def-lambda}}
  d_\theta(H,v_1,w_{(H,\pi,s)})+1+w_{(H,\pi,s)}(v_1) \enspace.
\end{equation}
If $(H,\pi)\in\cC^{i,j}$ for some integers $i,j\geq 1$ with $\alpha_i=s$, then by using the definition in line~\ref{cw:primary-threats} and by combining \eqref{eq:def-w} with the definitions in line~\ref{cw:primary-weight} and \ref{cw:set-primary-weight} we obtain from \eqref{eq:min-lambda-d-w-v1} that
\begin{subequations} \label{eq:min-lambda}
\begin{equation} \label{eq:min-lambda-pt}
  \min_{J\seq H:v_1\in J} \lambda_\theta(J,w_{(H,\pi,s)}) = 1+\sum_{t\in[r]}d_t^i \enspace.
\end{equation}
Similarly, if $(H,\pi)\in\cC^{i,j,k}$ for some integers $i,j,k\geq 1$ with $\alpha_i=s$, then by using the definition in line~\ref{cw:potential-secondary-threats} (recall that $\cC^{i,j,k}\seq\cT^{i,j,k}$) and by combining \eqref{eq:def-w} with the definitions in line~\ref{cw:primary-weight} and \ref{cw:set-secondary-weight} we obtain from \eqref{eq:min-lambda-d-w-v1}, using the monotonicity of the values $d_t^i$ in $i$, that
\begin{equation} \label{eq:min-lambda-st}
  \min_{J\seq H:v_1\in J} \lambda_\theta(J,w_{(H,\pi,s)}) \geq 1+\sum_{t\in[r]}d_t^i \enspace.
\end{equation}
\end{subequations}

By Lemma~\ref{lemma:finite-weights}, in the maximization in \eqref{eq:further-properties} it suffices to consider those $s\in[r]$ and vertex orderings $\pi\in\Pi(V(F))$, $\pi=(v_1,\ldots,v_f)$, for which $(F,\pi)\in\cH_s$. We clearly have
\begin{equation} \label{eq:min-lambda-rewrite}
  \min_{H\seq F} \lambda_\theta(H,w_{(F,\pi,s)}) =
  \min_{\substack{1\leq c\leq f \\ H\seq F\setminus\{v_1,\ldots,v_{c-1}\}:v_c\in H}} \lambda_\theta(H,w_{(F,\pi,s)}) \enspace.
\end{equation}
If $(F,\pi)\in\cC^{i,j}$ for some integers $i,j\geq 1$ with $\alpha_i=s$, then by \eqref{eq:min-lambda} and the monotonicity of the sum \eqref{eq:sum-d} in $i$, the minimum on the right hand side of \eqref{eq:min-lambda-rewrite} is attained for $c=1$, yielding
\begin{equation} \label{eq:min-lambda-sum-d-1}
  \min_{H\seq F} \lambda_\theta(H,w_{(F,\pi,s)}) = 1+\sum_{t\in[r]}d_t^i \enspace.
\end{equation}
If $(F,\pi)\in\cC^{i,j,k}$ for some integers $i,j,k\geq 1$ with $\alpha_i=s$, then by Lemma~\ref{lemma:partner}, the graph $(F,\pi')$, defined by $\pi':=(v_{k+1},v_1,v_2,\ldots,v_k,v_{k+2},\ldots,v_f)$, is contained in $\cC^{i,j}$ and satisfies
\begin{equation} \label{eq:w-F-pi-pi'}
  w_{(F,\pi,s)}(u)\leq w_{(F,\pi',s)}(u) \quad \text{for all $u\in F$} \enspace,
\end{equation}
implying that
\begin{equation} \label{eq:min-lambda-sum-d-2}
  \min_{H\seq F} \lambda_\theta(H,w_{(F,\pi,s)}) \leBy{eq:w-F-pi-pi'}
  \min_{H\seq F} \lambda_\theta(H,w_{(F,\pi',s)}) \eqBy{eq:min-lambda-sum-d-1}
  1+\sum_{t\in[r]}d_t^i \enspace.
\end{equation}

By Lemma~\ref{lemma:irrelevant-context-lambda} we can replace the weight function $w_{(F,\pi,s)}$ on the left hand side of \eqref{eq:min-lambda-sum-d-1} by $w_{(H,\pi|_H,s)}$ and the weight functions $w_{(F,\pi,s)}$ and $w_{(F,\pi',s)}$ in \eqref{eq:min-lambda-sum-d-2} by $w_{(H,\pi|_H,s)}$ and $w_{(H,\pi'|_H,s)}$, respectively. From the two modified equations the claim follows immediately, observing that as a consequence of the monotonicity of the sum \eqref{eq:sum-d} in $i$, their respective right hand sides are maximized for $i=\icheck$ as defined in the lemma.
\end{proof}

\section{Builder in the deterministic game} \label{sec:upper-bound}

In this section we prove Proposition~\ref{prop:Lambda-Builder} by explicitly constructing, for $F$, $r$, $\theta$ and $\beta$ as in the proposition, a Builder strategy that enforces a monochromatic copy of $F$ in the deterministic game with $r$ colors in at most $\amax$ steps (where $\amax=\amax(F,r)$ is defined in~\eqref{eq:amax} below), and that respects the generalized density restriction $(\theta,\beta)$.

\subsection{The pigeonholing}

We will derive Builder's strategy from the algorithm $\CW()$ (Algorithm~\ref{algo:cw} on page~\pageref{algo:cw}) in two steps, using an abstract version of the game as an intermediate step. The reader should not be put off by our introducing yet another game --- this abstract game is merely a convenience to separate a conceptually simple but important pigeonholing argument from the more interesting part of the proof. We give the pigeonholing argument in detail because there are some subtleties involved, and also because we want to derive an \emph{explicit} upper bound $\amax=\amax(F,r)$ (that in particular does not depend on $\theta$) on the number of steps that Builder needs to enforce a copy of $F$ in the original deterministic game.

The abstract game is played by two players AbstractBuilder and AbstractPainter that correspond to the players of the original deterministic game. The state of the abstract game after $t$ steps is described by a list  $(G^1,\dots,G^t)$ of $r$-colored graphs $G^i$, $1\leq i\leq t$, where the same $r$-colored graph may appear several times in the list. (Intuitively, these entries represent $r$-colored graphs of which Builder can enforce isolated copies on the board of the actual deterministic game, where by an isolated copy of some ($r$-colored) graph $G$ on the board we mean a copy of $G$ that is the union of one or several components.) In each step $t+1$ of the abstract game, AbstractBuilder constructs a new graph by choosing an arbitrary subset $\cX$ of the index set $\{1,\ldots,t\}$, and connecting an additional vertex $v$ in an arbitrary way to the disjoint union of the graphs $G^i$, $i\in \cX$. AbstractPainter then chooses a color $s\in[r]$ for $v$, and the resulting $r$-colored graph $G^{t+1}$ is added to the list. The game starts with the empty list (and thus the first graph $G^1$ constructed is simply an isolated vertex), and AbstractBuilder's goal is to create an $r$-colored graph $G^t$ that contains a monochromatic copy of $F$. Similarly to before we say that an AbstractBuilder strategy satisfies the generalized density restriction $(\theta, \beta)$ for given values $\theta>0$ and $\beta$ if, at all times, all subgraphs $H$ with $v(H)\geq 1$ of all graphs $G^i$ in AbstractBuilder's list  satisfy $\mu_\theta(H)\geq \beta$ (recall~\eqref{eq:theta-rho-game}).

The following lemma relates the abstract game to the original deterministic game.

\begin{lemma}[Link between abstract and original deterministic game] \label{lemma:translation}
Let $\ASTRAT$ be an arbitrary AbstractBuilder strategy for the abstract game with $r$ colors. If $\ASTRAT$ enforces a monochromatic copy of $F$ in at most $\tmax$ steps, then it gives rise to a Builder strategy $\STRAT$ for the original deterministic game with $r$ colors that enforces a monochromatic copy of $F$ in at most $(r+1)^\tmax$ steps. Furthermore, if $\ASTRAT$ satisfies the generalized density restriction $(\theta, \beta)$ for given values $\theta>0$ and $\beta\geq 0$, then also $\STRAT$ satisfies the generalized density restriction $(\theta, \beta)$.
\end{lemma}

\begin{proof} We simultaneously capture all possible ways the abstract game may evolve if AbstractBuilder plays according to $\ASTRAT$ by an $r$-ary rooted tree $\cT$ in which a node at depth $t$ is a list $b=(G^1,\ldots,G^t)$ of $r$-colored graphs $G^i$, $1\leq i\leq t$, and has as its $r$ children the nodes $b_s=(G^1,\ldots, G^t,G_s^{t+1})$, $s\in[r]$, where $G^{t+1}_s$ is obtained from $G^1,\ldots,G^t$ by applying the next construction step of $\ASTRAT$ and coloring the new vertex with color $s$. Thus the graphs $G_s^{t+1}$ differ only in the color assigned to the new vertex. 

We assume \wolog that AbstractBuilder stops playing as soon as a monochromatic copy of $F$ is created. Thus if $b=(G^1,\ldots,G^t)$ is a leaf of $\cT$, the graph $G^t$ (the last graph constructed) contains a monochromatic copy of $F$.
Furthermore, by the assumption of the lemma, the depth of the strategy tree $\cT$ is bounded by $\tmax$. In the following we assume \wolog that the depth of $\cT$ is exactly $\tmax$.

To derive $\STRAT$ from $\ASTRAT$, we compute for each node $b=(G^1, \dots, G^t)$ of $\cT$ a function $f_b:\{G^1, \dots, G^t\}\to\NN_0$ that specifies for each of the graphs $G^i$  the number of isolated copies of $G^i$ that are needed to implement the strategy $\ASTRAT$ in the original deterministic game. This can be done recursively as follows.

If $b=(G^1,\ldots,G^t)$ is a leaf of $\cT$, we set
\begin{subequations} \label{eq:multiplicities}
\begin{equation} \label{eq:multiplicities-base}
  f_b(G^t):=1
\end{equation}
and
\begin{equation} \label{eq:multiplicities-base-2}
  f_b(G^i):=0 \enspace, \quad 1\leq i\leq t-1 \enspace.
\end{equation}

If $b=(G^1,\ldots,G^t)$ is an internal node of~$\cT$, then letting~$\cX_b\seq \{1,\ldots,t\}$ denote the index set of the graphs that are used in the construction step corresponding to~$b$, and denoting the descendants of~$b$ by $b_s=(G^1,\ldots,G^t, G^{t+1}_s)$, $s\in[r]$, as before, we define for $1\leq i \leq t$,
\begin{equation} \label{eq:multiplicities-ind}
  f_b(G^i):=\begin{cases} \max_{s\in[r]} f_{b_s}(G^i) \enspace, & \text{if } i\notin \cX_b \enspace, \\
                          \max_{s\in[r]} f_{b_s}(G^i) + \sum_{s\in[r]} f_{b_s}(G_s^{t+1}) \enspace, & \text{if } i\in \cX_b \enspace.
            \end{cases}
\end{equation}
\end{subequations} 

With these definitions, $\STRAT$ is obtained from $\ASTRAT$ by proceeding as described by the strategy tree~$\cT$, and repeating every construction step corresponding to a given node~$b$ exactly $\sum_{s\in[r]} f_{b_s}(G_s^{t+1})$ times, each time connecting a new vertex to (previously unused) isolated copies of the graphs $G^i$, $i\in\cX_b$, on the board as specified by the corresponding step of the abstract game. By the pigeonhole principle, this guarantees that regardless of how Painter plays there is a color $\sigma$ such that at least $f_{b_\sigma}(G_\sigma^{t+1})$ isolated copies of $G_\sigma^{t+1}$ are created, and by our recursive definition in \eqref{eq:multiplicities-ind} it also follows that at least $f_{b_\sigma}(G_\sigma^{i})$ isolated copies of each graph $G_\sigma^{i}$, $1\leq i\leq t$, are left unused. Thus Builder may continue with the construction step corresponding to the node $b_\sigma$. This shows that at every node $b=(G^1, \dots, G^t)$ of $\cT$, Builder has at least $f_{b}(G^{i})$ isolated copies of every graph $G^i$ available. In particular, when he reaches a leaf of $\cT$, due to~\eqref{eq:multiplicities-base} he will have created at least one copy of a graph $G^{t}$ containing a monochromatic copy of $F$.

This shows that $\STRAT$ indeed creates a monochromatic copy of $F$ in the original deterministic game, and it remains to bound the number of steps it needs to do so. For every $t=1,\ldots,\tmax$ we denote by $c_t$ the maximum of $f_b(G^i)$ over all nodes $b=(G^1, \ldots, G^t)$ at depth $t$ in $\cT$ and all $1\leq i\leq t$. It follows from \eqref{eq:multiplicities-ind} that 
\begin{subequations} \label{eq:ct}
\begin{equation}
  c_t\leq (r+1) c_{t+1} \enspace, 
\end{equation}
and by \eqref{eq:multiplicities-base} and  \eqref{eq:multiplicities-base-2} we have
\begin{equation}
  c_\tmax = 1 \enspace.
\end{equation}
\end{subequations}
By definition of the rule how often to repeat each step of $\ASTRAT$ in $\STRAT$, the number of repetitions of a step that corresponds to a node $b$ at depth $t$ in $\cT$ is bounded by $r\cdot c_{t+1}$. It follows that the total number of Builder steps when executing $\STRAT$ is bounded by
\begin{equation*}
  \sum_{t=0}^{\tmax-1} r\cdot c_{t+1}
  \leBy{eq:ct} r \sum_{t=0}^{\tmax-1} (r+1)^t
  \leq (r+1)^\tmax\enspace,
\end{equation*}
as claimed.

Furthermore, as the strategy $\STRAT$ differs from $\ASTRAT$ merely in how often (Abstract)Builder's construction steps are repeated, and because for $\beta\geq 0$ it suffices to check the condition \eqref{eq:theta-rho-game} for all \emph{connected} subgraphs $H$ of the board, it follows that with $\ASTRAT$ also $\STRAT$ satisfies the generalized density restriction $(\theta, \beta)$.
\end{proof}

\subsection{Builder's strategy and proof of Proposition~\texorpdfstring{\ref{prop:Lambda-Builder}}{6}}

We now present AbstractBuilder's strategy $\ABUILD(F,r,\theta)$ that will yield our final Builder strategy $\BUILD(F,r,\theta)$ via Lemma~\ref{lemma:translation}. Throughout this section, $F$, $r$, and $\theta$ are fixed, and we usually omit these arguments when we refer to $\ABUILD(F,r,\theta)$ or $\CW(F,r,\theta,\alpha)$. 

The strategy $\ABUILD()$ proceeds in rounds along the lines of the algorithm $\CW()$. (As before, the term `round' refers to one iteration of the repeat-loop~(*) of $\CW()$.) 
$\ABUILD()$ maintains, for each color $s\in[r]$, a family $\cG_s\seq \cS(F)$ and a mapping~$G_s$ from $\cG_s$ to the $r$-colored graphs in AbstractBuilder's list.
For any $(H,\pi)\in\cG_s$ the graph $G_s(H,\pi)$ will always contain a distinguished monochromatic copy of $H$ in color $s$ to which we will refer as the \emph{central copy of $H$ in $G_s(H,\pi)$}; it is however possible that this copy was constructed in an order different from $\pi$ (this is where we make crucial use of Lemma~\ref{lemma:partner} proved in Section~\ref{sec:partner}).

At the same time, $\ABUILD()$ extracts a sequence $\alpha\in[r]^{r\cdot |\cS(F)|}$ from AbstractPainter's coloring decisions such that the following holds: After each round, the families $\cG_s$ contain all graphs from the families $\cH_s$ occuring after the same number of rounds of $\CW()$ with input sequence $\alpha$. We will also see that for each graph $(H,\pi)\in\cH_s$, the graph $G_s(H,\pi)$ on AbstractBuilder's list can indeed be used in further construction steps (without violating some given generalized density restriction) as indicated by the weight function $w_{(H,\pi,s)}$ computed by $\CW()$ with input sequence $\alpha$.

In order to construct a sequence $\alpha$ for which the above statements hold, $\ABUILD()$ uses variables defined by the algorithm $\CW()$ for several different input sequences. We will use the following notations: For any sequence $\alpha\in[r]^{i-1}$ and any $s\in[r]$ we let $\alpha\circ s\in[r]^{i}$ denote the concatenation of $\alpha$ with~$s$. When we refer to the algorithm $\CW()$ with some input sequence $\alpha\in[r]^{i}$, we tacitly assume that $\alpha$ is extended arbitrarily to a sequence  $\alpha'\in[r]^{r\cdot|\cS(F)|}$ with prefix~$\alpha$. As we will only use this convention when we refer to variables defined in the first $i$ iterations of the repeat-loop~(*) of $\CW()$, the values of $\alpha'$ beyond the prefix $\alpha$ are irrelevant (recall that the $i$-th iteration reads exactly the $i$-th element of the input sequence $\alpha'$).

A key ingredient in the construction of the sequence $\alpha$ is the following lemma. Recall that for any set $X$ and any integer $r\geq 1$, an \emph{$r$-coloring of $X$} is simply a mapping $f:X\rightarrow[r]$.

\begin{lemma}[Dominating color] \label{lemma:cross-product-coloring}
Let $r\geq 1$ be an integer and $X_1,\ldots,X_r$ finite, nonempty sets. For any $r$-coloring $f$ of $X_1\times\cdots\times X_r$ there is a color $\sigma\in[r]$ such that for every $x_\sigma\in X_\sigma$ there are elements $x_s\in X_s$, $s\in[r]\setminus\{\sigma\}$, with $f(x_1,\ldots,x_r)=\sigma$.
\end{lemma}

We defer the proof of Lemma~\ref{lemma:cross-product-coloring} to the next section.

Consider now the pseudocode description of $\ABUILD()$ in Algorithm~\ref{algo:builder}. Note that its loop structure mirrors the structure of $\CW()$, with the crucial difference that while the loop~(**) of $\CW()$ simply focuses on one color $\sigma\in[r]$ (as indicated by the $i$-th entry of the input sequence $\alpha$), for the strategy $\ABUILD()$ the `right' color $\sigma$ depends on the individual decisions of AbstractPainter occuring during the loop~(++), and is therefore not known until this loop terminates.

In the next section we will prove the following two properties of $\ABUILD(F,r,\theta)$.

\begin{lemma}[Well-definedness and duration of AbstractBuilder strategy] \label{lemma:abuild-steps}
For $F$, $r$, and $\theta$ as in Proposition~\ref{prop:Lambda-Builder}, the strategy $\ABUILD(F,r,\theta)$ enforces a monochromatic copy of $F$ in at most $r^2\cdot|\cS(F)|^{r+2}$ steps of the abstract game.
\end{lemma}

\begin{lemma}[AbstractBuilder strategy is legal] \label{lemma:abuild-restriction}
For $F$, $r$, $\theta$, and $\beta$ as in Proposition~\ref{prop:Lambda-Builder}, the strategy $\ABUILD(F,r,\theta)$ satisfies the generalized density restriction $(\theta,\beta)$.
\end{lemma}

Together with Lemma~\ref{lemma:translation}, the preceding statements about $\ABUILD(F,r,\theta)$ imply Proposition~\ref{prop:Lambda-Builder} straightforwardly.

\begin{algorithm} \label{algo:builder}
\DontPrintSemicolon
\LinesNumbered
\renewcommand\theAlgoLine{B\arabic{AlgoLine}}
\SetNlSty{}{}{}
\SetArgSty{}
\caption{AbstractBuilder strategy $\ABUILD(F,r,\theta)$}
\vspace{.2em}
\KwIn{a graph $F$ with at least one edge, an integer $r\geq 2$, a real number $\theta>0$}
\vspace{.2em}
$\alpha:=()$ \;
\ForEach{$s\in[r]$}{
  $\cG_s:=\emptyset$ 
}
$i:=0$ \;
\Repeat({(+)}){$(F,\pi)\in\cG_s$ for some $s\in[r]$ and $\pi\in\Pi(V(F))$ \label{builder:terminate} }{
  $i:=i+1$ \;
  \ForEach{$s\in[r]$}{
    Let $j_{{\max},s}$ denote the total number of iterations of the repeat-loop~(**) in the $i$-th iteration of the repeat-loop~(*) of the algorithm $\CW()$ with input sequence $\alpha\circ s$. \label{builder:jsmax} \;
    For $1\leq j\leq j_{{\max},s}$, let $\cC_s^{i,j}$ and $\cC_s^{i,j,k}$,  $k\geq 1$, denote the sets defined in the algorithm $\CW()$ with input sequence $\alpha\circ s$ in the corresponding iterations of the repeat-loops~(**) and (***) in line~\ref{cw:primary-threats}, or line~\ref{cw:secondary-threats-init} and \ref{cw:secondary-threats}, respectively. \label{builder:C-s} \;    
    $j_s:=1$ \label{builder:js-init} \;
  }
  \Repeat({(++)}){$j_\sigma>j_{{\max},\sigma}$ for some $\sigma\in [r]$ \label{builder:terminate-++} }{
    \ForEach{$((H_1,\pi_1),\ldots,(H_r,\pi_r))\in \cC_1^{i,j_1}\times\cdots\times\cC_r^{i,j_r}$ \label{builder:loop} }{
      For each color $s\in[r]$ let $v_{s1}$ denote the youngest vertex of $(H_s,\pi_s)$. AbstractBuilder constructs a new graph by taking the disjoint union of all graphs $G_s(H_s\setminus v_{s1},\pi_s\setminus v_{s1})$, $s\in[r]$, for which $v(H_s)\geq 2$, and by adding a new vertex $v$ in such a way that for each $s\in[r]$, coloring $v$ in color $s$ will extend the central copy of $H_s\setminus v_{s1}$ in $G_s(H_s\setminus v_{s1},\pi_s\setminus v_{s1})$ to a copy of $H_s$. AbstractPainter chooses a color $\sigmabar\in[r]$ for $v$, the resulting new $r$-colored graph $G$ is added to AbstractBuilder's list, and the newly created copy of $H_\sigmabar$ is designated as the central copy of $G$. \label{builder:connect} \;
    }
    By Lemma~\ref{lemma:cross-product-coloring}, for any combination of colors AbstractPainter chooses in the previous loop (line~\ref{builder:loop} and \ref{builder:connect}), there is a color $\sigmahat\in[r]$ such that for each graph $(H,\pi)\in\cC_\sigmahat^{i,j_\sigmahat}$, in at least one construction step an $r$-colored graph with a central copy of $H$ in color $\sigmahat$ is created. Fix such a color $\sigmahat$. \label{builder:dominant-color} \;
    \ForEach{$(H,\pi)\in\cC_\sigmahat^{i,j_\sigmahat}$}{
      Define $G_\sigmahat{(H,\pi)}$ to be the $r$-colored graph on AbstractBuilder's list resulting from an arbitrary construction step in line~\ref{builder:connect} that created a central copy of $H$ in color~$\sigmahat$. \label{builder:primary-threats} \;
    }
    $\cG_\sigmahat:=\cG_\sigmahat\cup \cC_\sigmahat^{i,j_\sigmahat}$ \label{builder:add-primary-threats} \;
    $k:=0$ \;
    \Repeat({(+++)}){$\cC_\sigmahat^{i,j_\sigmahat,k}=\emptyset$ \label{builder:terminate-+++} }{
      $k:=k+1$ \;
      \ForEach{$(H,\pi)\in\cC_\sigmahat^{i,j_\sigmahat,k}$, $\pi=(v_1,\ldots,v_h)$,}{
        By Lemma~\ref{lemma:partner}, the graph $(H,\pi')$, defined by $\pi':=(v_{k+1},v_1,\ldots,v_k,v_{k+2},\ldots,v_h)$ is contained in $\cC_\sigmahat^{i,j_\sigmahat}$, and consequently  $G_\sigmahat(H,\pi')$ was just defined in line~\ref{builder:primary-threats}. Let $G_\sigmahat(H,\pi):=G_\sigmahat(H,\pi')$. \label{builder:secondary-threats} \;
      }
      $\cG_\sigmahat:=\cG_\sigmahat\cup \cC_\sigmahat^{i,j_\sigmahat,k}$ \label{builder:add-secondary-threats} \;
    } 
    $j_\sigmahat:=j_\sigmahat+1$ \label{builder:js-inc} \;
  }
  $\alpha := \alpha\circ \sigma$ for this $\sigma$ \label{builder:alpha} \;
}
\end{algorithm}

\begin{proof}[Proof of Proposition~\ref{prop:Lambda-Builder}]
Using Lemma~\ref{lemma:abuild-steps} and Lemma~\ref{lemma:abuild-restriction}, we may apply Lemma~\ref{lemma:translation} to $\ABUILD(F,r,\theta)$ to obtain a strategy $\BUILD(F,r,\theta)$ which enforces a monochromatic copy of $F$ in the deterministic game with $r$ colors in at most
\begin{equation} \label{eq:amax}
  \amax=\amax(F,r) := (r+1)^{r^2\cdot|\cS(F)|^{r+2}}
\end{equation}
steps, and satisfies the generalized density restriction $(\theta,\beta)$.
\end{proof}

It remains to prove Lemma~\ref{lemma:cross-product-coloring}, Lemma~\ref{lemma:abuild-steps}, and Lemma~\ref{lemma:abuild-restriction}, which we will do in the next section.

\subsection{Analysis of \texorpdfstring{$\ABUILD()$}{AbstractBuild()}}

\begin{proof} [Proof of Lemma~\ref{lemma:cross-product-coloring}]
We refer to a color $\sigma\in[r]$ that satisfies the conditions of the lemma as a \emph{color that is dominating in $X_1\times\cdots\times X_r$}.

We argue by double induction over $r$ and $\sum_{s\in[r]}|X_s|$. To settle the induction base note that the claim is trivially true for $r=1$ and any finite set $X_1$, and also for $r\geq 2$ and $|X_1|=\cdots=|X_r|=1$. For the induction step let $r\geq 2$ and suppose that one of the sets $X_s$, $s\in[r]$, contains at least two elements. We assume \wolog that it is $X_1$, and fix an element $x\in X_1$.

By induction (over the sum of the cardinalities of the sets $X_s$) we know that for the restriction of $f$ to the set $(X_1\setminus\{x\})\times X_2\times\cdots\times X_r$ there is a dominating color $\sigma\in[r]$. If $\sigma\neq 1$, then $\sigma$ is also dominating in $X_1\times\cdots\times X_r$ and we are done. Otherwise we have $\sigma=1$, i.e.\ for all $x_1\in X_1\setminus\{x\}$ there are elements $x_s\in X_s$, $2\leq s\leq r$, with $f(x_1,\ldots,x_r)=1$. Therefore, if $f$ assigns color $1$ to any of the elements in $\{x\}\times X_2\times\cdots\times X_r$, then $\sigma=1$ is dominating in $X_1\times\cdots\times X_r$ and we are done as well. The only remaining case is that $f(x,\bullet,\cdots,\bullet)$ never uses color $1$, and therefore is an $(r-1)$-coloring of $X_2\times\cdots\times X_r$. By induction (over $r$), there is a color $\sigma'\in[r]\setminus\{1\}$ that is dominating in $X_2\times\cdots\times X_r$, and therefore also in $X_1\times\cdots\times X_r$. This settles the last remaining case.
\end{proof}

In order to prove Lemma~\ref{lemma:abuild-steps} and Lemma~\ref{lemma:abuild-restriction} we will make use of the following technical lemma, which relates the evolution of the families $\cG_s$ occurring in $\ABUILD()$ to the evolution of the families $\cH_s$ occurring in $\CW()$.

\begin{lemma} [Evolution of the families $\cG_s$] \label{lemma:correspondence}
At the end of each iteration of the repeat-loop~(++) during some iteration $i$ of the repeat-loop~(+) in $\ABUILD()$, for each $s\in[r]$ we have $\cG_s\supseteq \cH_s$, where $\cH_s$ denotes the value of $\cH_s$ after $j_s-1$ iterations of the repeat-loop (**) during iteration $i$ of the repeat-loop~(*) in $\CW()$ for the input sequence $\alpha\circ s$, for the current value of $j_s$. Here $\alpha\in[r]^{i-1}$ denotes the sequence that has been constructed in previous rounds of $\ABUILD()$ as a result of AbstractPainter's coloring decisions.
\end{lemma}

\begin{proof}
Consider the $i$-th iteration of the repeat-loop~(+), and let $\sigmahat$ be the color defined in line~\ref{builder:dominant-color} in some iteration of the repeat-loop~(++). Note that the set of graphs $(H,\pi)$ that are added to $\cG_\sigmahat$ in line~\ref{builder:add-primary-threats} and line~\ref{builder:add-secondary-threats} in this iteration is
\begin{equation} \label{eq:cup-T-i-j-a}
  \cH_\sigmahat^{i,j_\sigmahat} := \cC_\sigmahat^{i,j_\sigmahat}\cup\bigcup_{k\geq 1}\cC_\sigmahat^{i,j_\sigmahat,k}
\end{equation}
for the current value of $j_\sigmahat$, where the union in \eqref{eq:cup-T-i-j-a} is over all $k\geq 1$ for which $\cC_\sigmahat^{i,j_\sigmahat,k}$ is defined in line~\ref{builder:C-s}.
By the definitions in line~\ref{builder:C-s} these are exactly the graphs that are added to the family $\cH_\sigmahat$ in the $j_\sigmahat$-th iteration of the repeat-loop~(**) in the $i$-th iteration of the repeat-loop~(*) of the algorithm $\CW()$ for the input sequence $\alpha\circ\sigmahat$.
It follows inductively that at the end of the $i$-th iteration of the repeat-loop~(+), the set of graphs $(H,\pi)$ that have been added to $\cG_\sigma$ for the color $\sigma\in[r]$ satisfying the termination condition in line~\ref{builder:terminate-++} is
\begin{equation*}
  \cH_\sigma^i := \bigcup_{j_\sigma=1}^{j_{{\max},\sigma}} \cH_\sigma^{i,j_\sigma}\enspace. 
\end{equation*}
By the definition of $j_{{\max},s}$ in line~\ref{builder:jsmax}, these are exactly the graphs added to the family $\cH_\sigma$ in the $i$-th iteration of the repeat-loop~(*) of the algorithm $\CW()$ for the input sequence $\alpha\circ\sigma$. As moreover no graphs are added to the families $\cH_s$, $s\in[r]\setminus\{\sigma\}$, during the $i$-th round of $\CW()$ with input sequence $\alpha\circ\sigma$, it follows inductively that throughout, the families $\cG_s$ occuring in $\ABUILD()$ are related to the families $\cH_s$ occurring in $\CW()$ as claimed.
\end{proof}

\begin{proof}[Proof of Lemma~\ref{lemma:abuild-steps}]
We will first argue that $\ABUILD()$ is a well-defined winning strategy.

Note that whenever a graph $(H,\pi)\in\cS(F)$ is added to one of the families $\cG_s$, $s\in[r]$, in line~\ref{builder:add-primary-threats} or \ref{builder:add-secondary-threats}, then $G_s(H,\pi)$ is defined in \ref{builder:primary-threats} or in line~\ref{builder:secondary-threats}, respectively, and in either case this $r$-colored graph was added to AbstractBuilder's list in line~\ref{builder:connect}. Thus throughout the strategy $\ABUILD()$, for every $s\in[r]$ and all graphs $(H,\pi)\in\cG_s$, the graph $G_s(H,\pi)$ is well-defined and exists on AbstractBuilder's list. With the termination condition in line~\ref{builder:terminate} this implies in particular that when $\ABUILD()$ terminates, AbstractBuilder's list indeed contains a graph containing a monochromatic copy of $F$.

Next we show that whenever the construction step in line~\ref{builder:connect} is executed, all involved graphs $(H_s\setminus v_{s1},\pi_s\setminus v_{s1})$ are in the respective families $\cG_s$ at this point, and thus the graphs $G_s(H_s\setminus v_{s1},\pi_s\setminus v_{s1})$ used for the construction step are indeed on AbstractBuilder's list. It follows from the definition of $\cC^{i,j}$ as a subset of $\cC(\cH_s,F)$ (recall~\eqref{eq:def-C}) in line~\ref{cw:primary-threats} of~$\CW()$ that for each $s\in[r]$ and each $(H_s,\pi_s)\in\cC_s^{i,j_s}$ with $v(H_s)\geq 2$ in line~\ref{builder:loop}, the graph $(H_s\setminus v_{s1},\pi_s\setminus v_{s1})$ is contained in $\cH_s$ as defined in the $i$-th iteration of the repeat-loop~(*) at the beginning of the $j_s$-th iteration of the repeat-loop~(**) of $\CW()$ for the input sequence $\alpha\circ s$. Thus by Lemma~\ref{lemma:correspondence}, the graph $(H_s\setminus v_{s1},\pi_s\setminus v_{s1})$ is in the corresponding family $\cG_s$, as claimed.

Also note that whenever the construction step in line~\ref{builder:connect} is executed, the involved entries $G_s(H_s\setminus v_{s1},\pi_s\setminus v_{s1})$, $s\in[r]$, on AbstractBuilder's list are different from each other, as the corresponding central copies of $H_s\setminus v_{s1}$, $s\in[r]$, are all in different colors.

Together the above arguments show that $\ABUILD()$ is indeed a well-defined strategy for the abstract game, and it remains to bound the number of construction steps $\ABUILD()$ needs to enforce a monochromatic copy of $F$. By Lemma~\ref{lemma:correspondence} and the termination condition in line~\ref{builder:terminate}, the number of iterations of the repeat-loop~(+) until $\ABUILD()$ terminates is bounded by the number of iterations of the repeat-loop~(*) in $\CW()$ until the first of the families $\cH_s$, $s\in[r]$, contains the graph $(F,\pi)$ for some vertex-ordering $\pi\in\Pi(V(F))$. The termination condition in line~\ref{cw:terminate} and Lemma~\ref{lemma:algo-well-defined} therefore show that $\ABUILD()$ terminates after at most $r\cdot|\cS(F)|$ iterations of the repeat-loop~(+). In each iteration $i$, the number of iterations of the repeat-loop~(++) is at most $\sum_{s\in[r]}j_{\max,s}\leq r\cdot |\cS(F)|$, as the values $j_{\max,s}$ are bounded by $|\cS(F)|$ as argued in the proof of Lemma~\ref{lemma:algo-well-defined} on page~\pageref{page:j-loop}. Lastly, the number of iterations of the loop in line~\ref{builder:loop} is $|\cC_1^{i,j_1}|\cdots|\cC_r^{i,j_r}|\leq |\cS(F)|^r$, as the sets $\cC_s^{i,j_s}$ are subsets of $\cS(F)$. Multiplying those numbers yields the claimed bound on the total number of construction steps throughout the strategy.
\end{proof}

\begin{proof}[Proof of Lemma~\ref{lemma:abuild-restriction}]
For the reader's convenience, Figure~\ref{fig:ubvars} illustrates the notations used throughout the proof.

\begin{figure}
\centering
\PSforPDF{
 \psfrag{sinr}{$s\in[r]\setminus\{\sigmabar\}$}
 \psfrag{v}{$v$}
 \psfrag{vdots}{$\vdots$}
 \psfrag{Gs}{$G_s(H_s\setminus v_{s1},\pi_s\setminus v_{s1})$}
 \psfrag{Gsigma}{$G_\sigmabar(H_\sigmabar\setminus v_{\sigmabar1},\pi_\sigmabar\setminus v_{\sigmabar1})$}
 \psfrag{G}{\Large $G$}
 \psfrag{Hp}{\Large $H'$}
 \psfrag{Hsp}{$H_s'$}
 \psfrag{Hsigma}{$H_\sigmabar$}
 \psfrag{Jsigma}{$J_\sigmabar$}
 \psfrag{Js}{$J_s$}
 \psfrag{Cs}{$C_s$}
 \includegraphics{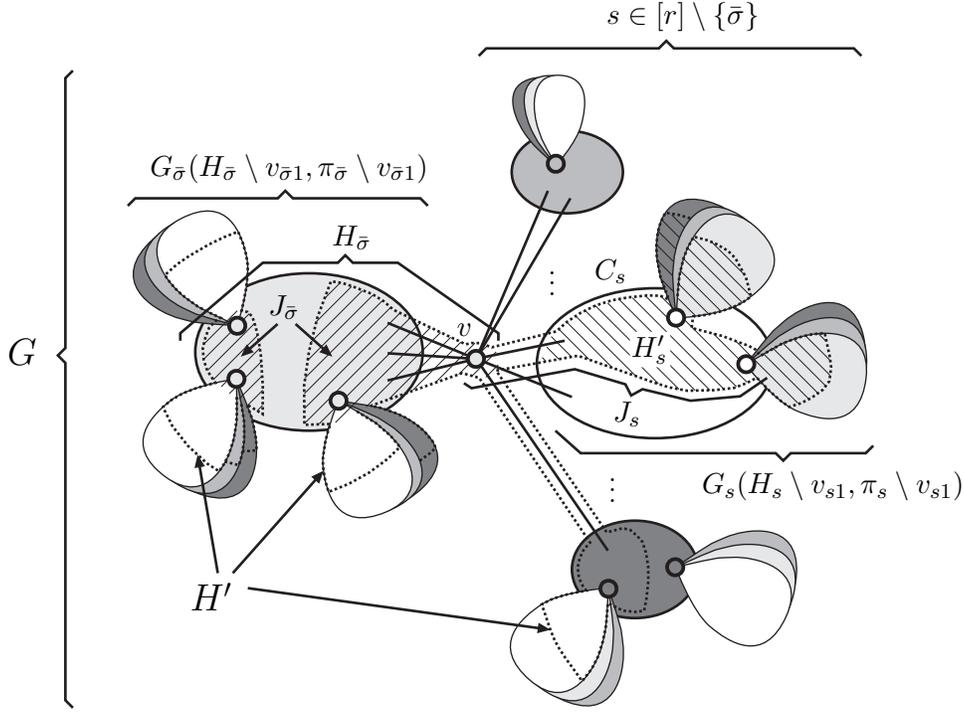}
}
\caption{Notations used in the proof of Lemma~\ref{lemma:abuild-restriction}.} \label{fig:ubvars}
\end{figure}

We prove inductively that for each construction step in line~\ref{builder:connect} in some round $i$ of $\ABUILD()$ the following holds:
Recall the notations from the algorithm, and let $H'$, $v(H')\geq 1$, be a subgraph of the newly constructed graph $G$ such that each component of $H'$ shares at least one vertex with the central copy of $H_\sigmabar$ in $G$. Letting $J_\sigmabar$ denote the intersection of $H'$ with this central copy, we have 
\begin{equation} \label{eq:mu-H'-lambda-J}
  \mu_\theta(H') \geq \lambda_\theta(J_\sigmabar,w_{(H_{\sigmabar},\pi_\sigmabar,\sigmabar),\sigmabar})\enspace,
\end{equation}
where here and throughout we denote for any $s\in[r]$ by $w_{(H,\pi,s),s}$ the weight function $w_{(H,\pi,s)}$ computed by $\CW()$ for the input sequence $\alpha \circ s$. (Recall the definitions in \eqref{eq:def-mu}, \eqref{eq:def-lambda} and \eqref{eq:def-w}, and that during the $i$-th round of $\ABUILD()$, the sequence $\alpha$ has length $i-1$.)

For subgraphs $H'\seq G$ that do not contain the new vertex $v$, the claim follows by induction if $G_\sigmabar(H_\sigmabar\setminus v_{\sigmabar 1},\pi_\sigmabar\setminus v_{\sigmabar 1})$ was defined in line~\ref{builder:primary-threats} (either in the same or in an earlier iteration of the repeat-loop~(+)), or by induction and by Lemma~\ref{lemma:partner} if $G_\sigmabar(H_\sigmabar\setminus v_{\sigmabar 1},\pi_\sigmabar\setminus v_{\sigmabar 1})$ was defined in line~\ref{builder:secondary-threats} (either in the same or in an earlier iteration of the repeat-loop~(+)), recalling the definition~\eqref{eq:def-lambda} and the fact that the functions $w_{(H_\sigmabar\setminus v_{\sigmabar 1},\pi_\sigmabar\setminus v_{\sigmabar 1},\sigmabar),\sigmabar}$ and $w_{(H_\sigmabar,\pi_\sigmabar,\sigmabar),\sigmabar}$ assign the same weight to all vertices of $H_\sigmabar$ different from $v_{\sigmabar 1}$ (recall \eqref{eq:def-w}).

It remains to prove \eqref{eq:mu-H'-lambda-J} for subgraphs $H'$ that contain the vertex $v$. Note that throughout the $i$-th iteration of the repeat-loop~(+), whenever the construction step in line~\ref{builder:connect} is executed, by the definition of the sets $\cC^{i,j_s}_s$ (see line~\ref{builder:C-s} of $\ABUILD()$ and line~\ref{cw:primary-threats} of $\CW()$), all graphs $(H_s,\pi_s)\in\cC^{i,j_s}_s$ used for the construction step satisfy
\begin{equation} \label{eq:d-H-s}
  d_\theta(H_s,v_{s1},w_{(H_s,\pi_s,s),s})=d_s^i \enspace, \quad s\in[r] \enspace,
\end{equation}
where the values $d_s^i$, $s\in[r]$, are defined in line~\ref{cw:least-dangerous-threat} of $\CW()$ with input sequence $\alpha$. Note that these values depend only on the first $i-1$ entries of $\alpha$, i.e., they are the same during the $i$-th iteration of the repeat-loop~(*) of $\CW()$ for each input sequence $\alpha\circ s$, $s\in[r]$.

For the weight assigned to the youngest vertex $v_{\sigmabar 1}$ of $(H_\sigmabar,\pi_\sigmabar)$ by $\CW()$ with input sequence $\alpha\circ\sigmabar$, we thus obtain by combining \eqref{eq:def-w} with the definitions in line~\ref{cw:primary-weight} and \ref{cw:set-primary-weight} that
\begin{equation} \label{eq:weight-v1}
  w_{(H_\sigmabar,\pi_\sigmabar,\sigmabar),\sigmabar}(v_{\sigmabar 1}) = \sum_{s\in[r]\setminus\{\sigmabar\}} d_s^i
 \eqBy{eq:d-H-s} \sum_{s\in[r]\setminus\{\sigmabar\}} d_\theta(H_s,v_{s1},w_{(H_s,\pi_s,s),s}) \enspace.
\end{equation}

For each $s\in[r]$ with $v(H_s)\geq 2$ we define the graph $H_s'$ as the intersection of $H'$ with the copy of $G_s(H_s\setminus v_{s1},\pi_s\setminus v_{s1})$ used for the construction of $G$. Furthermore, for each such $s\in[r]$ we define a subgraph $J_s\seq H_s$ with $v_{s1}\in J_s$ as follows:
Let $C_s$ denote the central copy of $H_s\setminus v_{s1}$ in the copy of $G_s(H_s\setminus v_{s1},\pi_s\setminus v_{s1})$ used for the construction of $G$, and recall that the new vertex $v$ completes $C_s$ to a copy of $H_s$. Let $J_s\seq H_s$ denote the graph that is isomorphic to the intersection of $H'$ with this copy of $H_s$, and note that $H'_s$ intersects $C_s$ in a copy of $J_s\setminus v_{s1}$. For all $s\in[r]$ with $v(H_s)=1$ (i.e., $H_s$ consists only of an isolated vertex) we define $H_s'$ as the null graph (the graph whose vertex set is empty) and set $J_s:=H_s$.
Using these definitions we obtain
\begin{subequations} \label{eq:v-e-H'-2}
\begin{align}
  v(H') &= \sum_{s\in[r]} v(H_s') + 1 \enspace, \\
  e(H') &= \sum_{s\in[r]} \big(e(H_s')+\deg_{J_s}(v_{s1})\big) \enspace.
\end{align}
\end{subequations}
Furthermore, for every $s\in[r]$ we have
\begin{equation} \label{eq:mu-H-s'-lambda-J-s}
  \mu_\theta(H_s') \geq \lambda_\theta(J_s\setminus v_{s1},w_{(H_s,\pi_s,s),s})
\end{equation}
(this holds trivially if $v(H_s)=1$; otherwise, similarly to before, if $G_s(H_s\setminus v_{s1},\pi_s\setminus v_{s1})$ was defined in line~\ref{builder:primary-threats} then this follows by induction, whereas if $G_s(H_s\setminus v_{s1},\pi_s\setminus v_{s1})$ was defined in line~\ref{builder:secondary-threats} then this follows by induction and by Lemma~\ref{lemma:partner}). 

Combining our previous observations we obtain
\begin{align*}
  \mu_\theta(H') &\eqByM{\eqref{eq:def-mu},\eqref{eq:v-e-H'-2}} \sum_{s\in[r]} \big(\mu_\theta(H_s')-\deg_{J_s}(v_{s1})\cdot\theta \big) + 1 \\
                 &\geBy{eq:mu-H-s'-lambda-J-s} \lambda_\theta(J_\sigmabar\setminus v_{\sigmabar 1},w_{(H_\sigmabar,\pi_\sigmabar,\sigmabar),\sigmabar})-\deg_{J_\sigmabar}(v_{\sigmabar 1})\cdot\theta
                 \\ &\phantom{{}\geBy{eq:mu-H-s'-lambda-J-s}{}} {}+
                    \sum_{s\in[r]\setminus\{\sigmabar\}} \underbrace{
                      \big(\lambda_\theta(J_s\setminus v_{s1},w_{(H_s,\pi_s,s),s})-\deg_{J_s}(v_{s1})\cdot\theta\big)
                      }_{\geByM{\eqref{eq:def-d},\eqref{eq:def-lambda}} d_\theta(H_s,v_{s1},w_{(H_s,\pi_s,s),s})} + 1 \\
                 &\geBy{eq:weight-v1} \lambda_\theta(J_\sigmabar\setminus v_{\sigmabar 1},w_{(H_\sigmabar,\pi_\sigmabar,\sigmabar),\sigmabar})-\deg_{J_\sigmabar}(v_{\sigmabar 1})\cdot\theta+ w_{(H_\sigmabar,\pi_\sigmabar,\sigmabar),\sigmabar}(v_{\sigmabar 1}) + 1 \\
                 &\eqBy{eq:lambda-recursive} \lambda_\theta(J_\sigmabar,w_{(H_\sigmabar,\pi_\sigmabar,\sigmabar),\sigmabar}) \enspace,
\end{align*}
completing the inductive proof of \eqref{eq:mu-H'-lambda-J}.

From \eqref{eq:mu-H'-lambda-J} it follows in particular that for every graph $G$ that is added to AbstractBuilder's list during the $i$-th iteration of the repeat-loop~(+), every connected subgraph $H'\seq G$ containing the last added vertex $v$ satisfies
\begin{equation} \label{eq:mu-sum-d}
  \mu_\theta(H')\geq \min_{J\seq H:v_1\in J} \lambda_\theta(J,w_{(H,\pi,s),s})
\end{equation}
for some $s\in[r]$ and some $(H,\pi)$, $\pi=(v_1, \dots, v_h)$, from one of the sets $\cC^{i,j}$ defined in the $i$-th iteration of $\CW()$ when called with input sequence $\alpha\circ s$.

As argued in the proof of Lemma~\ref{lemma:Lambda-dsi-sum} (see \eqref{eq:min-lambda-pt}), the right hand side of \eqref{eq:mu-sum-d} equals $1+\sum_{s\in[r]}d_s^i$ for the values $d_s^i$ defined by $\CW()$ with input sequence $\alpha$.
Regardless of how the sequence $\alpha$ constructed by $\ABUILD()$ evolves in further iterations of the repeat-loop~(+), this quantity is decreasing in $i$ by the first part of Lemma~\ref{lemma:di-wi-monotonicity}.
Moreover, by Lemma~\ref{lemma:correspondence} and the termination condition in line~\ref{builder:terminate}, $\ABUILD()$ terminates after at most $\icheck$ iterations, for $\icheck$ as defined in Lemma~\ref{lemma:Lambda-dsi-sum}.
It follows that all connected subgraphs $H'$ with $v(H')\geq 1$ of all graphs $G$ added to AbstractBuilder's list in the course of $\ABUILD()$ satisfy
\begin{equation} \label{eq:mu-sum-d-2}
 \mu_\theta(H') \geq 1+\sum_{s\in[r]}d_s^\icheck \enspace,
 \end{equation}
where $\icheck$ and $d_s^\icheck$, $s\in[r]$, are defined in $\CW()$ for the input sequence $\alpha$ constructed by $\ABUILD()$.

By Lemma~\ref{lemma:Lambda-dsi-sum} and the definition of $\Lambda_\theta()$ in~\eqref{eq:def-Lambda}, we thus obtain from \eqref{eq:mu-sum-d-2} that
\begin{equation*}
  \mu_\theta(H')\geq \Lambda_\theta(F,r)\geBy{eq:Lambda-geq-beta} \beta
\end{equation*}
for all connected subgraphs $H'$ with $v(H')\geq 1$ of all graphs $G$ appearing on AbstractBuilder's list. Due to the assumption that $\beta\geq 0$, the same statement also follows for all disconnected subgraphs $H'\seq G$ with $v(H')\geq 1$, concluding the proof that the strategy $\ABUILD(F,r,\theta)$ respects the generalized density restriction $(\theta,\beta)$ throughout.
\end{proof}

\section{Painter in the deterministic game} \label{sec:lower-bound}

In this section we prove Proposition~\ref{prop:Lambda-Painter} by explicitly constructing, for $F$, $r$, $\theta$ and $\beta$ as in the proposition, a Painter strategy that avoids creating a monochromatic copy of $F$ in the deterministic game with $r$ colors and generalized density restriction $(\theta,\beta)$.

\subsection{Painter's strategy and proof of Proposition~\texorpdfstring{\ref{prop:Lambda-Painter}}{7}} \label{sec:strategy}

Consider the following Painter strategy, which has four parameters: a graph $F$ with at least one edge, an integer $r\geq 2$, a real number $\theta>0$ and a sequence $\alpha\in[r]^{r\cdot|\cS(F)|}$. The strategy uses the output of Algorithm~\ref{algo:cw}:  $\big((\cH_s,w_s)\big)_{s\in[r]}:=\CW(F,r,\theta,\alpha)$. In each step of the game, Painter picks a color as follows: Let $v$ denote the vertex added in the current step, and for each $s\in[r]$, define
\begin{equation} \label{eq:def-D-s}
  \cD_s:=\left\{(H,\pi)\in\cS(F) \bigmid \parbox{0.47\displaywidth}{assigning color $s$ to $v$ would create a copy of $(H,\pi)$ in color $s$ on the board} \right\} \enspace.
\end{equation}
(Note that this requires Painter to memorize the order in which the vertices on the board arrived.) 
Calculate for each color $s\in[r]$ the value
\begin{equation} \label{eq:def-d-s}
  d(s) := \min_{(H,\pi)\in\cD_s} \lambda_\theta(H,w_{(H,\pi,s)}) \enspace,
\end{equation}
where $\lambda_\theta()$ is defined in \eqref{eq:def-lambda}, and $w_{(H,\pi,s)}()$ is defined in \eqref{eq:def-w} using $\cH_s\seq \cS(F)$ and $w_s:\cH_s\to\RR$ as returned by Algorithm~\ref{algo:cw}. (It is possible that $d(s)=-\infty$ for some colors $s\in[r]$.) Then select $\sigma\in[r]$ as the color for which this value is maximal, and assign color $\sigma$ to the vertex $v$.

Intuitively, the parameter $\lambda_\theta(H,w_{(H,\pi,s)})$ measures the `level of danger' that the Painter strategy encoded by $\alpha$ assigns to copies of the ordered graph $(H,\pi)$ in color $s$, where a graph is considered the more dangerous the smaller its $\lambda_\theta()$-value is. Thus the definition of $d(s)$ in \eqref{eq:def-d-s} corresponds to determining the most dangerous graph in color~$s$ that would be created by assigning color $s$ to $v$, and our strategy selects $\sigma$ as the color for which this most dangerous graph is least dangerous.

If several colors have the same maximal value of $d(s)$, the above rule does not determine a color $\sigma$ uniquely.  Such ties are broken as follows: Consider the families
\begin{equation} \label{eq:tie-break-graphs}
  \cH_s'=\cH_s'(F,r,\theta,\alpha):=\bigcup_{\alpha_i=s\wedge (i=1\vee \alpha_i\neq \alpha_{i-1})} \cC^{i,1} \enspace,
\end{equation}
$s\in{r}$, where $\cC^{i,j}$ are the sets defined in line~\ref{cw:primary-threats} of the algorithm $\CW(F,r,\theta,\alpha)$. Note that these families are fixed throughout Painter's strategy. In Lemma~\ref{lemma:tie-breaking} below we will show that ties can arise only between \emph{two} different colors, and that whenever such a tie arises, then for exactly one of the two colors the set
\begin{equation} \label{eq:def-J-s}
  \cJ_s:= \argmin_{(H,\pi)\in\cD_s} \lambda_\theta(H,w_{(H,\pi,s)})
\end{equation}
contains an ordered graph from the corresponding family $\cH'_s$. Our tie-breaking rule is to then pick the \emph{other} color, i.e., the color $\sigma\in[r]$ for which $\cJ_\sigma$ contains no graph from $\cH'_\sigma$. (Intuitively, Painter considers the ordered graphs in the families $\cH'_s$, $s\in[r]$, as slightly more dangerous than other ordered graphs with the same $\lambda_\theta()$-value.)

In the following we denote the Painter strategy defined above by $\PAINT(F,r,\theta,\alpha)$. Note that this strategy can be employed both in the deterministic two-player game and in the original probabilistic process.

\begin{remark} \label{remark:priority-list-from-strategy}
Note that the actual $\lambda_\theta()$-values of monochromatic ordered subgraphs of $F$ are not relevant in the above strategy --- all that matters is the partial order on the set $\cS(F)\times [r]$ induced by the $\lambda_\theta()$-values and our tie-breaking rule. This partial order can be extended arbitrarily to a total order by defining an arbitrary order among all elements of $\cS(F)\times [r]$ that have the same $\lambda_\theta()$-value and are in one of the sets $\cH'_s$, and among all elements that have the same $\lambda_\theta()$-value and are \emph{not} in one of the sets $\cH'_s$. Thus the strategy $\PAINT(F,r,\theta,\alpha)$ can indeed be represented as a priority list of ordered monochromatic subgraphs of $F$, as described in Section~\ref{sec:algorithms}.
\end{remark}

A careful analysis of the strategy $\PAINT(F,r,\theta,\alpha)$ will eventually yield the following key lemma. As its statement is purely deterministic, it is applicable to both the deterministic game and the probabilistic process. Note that the lemma does not assume any density restrictions for the evolving board.

\begin{lemma}[Witness graph invariant] \label{lemma:witness-graphs} For $F$, $r$, $\theta$, and $\alpha$ as specified in Algorithm~\ref{algo:cw} there is a constant $\vmax=\vmax(F,r,\theta,\alpha)$ such that if Painter plays according to the strategy $\PAINT(F,r,\theta,\alpha)$ then the following invariant is maintained throughout: \\
The board contains a graph $K'$ with $v(K')\leq \vmax$ and
\begin{equation} \label{eq:mu-negative}
  \mu_\theta(K')<0\enspace,
\end{equation}
or for every $s\in[r]$ and every $(H,\pi)\in\cS(F)$ we have that every copy of $(H,\pi)$ in color $s$ on the board is contained in a graph $H'$ with $v(H')\leq \vmax$ and
\begin{equation} \label{eq:mu-H'-lambda-H}
  \mu_\theta(H')\leq \lambda_\theta(H,w_{(H,\pi,s)}) \enspace,
\end{equation}
where $\mu_\theta()$, $\lambda_\theta()$, and $w_{(H,\pi,s)}$ are defined in \eqref{eq:def-mu}, \eqref{eq:def-lambda}, and \eqref{eq:def-w} (using $\cH_s\seq \cS(F)$ and $w_s:\cH_s\to\RR$ as returned by Algorithm~\ref{algo:cw}), respectively.
\end{lemma}

\begin{remark} \label{remark:vmax}
As we shall see shortly, the statement that the size of the graphs $K'$ and $H'$ in Lemma~\ref{lemma:witness-graphs} is bounded by some constant $\vmax=\vmax(F,r,\theta,\alpha)$ is not needed to prove Proposition~\ref{prop:Lambda-Painter}. However, it will be crucial for proving the lower bound part of Theorem~\ref{thm:main-2} in Section \ref{sec:lower-bound-probabilistic} below (recall the remarks in Section~\ref{sec:proof-thm1} and Section~\ref{sec:proof-thm2}). In fact, the proof of the existence of the bound $\vmax$ relies primarily on our tie-breaking rule described above; a version of Lemma~\ref{lemma:witness-graphs} \emph{without} a bound on the size of the graphs $K'$ and $H'$ (which suffices to infer Proposition~\ref{prop:Lambda-Painter}) can also be proven if ties are broken arbitrarily. In this case the proof of Lemma~\ref{lemma:witness-graphs} can be simplified considerably; in particular, Lemma~\ref{lemma:argmins-Hsp}, Lemma~\ref{lemma:forward-walk}, Lemma~\ref{lemma:tie-breaking} and the second part of Lemma~\ref{lemma:sufficient-weight} below are not needed.
The reader might want to skip those parts on his first read-through, or if he is only interested in the deterministic game.
\end{remark}

With Lemma~\ref{lemma:witness-graphs} in hand, the proof of Proposition~\ref{prop:Lambda-Painter} is straightforward.

\begin{proof}[Proof of Proposition~\ref{prop:Lambda-Painter}]
Let $\alpha\in[r]^{r\cdot|\cS(F)|}$ be a sequence for which the minimum in the definition of $\Lambda_\theta(F,r)$ in \eqref{eq:def-Lambda} is attained. By the definition in \eqref{eq:def-Lambda}, for all colors $s\in[r]$ and all vertex orderings $\pi\in\Pi(V(F))$ there is a subgraph $H\seq F$ with
\begin{equation} \label{eq:lambda-leq-beta}
  \lambda_\theta(H,w_{(H,\pi|_H,s)})\leq \Lambda_\theta(F,r) \lBy{eq:Lambda-less-beta}\beta \enspace.
\end{equation}
Suppose now that Painter plays according to the strategy $\PAINT(F,r,\theta,\alpha)$ and that for some $\pi\in\Pi(V(F))$ a copy of $(F,\pi)$ in some color $s\in[r]$ appears on the board. Choose $H\seq F$ such that \eqref{eq:lambda-leq-beta} holds. Then, by Lemma~\ref{lemma:witness-graphs}, the board contains a graph $K'$ with
\begin{equation*}
  \mu_\theta(K') \lBy{eq:mu-negative} 0\leq \beta \enspace,
\end{equation*}
or the copy of $(H,\pi|_H)$ in color $s$ that is contained in the copy of $(F,\pi)$ is contained in a graph $H'$ with
\begin{equation*}
  \mu_\theta(H')\leBy{eq:mu-H'-lambda-H} \lambda_\theta(H,w_{(H,\pi|_H,s)}) \lBy{eq:lambda-leq-beta} \beta \enspace.
\end{equation*}
None of the two cases can occur if Builder adheres to the generalized density restriction $(\theta,\beta)$, and consequently Painter can avoid creating a monochromatic copy of $F$ in the deterministic $F$-avoidance game with $r$ colors and generalized density restriction $(\theta,\beta)$ by playing according to the strategy $\PAINT(F,r,\theta,\alpha)$.
\end{proof}

The rest of this section is devoted to proving Lemma~\ref{lemma:witness-graphs}. To do so we will need a number of technical lemmas. Throughout the following, $F$, $r$, $\theta$ and $\alpha$ are fixed, and we usually omit these arguments when we refer to $\CW(F,r,\theta,\alpha)$ or $\PAINT(F,r,\theta,\alpha)$. We let $\cH_s\seq \cS(F)$ and $w_s:\cH_s\to\RR$ denote the return values of $\CW()$, and $w_{(H,\pi,s)}$ the weight function defined in \eqref{eq:def-w} with respect to these return values.

\subsection{A geometric viewpoint}

We begin by relating the strategy $\PAINT()$ and many of the quantities defined in previous parts of this paper to a simple geometric object. This geometric viewpoint will be a key ingredient in our proof of Lemma~\ref{lemma:witness-graphs}. 

\begin{definition}[Axis-parallel decreasing walk]
We say that $(x_\nu)_{1\leq \nu\leq k}$, $x_\nu\in\RR^r$, is a \emph{decreasing axis-parallel walk in $\RR^r$} if for any two subsequent elements $x_\nu$ and $x_{\nu+1}$ there is a coordinate $s\in[r]$ such that $x_{\nu+1,s}<x_{\nu,s}$ and $x_{\nu+1,t}=x_{\nu,t}$ for all $t\in[r]\setminus\{s\}$.
\end{definition}

The following lemma is an immediate consequence of this definition.

\begin{lemma}[Order on the walk] \label{lemma:walk-order}
Let $(x_\nu)_{1\leq \nu\leq k}$, $x_\nu\in\RR^r$, be a decreasing axis-parallel walk in $\RR^r$. For any two elements $x_\mu$, $x_\nu$ we have $1+\sum_{t\in[r]} x_{\nu,t}\leq 1+\sum_{t\in[r]} x_{\mu,t}$ if and only if\/ $x_{\nu,t}\leq x_{\mu,t}$ for all\/ $t\in[r]$.
\end{lemma}

We can think of the points $(x_\nu)_{1\leq \nu\leq k}$ of a decreasing axis-parallel walk as lying on a sequence of consecutive axis-parallel line segments. For technical reasons we also specify a direction $\sigma\in[r]$ in which, intuitively, the walk continues beyond the point $x_k$. Moreover, we sometimes allow the last segment of this walk to degenerate into the single point $x_k$. This is made precise in the following definition.

\begin{definition}[Extended walk, turning point, starting point/endpoint, segment, order]
Given some decreasing axis-parallel walk $(x_\nu)_{1\leq \nu\leq k}$, $x_\nu\in\RR^r$, and some $\sigma\in[r]$, we say that the pair $\big((x_\nu)_{1\leq \nu\leq k},\sigma\big)$ is an \emph{extended decreasing axis-parallel walk in $\RR^r$}.
We refer to a point $x_\nu$, $2\leq\nu\leq k-1$, on this extended walk as a \emph{turning point} if the coordinate in which $x_\nu$ differs from $x_{\nu-1}$ is different from the coordinate in which $x_\nu$ differs from $x_{\nu+1}$. We also define $x_1$ to be a turning point. Moreover, $x_k$ is a turning point if and only if the coordinate in which $x_k$ differs from $x_{k-1}$ is different from $\sigma$.

Given two consecutive turning points $x_\mu$ and $x_\nu$, $\mu<\nu$, that differ in some coordinate $s\in[r]$, we refer to the line segment that connects them as an \emph{$s$-segment}, and we call $x_\mu$ the \emph{starting point} and $x_\nu$ the \emph{endpoint} of this segment. Furthermore, we refer to the line segment that connects the last turning point $x_\mu$ to the last point $x_k$ of the walk as a $\sigma$-segment (if $\mu=k$, this segment degenerates to a single point $x_k$). We call $x_\mu$ the starting point and $x_k$ the endpoint of this segment.

For two points $x_\mu$ and $x_\nu$, we say that \emph{$x_\mu$ is higher on the walk than $x_\nu$} (or equivalently, $x_\nu$ is lower on the walk than $x_\mu$) if $\mu<\nu$. We extend this notion to segments on the extended walk by saying that an $s$-segment $\Gamma$ is higher on the walk than an $s'$-segment $\Gamma'$ (or equivalently, $\Gamma'$ is lower than $\Gamma$) if the starting point of $\Gamma$ is higher on the walk than the starting point of $\Gamma'$.
\end{definition}

Note that according to Lemma~\ref{lemma:di-wi-monotonicity}, the points $(d_1^i,\ldots,d_r^i)$, $i\geq 1$, form a decreasing axis-parallel walk in~$\RR^r$. In the following we define, for every $s\in[r]$ and every $(H,\pi)\in\cH_s$, a point $x_{(H,\pi,s)}\in\RR^r$ on one of the line segments of this walk.

Let $d_s^i$, $s\in[r]$, denote the values defined in line~\ref{cw:least-dangerous-threat} of the algorithm $\CW()$, and $\cC^{i,j}$ and $\cC^{i,j,k}$ the sets defined in line~\ref{cw:primary-threats}, or line~\ref{cw:secondary-threats-init} and \ref{cw:secondary-threats}, respectively. Recall from Section~\ref{sec:algorithm} that for each $s\in[r]$, the sets $\cC^{i,j}$ and $\cC^{i,j,k}$ for which $\alpha_i=s$ form a partition of the family $\cH_s$.

\begin{definition}[$x$-points] \label{def:x-H-pi-s}
For every $s\in[r]$ and every graph $(H,\pi)\in\cH_s$, $\pi=(v_1,\ldots,v_h)$, we define a point $x_{(H,\pi,s)}\in\RR^r$ as follows:
\begin{subequations} \label{eq:def-x-H-pi-s}
\begin{enumerate}[(a)]
\item 
If $(H,\pi)\in\cC^{i,j}$ for some $i,j\geq 1$ with $\alpha_i=s$, then we define
\begin{equation} \label{eq:vec-pt}
  x_{(H,\pi,s)} := \big(d_1^i,\ldots,d_{s-1}^i,d_\theta(H,v_1,w_{(H,\pi,s)})=d_s^i,d_{s+1}^i,\ldots,d_r^i\big)\in\RR^r \enspace.
\end{equation}

\item
If $(H,\pi)\in\cC^{i,j,k}$ for some $i,j,k\geq 1$ with $\alpha_i=s$, then we define
\begin{equation} \label{eq:vec-st}
  x_{(H,\pi,s)} := \big(d_1^\ihat,\ldots,d_{s-1}^\ihat,d_\theta(H,v_1,w_{(H,\pi,s)}),d_{s+1}^\ihat,\ldots,d_r^\ihat\big)\in\RR^r \enspace,
\end{equation}
where $\ihat$ is the value defined either in line~\ref{cw:set-weight-index-lower} or in line~\ref{cw:set-weight-index-higher} for this $(H,\pi)$.
\end{enumerate}
\end{subequations}
\end{definition}

The fact that $d_\theta(H,v_1,w_{(H,\pi,s)})=d_s^i$ in part~(a) of the previous definition follows directly from the definition in line~\ref{cw:primary-threats} of the algorithm $\CW()$.

\begin{figure}
\centering
\PSforPDF{
 \psfrag{col1}{Color 1}
 \psfrag{col2}{Color 2}
 \psfrag{hsp}{$\cH_s',s\in\{1,2\}$}
 \psfrag{hsmhsp}{$\cH_s\setminus\cH_s',s\in\{1,2\}$}
 \psfrag{h1}{$\cH_1=\cS(F)$}
 \psfrag{h2}{$\cH_2$}
 \psfrag{0}{0}
 \psfrag{alpha2}{$\alpha=(1,1,2,1,2,\ldots,2,1,\ldots)$}
 \psfrag{th2}{$\cC(\cH_2,F)$}
 \psfrag{doo}{$d_1^1$}
 \psfrag{d12}{$d_1^2$}
 \psfrag{d13}{$d_1^3=d_1^4$}
 \psfrag{d15}{$d_1^5$}
 \psfrag{d21}{$d_2^1=d_2^2=d_2^3$}
 \psfrag{d24}{$d_2^4=d_2^5$}
 \psfrag{xhpi1}{$x_{(H,\pi,1)}$}
 \psfrag{dtheta}[c][c][1][90]{$d_\theta(H,v_1,w_{(H,\pi,1)})$}
 \psfrag{walk}{\Large $\cW=\cW(F,r,\theta,\alpha)$}
 \psfrag{ldots}{$\ldots$}
 \psfrag{vdots}{$\vdots$}
 \psfrag{ddots}{\reflectbox{$\ddots$}}
 \includegraphics[width=\textwidth]{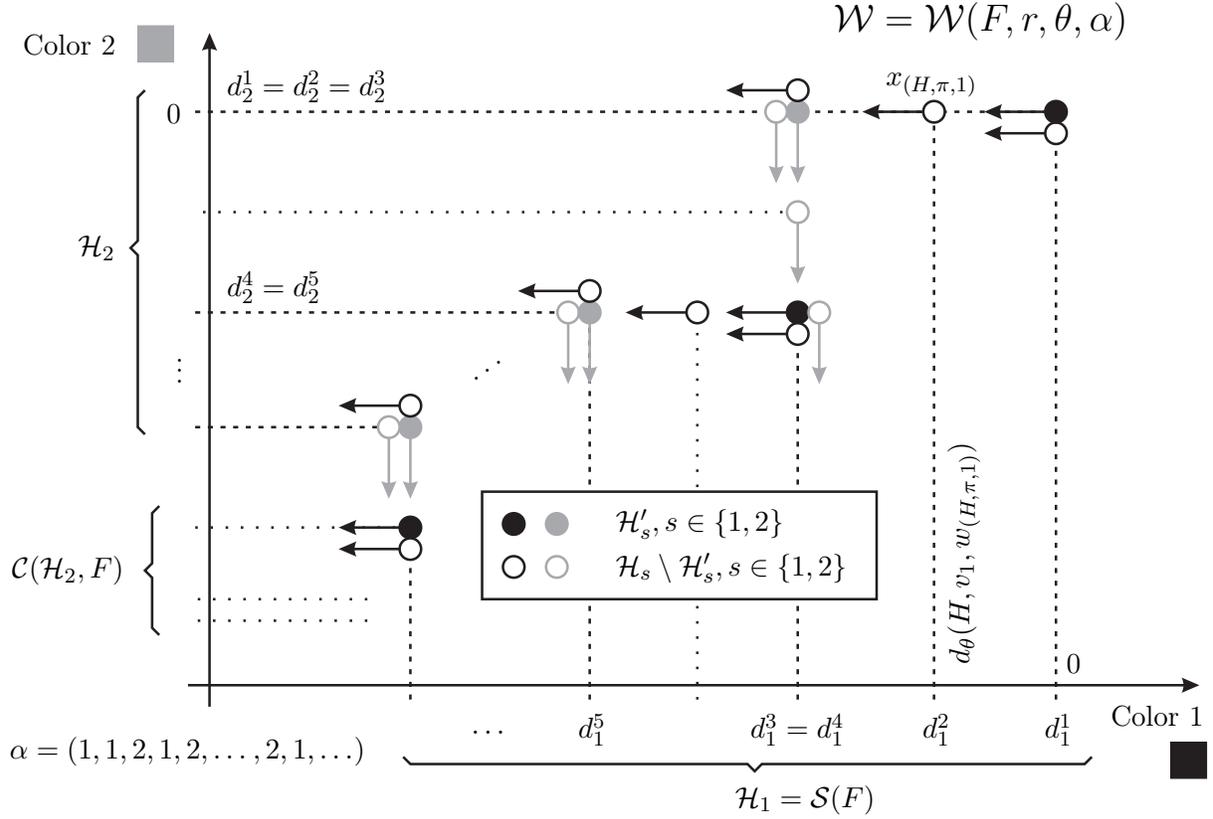}
}
\caption{Illustration of Definition~\ref{def:x-H-pi-s} and Lemma~\ref{lemma:walk}. The figure shows how certain variables of the algorithm $\CW(F,r,\theta,\alpha)$ might evolve for some graph $F$, some value of $\theta$, $r=2$ and $\alpha=(1,1,2,1,2,\ldots,2,1,\ldots)$ (where the algorithm terminates in the round corresponding to the last $1$-entry shown). For any graph $(H,\pi)\in\cH_s$, $s\in\{1,2\}$, a bullet shows the location of the point $x_{(H,\pi,s)}$, and the arrow attached to the bullet points along the $s$-axis. One bullet may represent multiple points at the same location if the corresponding graphs are of the same type (see legend). If graphs with the same associated point are of different type, then the bullets are drawn directly adjacent to each other (instead of on top of each other) to maintain readability.} \label{fig:walk}
\end{figure}

The following lemma states that the points $x_{(H,\pi,s)}$ defined above indeed form an (extended) decreasing axis-parallel walk. We point out that the order in which the points $x_{(H,\pi,s)}$ appear on this walk is not necessarily the order in which the corresponding graphs $(H,\pi)$ are added to the families $\cH_s$ in the course of the algorithm $\CW()$. To be more precise, such a statement is true for the graphs that are added via one of the sets $\cC^{i,j}$, but not for the graphs that are added via one of the sets $\cC^{i,j,k}$.

Of particular importance are the points $x_{(H,\pi,s)}$ for the graphs $(H,\pi)\in\cH_s'$, where $\cH_s'\seq \cH_s$ are the families defined in \eqref{eq:tie-break-graphs} for our tie-breaking rule. As it turns out, those points are always turning points of the walk. Figure~\ref{fig:walk} illustrates Definition~\ref{def:x-H-pi-s} and the different statements of Lemma~\ref{lemma:walk}. 

\begin{lemma}[Walk formed by $x$-points] \label{lemma:walk}
Let $\imax$ denote the total number of iterations of the repeat-loop~(*) of $\CW()$.
The elements in the set $\{x_{(H,\pi,s)}\mid s\in[r] \wedge (H,\pi)\in\cH_s\}$ can be ordered to form a decreasing axis-parallel walk $W=W(F,r,\theta,\alpha)$ in $\RR^r$ such that the extended walk $\cW=\cW(F,r,\theta,\alpha):=(W,\alpha_\imax)$ satisfies the following properties:
\begin{enumerate}[(i)]
\item For any $s\in[r]$ and any graph $(H,\pi)\in\cH_s$, the point $x_{(H,\pi,s)}$ is contained in an $s$-segment.
\item For any $s$-segment $\Gamma$, there is a graph $(H,\pi)\in\cH_s$ such that $x_{(H,\pi,s)}$ is the starting point of $\Gamma$.
\item For any $s$-segment $\Gamma$, if $x_{(H,\pi,s)}$, $\pi=(v_1,\ldots,v_h)$, is the starting point of $\Gamma$, then there is some $J\seq H$ with $v_1\in J$ such that $x_{(J,\pi|_J,s)}=x_{(H,\pi,s)}$ and $(J,\pi|_J)\in\cH_s'$.
\item For any $s$-segment $\Gamma$, if $x_{(H,\pi,s)}\in\Gamma$ is not the starting point of\/ $\Gamma$, then $(H,\pi)\notin\cH_s'$.
\item The lowest segment of $\cW$ is an $\alpha_\imax$-segment and $\cH_{\alpha_\imax}=\cS(F)$.
\end{enumerate}
\end{lemma}

\begin{proof}
By the first part of Lemma~\ref{lemma:di-wi-monotonicity}, the sequence $\Wbar:=(\Wbar^i)_{1\leq i\leq \imax}$, $\Wbar^i:=(d_1^i,\ldots,d_r^i)$, is a decreasing axis-parallel walk with the property that $\Wbar^i$ and $\Wbar^{i+1}$ differ exactly in the coordinate $\alpha_i$.
We first show that the set
\begin{equation} \label{eq:partial-union-x-H-pi-s}
  \{x_{(H,\pi,s)} \mid s\in[r] \; \wedge \; \text{$(H,\pi)\in\cC^{i,j}$ for some $i,j\geq 1$ with $\alpha_i=s$} \}
\end{equation}
coincides exactly with the set of elements of this walk, and that the extended walk $\cWbar:=(\Wbar,\alpha_\imax)$ satisfies the properties of the lemma. In the second part of the proof we argue that for any $s\in[r]$ and any graph $(H,\pi)\in\cC^{i,j,k}$ for some $i,j,k\geq 1$ with $\alpha_i=s$, the point $x_{(H,\pi,s)}$ lies on one of the segments of $\cWbar$, and that the subdivided walk obtained by inserting all those points into the walk $\cWbar$ still satisfies the claimed properties.

First note that on the extended walk $\cWbar$, every point $\Wbar^i$, $1\leq i\leq \imax$, is contained in an $\alpha_i$-segment.

For any $s\in[r]$, any $i,j\geq 1$ with $\alpha_i=s$ and any graph $(H,\pi)\in\cC^{i,j}$, by the definition in \eqref{eq:vec-pt} we have
\begin{equation} \label{eq:x-H-pi-s-d-vec}
  x_{(H,\pi,s)}=(d_1^i,\ldots,d_r^i)=\Wbar^i
\end{equation}
(independently of $j$). Thus property~(i) is satisfied for the elements in the set \eqref{eq:partial-union-x-H-pi-s} and the walk $\cWbar$.

Recall that for each $1\leq i\leq \imax$ the set $\cC^{i,1}$ is nonempty (see the definitions in line~\ref{cw:least-dangerous-threat} and line~\ref{cw:primary-threats}), implying that there is a graph $(H,\pi)\in\cC^{i,1}$ which for $s:=\alpha_i$ satisfies $x_{(H,\pi,s)}=\Wbar^i$, proving in particular property~(ii) for the walk $\cWbar$.

To prove properties~(iii) and (iv), we fix some $s\in[r]$, some $i,j\geq 1$ with $\alpha_i=s$ and some graph $(H,\pi)\in\cC^{i,j}$, $\pi=(v_1,\ldots,v_h)$. Let $\Gamma$ denote the $s$-segment of $\cWbar$ containing the point $x_{(H,\pi,s)}$. We distinguish two cases depending on whether $x_{(H,\pi,s)}$ is the starting point of $\Gamma$ or not. Note that by the first part of Lemma~\ref{lemma:di-wi-monotonicity}, $x_{(H,\pi,s)}=\Wbar^i$ is the starting point of $\Gamma$ if and only if $i=1$ or $\alpha_i\neq \alpha_{i-1}$. 

We first consider the case that $x_{(H,\pi,s)}$ is the starting point of $\Gamma$, i.e., we have
\begin{equation} \label{eq:i-sp-condition}
  i=1 \;\;\text{or}\;\; \alpha_i\neq \alpha_{i-1} \enspace.
\end{equation}
If $j=1$, then by \eqref{eq:tie-break-graphs} and \eqref{eq:i-sp-condition} we have $(H,\pi)\in\cH_s'$. If $j>1$, then by the second part of Lemma~\ref{lemma:end-**} (recall that by the definition in line~\ref{cw:primary-threats} we have $d_\theta(H,v_1,w_{(H,\pi,s)})=d_s^i$)
there is a subgraph $J\seq H$ with $v_1\in J$ for which $(J,\pi|_J)$ is contained in $\cC_s(d_s^i)$. By the definition in line~\ref{cw:forward-threats} we have $\cC_s(d_s^i)=\cC^{i,1}$, implying that $(J,\pi|_J)\in\cC^{i,1}$ and
\begin{equation*}
  x_{(J,\pi|_J,s)}\eqBy{eq:x-H-pi-s-d-vec} x_{(H,\pi,s)} \enspace.
\end{equation*}
Furthermore, using \eqref{eq:i-sp-condition} it follows from the definition in \eqref{eq:tie-break-graphs} that $(J,\pi|_J)\in\cH_s'$, proving that property~(iii) holds for the walk $\cWbar$.

If on the other hand $x_{(H,\pi,s)}$ is not the starting point of $\Gamma$, i.e., $i>1$ and $\alpha_i=\alpha_{i-1}$, then by the definition in \eqref{eq:tie-break-graphs} we have $(H,\pi)\notin\cH_s'$, proving property~(iv) for the walk $\cWbar$.

Note that by the definition of $\cWbar$, the lowest segment of this walk is indeed an $\alpha_\imax$-segment. By the termination condition in line~\ref{cw:terminate} and the observation that during the $i$-th iteration of the repeat-loop~(*), none of the families $\cH_s$, $s\in[r]\setminus\{\alpha_i\}$, is modified, we have $\cH_{\alpha_\imax}=\cS(F)$. Together this proves property~(v) for the walk $\cWbar$.

To complete the proof of the lemma we fix some $s\in[r]$, some $i,j,k\geq 1$ with $\alpha_i=s$ and some graph $(H,\pi)\in\cC^{i,j,k}$, $\pi=(v_1,\ldots,v_h)$, and show that the point $x_{(H,\pi,s)}$ lies on some $s$-segment of the walk $\cWbar$ (possibly in between two points $\Wbar^\ibar$ and $\Wbar^{\ibar+1}$), and that by including all such points $x_{(H,\pi,s)}$ into the walk $\cWbar$ we obtain a subdivided walk $\cW$ that still satisfies the claimed properties (note that beside~(i) we only need to verify that properties~(iii) and (iv) are maintained).

Note that by the definitions in line~\ref{cw:set-weight-index-lower} and line~\ref{cw:set-weight-index-higher} we have $\alpha_\ihat=\alpha_i=s$ for $\ihat$ as in part~(b) of Definition~\ref{def:x-H-pi-s}. Using this relation, the definition in \eqref{eq:vec-st}, and Lemma~\ref{lemma:d-sandwiched} we obtain that $x_{(H,\pi,s)}$ lies on the $s$-segment $\Gamma$ that contains $\Wbar^\ihat$ and $\Wbar^{\ihat+1}$ on the walk $\cWbar$, showing that the walk $\cW$ satisfies property~(i).

By the strict inequality in~\eqref{eq:d-between-leq-l}, $x_{(H,\pi,s)}$ can not be the starting point of $\Gamma$ if $\ihat$ was defined in line~\ref{cw:set-weight-index-lower}. Moreover, by \eqref{eq:d-between-l-leq} the point $x_{(H,\pi,s)}$ is the starting point of $\Gamma$ if and only if
\begin{equation} \label{eq:d-H-d-s-ihat}
  d_\theta(H,v_1,w_{(H,\pi,s)})=d_s^\ihat
\end{equation}
and
\begin{equation} \label{eq:ihat-sp-condition}
  \ihat=1 \;\;\text{or}\;\; \alpha_\ihat\neq \alpha_{\ihat-1} \enspace.
\end{equation}
In this case, by the condition in line~\ref{cw:condition-lower-index} there is a subgraph $J\seq H$ with $v_1\in J$ such that $(J,\pi|_J)$ is contained in $\cC_s(d_\theta(H,v_1,w_{(H,\pi,s)}))$. Using \eqref{eq:d-H-d-s-ihat} and the definition in line~\ref{cw:forward-threats} shows that $(J,\pi|_J)\in\cC^{\ihat,1}$, implying that
\begin{equation*}
  x_{(J,\pi|_J,s)}\eqBy{eq:x-H-pi-s-d-vec} (d_1^\ihat,\ldots,d_r^\ihat)
  \eqByM{\eqref{eq:vec-st},\eqref{eq:d-H-d-s-ihat}} x_{(H,\pi,s)} \enspace.
\end{equation*}
Furthermore, using \eqref{eq:ihat-sp-condition} it follows from the definition in \eqref{eq:tie-break-graphs} that $(J,\pi|_J)\in\cH_s'$, proving that property~(iii) holds for the walk $\cW$.

As none of the graphs in the sets $\cC^{i,j,k}$ with $\alpha_i=s$ is contained in $\cH_s'$ (recall~\eqref{eq:tie-break-graphs}), the walk $\cW$ trivially satisfies property~(iv). This completes the proof.
\end{proof}

\subsection{Relation of the walk to other quantities}
In the following lemmas we establish several relations between the walk $\cW$ defined in Lemma~\ref{lemma:walk}, the parameters $d_\theta()$ and $w_{(H,\pi,s)}$ used in the algorithm $\CW()$, and the parameter $\lambda_\theta()$ and the families $\cH_s'$ used in the definition of the strategy $\PAINT()$. We will see that for the ordered monochromatic subgraphs of $F$ that are relevant for the strategy $\PAINT()$, the order of the corresponding $x$-points on the walk $\cW$ coincides with the ordering given by the $\lambda_\theta()$-values --- the lower on the walk the point $x_{(H,\pi,s)}$ appears, the lower the value $\lambda(H,w_{(H,\pi,s)})$, i.e., the more dangerous a copy of $(H,\pi)$ in color $s$ is considered (see Lemma \ref{lemma:argmin-lambda-x-vec} below).

\begin{lemma}[$d_\theta()$-value and weight from $x$-point] \label{lemma:w-sum-x}
For any $s\in[r]$ and any graph $(H,\pi)\in\cH_s$, $\pi=(v_1,\ldots,v_h)$, we have
\begin{equation*}
  d_\theta(H,v_1,w_{(H,\pi,s)}) = x_{(H,\pi,s),s}
\end{equation*}
and
\begin{equation} \label{eq:w-sum-x}
  w_{(H,\pi,s)}(v_1)=\sum_{t\in[r]\setminus\{s\}} x_{(H,\pi,s),t} \enspace.
\end{equation}
\end{lemma}

\begin{proof}
The first part of the lemma is an immediate consequence of the definition in \eqref{eq:def-x-H-pi-s}.

For any $s\in[r]$ and any graph $(H,\pi)\in\cH_s$ as in part~(a) of Definition~\ref{def:x-H-pi-s} we obtain, using the definitions in line~\ref{cw:primary-weight} and line~\ref{cw:set-primary-weight},
\begin{equation} \label{eq:primary-weight}
  w_{(H,\pi,s)}(v_1) \eqBy{eq:def-w} w_s(H,\pi)=w^i= \sum_{t\in[r]\setminus\{s\}} d_t^i \eqBy{eq:vec-pt} \sum_{t\in[r]\setminus\{s\}} x_{(H,\pi,s),t} \enspace.
\end{equation}

For any $s\in[r]$ and any graph $(H,\pi)\in\cH_s$ as in part~(b) of Definition~\ref{def:x-H-pi-s} we obtain, using the definitions in line~\ref{cw:primary-weight} and line~\ref{cw:set-secondary-weight},
\begin{equation} \label{eq:secondary-weight}
  w_{(H,\pi,s)}(v_1) \eqBy{eq:def-w} w_s(H,\pi)=w^\ihat=\sum_{t\in[r]\setminus\{s\}} d_t^\ihat \eqBy{eq:vec-st} \sum_{t\in[r]\setminus\{s\}} x_{(H,\pi,s),t} \enspace.
\end{equation}

Together \eqref{eq:primary-weight} and \eqref{eq:secondary-weight} prove the second part of the lemma.
\end{proof}

For the next lemma, recall the definition of $\cC(\cH,F)$ in~\eqref{eq:def-C}.

\begin{lemma}[Graphs in $\cC(\cH_s,F)$ have smallest $d_\theta()$-value] \label{lemma:d-CHs-dHs}
Let $\sigma\in[r]$ and $(H,\pi)\in\cH_\sigma$, $\pi=(v_1,\ldots,v_h)$, and let $s\in[r]$ and $(J,\tau)\in\cC(\cH_s,F)$, $\tau=(u_1,\ldots,u_c)$. \\
If $s=\sigma$, then we have
\begin{align*}
  d_\theta(J,u_1,w_{(J,\tau,\sigma)}) &< x_{(H,\pi,\sigma),\sigma} = d_\theta(H,v_1,w_{(H,\pi,\sigma)}) \enspace.
\intertext{If $s\neq \sigma$, then we have}
  d_\theta(J,u_1,w_{(J,\tau,s)}) &\leq x_{(H,\pi,\sigma),s} \enspace.
\end{align*}
\end{lemma}

\begin{proof}
Let $\imax$ denote the total number of iterations of the repeat-loop~(*) in $\CW()$.

First suppose that $s=\sigma$. Denoting by $\ihat$ the largest index $\ibar\leq \imax$ for which $\alpha_\ibar=\sigma$, the first part of Lemma~\ref{lemma:di-wi-monotonicity} and the definitions in line~\ref{cw:primary-threats} and line~\ref{cw:potential-secondary-threats} show that $d_\theta(H,v_1,w_{(H,\pi,\sigma)})\geq d_\sigma^\ihat$. By the termination condition in line~\ref{cw:this-round-complete} we also have $d_\theta(J,u_1,w_{(J,\tau,\sigma)})<d_\sigma^\ihat$. We thus obtain $d_\theta(J,u_1,w_{(J,\tau,\sigma)})<d_\theta(H,v_1,w_{(H,\pi,\sigma)})$. By the definition in \eqref{eq:def-x-H-pi-s} the right hand side of this last inequality equals $x_{(H,\pi,\sigma),\sigma}$, proving the first part of the lemma.

Now suppose that $s\neq \sigma$. By the definition in \eqref{eq:def-x-H-pi-s} we have
\begin{equation} \label{eq:x-H-pi-sigma}
  x_{(H,\pi,\sigma),s}=d_s^i
\end{equation}
for some $1\leq i\leq \imax$.
By the termination condition in line~\ref{cw:terminate} and the observation that during the $\ibar$-th iteration of the repeat-loop~(*), none of the families $\cH_t$, $t\in[r]\setminus\{\alpha_\ibar\}$, is modified, we must have $\alpha_\imax\neq s$, as we would have $\cH_s=\cS(F)$ otherwise, implying that $\cC(\cH_s,F)$ would be empty.
So let $\ihat$ be the largest index $\ibar\leq \imax-1$ for which $\alpha_\ibar=s$.
By the definition in line~\ref{cw:least-dangerous-threat} we have
\begin{equation} \label{eq:d-J-d-s-ihatp1}
  d_\theta(J,u_1,w_{(J,\tau,s)})\leq d_s^{\ihat+1} \enspace.
\end{equation}
Using the first part of Lemma~\ref{lemma:di-wi-monotonicity} twice we obtain
\begin{equation*}
  d_s^{\ihat+1}=\cdots=d_s^\imax \quad \text{and} \quad d_s^\imax \leq d_s^i \enspace,
\end{equation*}
which together with \eqref{eq:x-H-pi-sigma} and \eqref{eq:d-J-d-s-ihatp1} yields the second part of the lemma.
\end{proof}

\begin{lemma}[Relation between $\lambda_\theta()$-value and $x$-point] \label{lemma:lambda-x-vec}
Let $s\in[r]$. For any graph $(H,\pi)\in\cH_s$, $\pi=(v_1,\ldots,v_h)$, we have
\begin{equation*}
  \lambda_\theta(H,w_{(H,\pi,s)}) \geq 1+\sum_{t\in[r]} x_{(H,\pi,s),t} \enspace.
\end{equation*}
Moreover, there is a subgraph $\Jhat\seq H$ with $v_1\in\Jhat$ satisfying
\begin{equation*}
  \lambda_\theta(\Jhat,w_{(\Jhat,\pi|_\Jhat,s)}) = 1+\sum_{t\in[r]} x_{(\Jhat,\pi|_\Jhat,s),t} = 1+\sum_{t\in[r]} x_{(H,\pi,s),t} \enspace.
\end{equation*}
\end{lemma}

\begin{proof}
We clearly have
\begin{equation} \label{eq:lambda-w-d-H}
  \lambda_\theta(H,w_{(H,\pi,s)})
  \eqBy{eq:def-lambda} \sum_{u\in H} \big(1+w_{(H,\pi,s)}(u)\big)-e(H)\cdot\theta
  \geBy{eq:def-d} 1+w_{(H,\pi,s)}(v_1) + d_\theta(H,v_1,w_{(H,\pi,s)}) \enspace.
\end{equation}
By Lemma~\ref{lemma:w-sum-x} the right hand side of \eqref{eq:lambda-w-d-H} equals $1+\sum_{t\in[r]} x_{(H,\pi,s),t}$, proving the first part of the lemma.

Now consider a graph $\Jhat$ from the family 
\begin{equation} \label{eq:J-hat}
  \argmin_{J\seq H:v_1\in J} \Big(\sum_{u\in J\setminus v_1} \big(1+w_{(H,\pi,s)}(u)\big)-e(J)\cdot\theta\Big) \enspace.
\end{equation}
Using the definition of $d_\theta()$ in \eqref{eq:def-d} we obtain
\begin{equation} \label{eq:d-H-d-Jhat}
  d_\theta(H,v_1,w_{(H,\pi,s)})
  \eqByM{\eqref{eq:def-d},\eqref{eq:J-hat}} \sum_{u\in \Jhat\setminus v_1} \big(1+w_{(H,\pi,s)}(u)\big) - e(\Jhat)\cdot\theta
  \eqByM{\eqref{eq:def-d},\eqref{eq:J-hat}} d_\theta(\Jhat,v_1,w_{(H,\pi,s)}) \enspace.
\end{equation}
Furthermore, Lemma~\ref{lemma:irrelevant-context-d} yields that
\begin{equation} \label{eq:w-H-w-Jhat-v1-vj}
  w_{(H,\pi,s)}(u)=w_{(\Jhat,\pi|_\Jhat,s)}(u) \quad \text{for all $u\in \Jhat$} \enspace.
\end{equation}
Recall from the first part of the proof that the right hand side of \eqref{eq:lambda-w-d-H} equals $1+\sum_{t\in[r]} x_{(H,\pi,s),t}$.
Applying \eqref{eq:d-H-d-Jhat}, \eqref{eq:w-H-w-Jhat-v1-vj} and Lemma~\ref{lemma:w-sum-x} shows that the right hand side of \eqref{eq:lambda-w-d-H} also equals $1+\sum_{t\in[r]} x_{(\Jhat,\pi|_\Jhat,s),t}$. Furthermore, applying the first equality in \eqref{eq:d-H-d-Jhat}, \eqref{eq:w-H-w-Jhat-v1-vj} and the definition of $\lambda_\theta()$ in \eqref{eq:def-lambda} shows that the right hand side of \eqref{eq:lambda-w-d-H} equals $\lambda_\theta(\Jhat,w_{(\Jhat,\pi|_\Jhat,s)})$, completing the proof of the second part of the lemma.
\end{proof}

We say that a family $\cD$ of ordered graphs is \emph{closed under taking subgraphs that contain the youngest vertex} if for any $(H,\pi)\in\cD$, $\pi=(v_1,\ldots,v_h)$, we have that for every $J\seq H$ with $v_1\in J$ the ordered graph $(J,\pi|_J)$ is also contained in $\cD$.

Note that the families $\cD_s$, $s\in[r]$, used by the strategy $\PAINT()$ and defined in \eqref{eq:def-D-s} are nonempty and closed under taking subgraphs that contain the youngest vertex.

\begin{lemma}[$x$-point of $\lambda_\theta()$-minimizing graphs] \label{lemma:argmin-lambda-x-vec}
Let $s\in[r]$ and $\cD_s\seq\cH_s$ a nonempty family of ordered graphs that is closed under taking subgraphs that contain the youngest vertex. For any graph $(J,\tau)$ from the family
\begin{equation*}
  \argmin_{(H,\pi)\in\cD_s} \lambda_\theta(H,w_{(H,\pi,s)})
\end{equation*}
we have
\begin{equation*}
  \lambda_\theta(J,w_{(J,\tau,s)}) = 1+\sum_{t\in[r]} x_{(J,\tau,s),t} \enspace.
\end{equation*}
\end{lemma}

\begin{proof}
The claim follows immediately from Lemma~\ref{lemma:lambda-x-vec}, using the closure property of the family $\cD_s$ and the choice of $(J,\tau)$.
\end{proof}

\begin{lemma}[$d_\theta()$-value of $\lambda_\theta()$-minimizing graphs] \label{lemma:argmin-d-lambda}
Let $s\in[r]$ and let $\cD_s\seq\cH_s\cup\cC(\cH_s,F)$ be a nonempty family of ordered graphs that is closed under taking subgraphs that contain the youngest vertex. Furthermore, let $(J,\tau)$, $\tau=(u_1,\ldots,u_c)$, be an inclusion-minimal graph from the family
\begin{equation*}
  \argmin_{(H,\pi)\in\cD_s} \lambda_\theta(H,w_{(H,\pi,s)}) \enspace.
\end{equation*}
Then we have
\begin{equation} \label{eq:d-J-tau}
  \lambda_\theta(J\setminus u_1,w_{(J,\tau,s)})-\deg_J(u_1)\cdot\theta = d_\theta(J,u_1,w_{(J,\tau,s)}) \enspace.
\end{equation}
\end{lemma}

\begin{proof}
We distinguish two cases depending on whether $(J,\tau)\in\cD_s\seq \cH_s\cup\cC(\cH_s,F)$ is contained in $\cH_s$ or in $\cC(\cH_s,F)$.

If $(J,\tau)\in\cH_s$, then by Lemma~\ref{lemma:argmin-lambda-x-vec} we have
\begin{equation} \label{eq:lambda-J-tau}
  \lambda_\theta(J,w_{(J,\tau,s)})=1+\sum_{t\in[r]} x_{(J,\tau,s),t} \enspace.
\end{equation}
Rewriting the left hand side of \eqref{eq:lambda-J-tau} according to \eqref{eq:lambda-recursive} and the right hand side according to Lemma~\ref{lemma:w-sum-x} yields the desired equality \eqref{eq:d-J-tau}.

We now consider the case $(J,\tau)\in\cC(\cH_s,F)$ (in this case we have $\lambda_\theta(J,w_{(J,\tau,s)})=-\infty$ by Lemma~\ref{lemma:finite-weights}). We clearly have
\begin{equation} \label{eq:lambda-J-minus-deg}
  \lambda_\theta(J\setminus u_1,w_{(J,\tau,s)})-\deg_J(u_1)\cdot\theta
  \eqBy{eq:def-lambda} \sum_{u\in J\setminus u_1} \big(1+w_{(J,\tau,s)}(u)\big)-e(J)\cdot\theta
  \geBy{eq:def-d} d_\theta(J,u_1,w_{(J,\tau,s)}) \enspace,
\end{equation}
and it remains to show that this inequality is in fact an equality. If the last inequality in \eqref{eq:lambda-J-minus-deg} were strict, then, as in the proof of Lemma~\ref{lemma:lambda-x-vec} (cf.~\eqref{eq:J-hat}, \eqref{eq:d-H-d-Jhat} and \eqref{eq:w-H-w-Jhat-v1-vj}), there would be a proper subgraph $\Jhat\sneq J$ with $u_1\in\Jhat$ satisfying
\begin{equation} \label{eq:dJ-dJhat}
  d_\theta(J,u_1,w_{(J,\tau,s)})=d_\theta(\Jhat,u_1,w_{(\Jhat,\tau|_\Jhat,s)}) \enspace.
\end{equation}
As $(J,\tau)\in\cC(\cH_s,F)$ we have $(J\setminus u_1,\tau\setminus u_1)\in\cH_s$, which by Lemma~\ref{lemma:subgraph-monotonicity} implies that $(\Jhat\setminus u_1,\tau|_{\Jhat\setminus u_1})\in\cH_s$ as well. Hence $(\Jhat,\tau|_\Jhat)$ must be in $\cH_s\cup\cC(\cH_s,F)$. Using \eqref{eq:dJ-dJhat} and the first part of Lemma~\ref{lemma:d-CHs-dHs} shows that $(\Jhat,\tau|_\Jhat)$ must be contained in $\cC(\cH_s,F)$. But then we have $\lambda_\theta(\Jhat,w_{(\Jhat,\tau|_\Jhat,s)})=-\infty$ by Lemma~\ref{lemma:finite-weights}, a contradiction to the inclusion-minimality of $(J,\tau)$ (here we used again the closure property of the family $\cD_s$). Therefore the last inequality in \eqref{eq:lambda-J-minus-deg} holds with equality, proving the lemma also in this case.
\end{proof}

\begin{lemma}[$\lambda_\theta()$-minimizing graphs in $\cH_s'$] \label{lemma:argmins-Hsp}
Let $s\in[r]$ and let $\cD_s\seq\cH_s$ be a nonempty family of ordered graphs that is closed under taking subgraphs that contain the youngest vertex. Furthermore, let $(J,\tau)$ be an inclusion-minimal graph from the family
\begin{equation*}
  \argmin_{(H,\pi)\in\cD_s} \lambda_\theta(H,w_{(H,\pi,s)})
\end{equation*}
and suppose that\/ $x_{(J,\tau,s)}$ is the starting point of some $s$-segment of the walk $\cW$ defined in Lemma~\ref{lemma:walk}. Then $(J,\tau)$ is contained in $\cH_s'$.
\end{lemma}

\begin{proof}
By Lemma~\ref{lemma:argmin-lambda-x-vec} we have
\begin{equation} \label{eq:lambda-J-x-J-tau-s}
  \lambda_\theta(J,w_{(J,\tau,s)}) = 1+\sum_{t\in[r]} x_{(J,\tau,s),t} \enspace.
\end{equation}
Let $u_1$ denote the youngest vertex of $(J,\tau)$ and let $\tJ\sneq J$ be any proper subgraph of $J$ with $u_1\in\tJ$. By Lemma~\ref{lemma:lambda-x-vec} there is a subgraph $\Jhat\seq \tJ$ with $u_1\in\Jhat$ satisfying
\begin{equation} \label{eq:lambda-tJ-x-tJ-tau-s}
  \lambda_\theta(\Jhat,w_{(\Jhat,\tau|_\Jhat,s)}) = 1+\sum_{t\in[r]} x_{(\tJ,\tau|_\tJ,s),t} \enspace.
\end{equation}
By the inclusion-minimal choice of $(J,\tau)$, \eqref{eq:lambda-tJ-x-tJ-tau-s} must be strictly larger than \eqref{eq:lambda-J-x-J-tau-s}, i.e., we have
\begin{equation*}
  1+\sum_{t\in[r]} x_{(J,\tau,s),t} < 1+\sum_{t\in[r]} x_{(\tJ,\tau|_\tJ,s),t} \enspace,
\end{equation*}
in particular
\begin{equation*}
  x_{(J,\tau,s)}\neq x_{(\tJ,\tau|_\tJ,s)} \enspace.
\end{equation*}
Using this observation together with the assumption that $x_{(J,\tau,s)}$ is the starting point of some $s$-segment of $\cW$, it follows from property~(iii) in Lemma~\ref{lemma:walk} that $(J,\tau)$ must be contained in $\cH_s'$.
\end{proof}

\begin{lemma}[$x$-points of graphs from $\cH_s'$ on the walk] \label{lemma:forward-walk}
Let $s\in[r]$ and $(J,\tau)\in\cH_s'$, $\tau=(u_1,\ldots,u_c)$. Moreover, let\/ $1\leq b\leq c-1$ and define $(J^{-b},\tau^{-b}):=(J\setminus\{u_1,\ldots,u_b\},\tau\setminus\{u_1,\ldots,u_b\})$. \\
Then $x_{(J,\tau,s)}$ is lower than $x_{(J^{-b},\tau^{-b},s)}$ on the walk $\cW$ defined in Lemma~\ref{lemma:walk} and both points are contained in different $s$-segments of this walk.
\end{lemma}

\begin{proof}
By the definition of $\cH_s'$ in \eqref{eq:tie-break-graphs} we have $(J,\tau)\in\cC^{i,1}$ for some $i\geq 1$ with $\alpha_i=s$, i.e., $(J,\tau)$ was added to the family $\cH_s$ in the first iteration of the repeat-loop~(**) in the $i$-th iteration of the repeat-loop~(*) of $\CW()$.
By the definition in \eqref{eq:vec-pt} we have
\begin{equation} \label{eq:x-J-s}
  x_{(J,\tau,s),s} = d_\theta(J,u_1,w_{(J,\tau,s)}) = d_s^i \enspace.
\end{equation}
By Lemma~\ref{lemma:subgraph-monotonicity}, the graph $(J^{-b},\tau^{-b})$ was added to the family $\cH_s$ either before the graph $(J,\tau)$ or together with it. But as $(J^{-b},\tau^{-b})$ is a predecessor of $(J,\tau)$ in the tree $\cT(F)$ defined after~\eqref{eq:def-S}, it follows from the definition in line~\ref{cw:primary-threats} that $(J^{-b},\tau^{-b})$ must have already been contained in $\cH_s$ at the beginning of the $i$-th iteration of the repeat-loop~(*). Applying Lemma~\ref{lemma:beginning-loop-*} yields
\begin{equation} \label{eq:x-J-minus-s}
  x_{(J^{-b},\tau^{-b},s),s} \eqBy{eq:def-x-H-pi-s} d_\theta(J^{-b},u_{b+1},w_{(J^{-b},\tau^{-b},s)}) > d_s^i \enspace.
\end{equation}
Combining \eqref{eq:x-J-s} and \eqref{eq:x-J-minus-s} shows that $x_{(J,\tau,s)}$ is lower than $x_{(J^{-b},\tau^{-b},s)}$ on the walk $\cW$. As by the assumption $(J,\tau)\in\cH_s'$ and property~(iv) from Lemma~\ref{lemma:walk} the point $x_{(J,\tau,s)}$ is the starting point of an $s$-segment of $\cW$, this implies that both points must be contained in different $s$-segments of this walk.
\end{proof}

\subsection{Analysis of \texorpdfstring{$\PAINT()$}{Paint()}}
We are now in a position to actually analyze our Painter strategy $\PAINT()$.
Recall from Section~\ref{sec:strategy} that the parameter $d(s)$ defined in \eqref{eq:def-d-s} might be equal to $-\infty$ for some colors $s\in[r]$ (intuitively, Painter considers such a color extremely dangerous). The following lemma shows that $\PAINT()$ never chooses such a color.

\begin{lemma}[Painter strategy creates only graphs from $\cH_s$] \label{lemma:Dsigma-seq-Hsigma}
Consider a fixed step of the game, and let the families $\cD_s\seq \cS(F)$, $s\in[r]$, and the values $d(s)\in\mathbb{R}\cup\{-\infty\}$ be defined as in \eqref{eq:def-D-s} and \eqref{eq:def-d-s}, respectively. For any $\sigma\in\argmax_{s\in[r]} d(s)$ the value $d(\sigma)$ is finite and we have $\cD_\sigma\seq\cH_\sigma$.

Consequently, playing according to the strategy $\PAINT()$ throughout ensures that for all $s\in[r]$ we always have $\cD_s\seq\cH_s\cup\cC(\cH_s,F)$ (even if ties are broken arbitrarily).
\end{lemma}

\begin{proof}
By the definition in \eqref{eq:def-d-s} and Lemma~\ref{lemma:finite-weights} the value $d(s)$ is finite if and only if $\cD_s\seq \cH_s$.
By the termination condition in line~\ref{cw:terminate} there is some color $s\in[r]$ for which $\cH_s=\cS(F)$. For this color we therefore have $\cD_s\seq\cH_s$, implying that the corresponding value $d(s)$ is finite. This shows that for any $\sigma\in\argmax_{s\in[r]} d(s)$, the value $d(\sigma)$ is finite and therefore $\cD_\sigma\seq\cH_\sigma$, proving the first part of the lemma. The second part follows inductively by observing that the strategy $\PAINT()$ in each step picks a color $\sigma\in\argmax_{s\in[r]} d(s)$ (regardless of the tie-breaking rule), showing that in this step only graphs from the family $\cD_\sigma\seq\cH_\sigma$ in color $\sigma$ are created on the board. 
\end{proof}

The following lemma shows that the tie-breaking rule of the strategy $\PAINT()$, which uses the families $\cH_s'$ and $\cJ_s$ defined in \eqref{eq:tie-break-graphs} and \eqref{eq:def-J-s}, is indeed well-defined.

\begin{lemma}[Well-definedness of Painter strategy] \label{lemma:tie-breaking}
Ties in the strategy $\PAINT()$ can arise only between two different colors, and if they arise then for exactly one of the two colors (say $\sigma$) we have $\cJ_\sigma\cap\cH_\sigma'=\emptyset$, and for the other color (say $s$) we have $\cJ_s\cap\cH_s'\neq\emptyset$. (Thus the tie-breaking rule will decide for color $\sigma$.)

If such a tie arises, the walk $\cW$ defined in Lemma~\ref{lemma:walk} contains a $\sigma$-segment $\Gamma$ whose endpoint $x\in\mathbb{R}^r$ is also the starting point of an $s$-segment $\Gamma'$, and for any $(J_\sigma,\tau_\sigma)\in\cJ_\sigma$ and any $(J_s,\tau_s)\in\cJ_s$ we have $x_{(J_\sigma,\tau_\sigma,\sigma)}=x_{(J_s,\tau_s,s)}=x$.
\end{lemma}

\begin{proof}
Recall that the families $\cD_s$, $s\in[r]$, defined in \eqref{eq:def-D-s} are nonempty and closed under taking subgraphs that contain the youngest vertex. Fix some color $\sigma\in[r]$ such that
\begin{equation} \label{eq:d-s-d-sigma}
  d(s)\leq d(\sigma) \quad, \enspace s\in[r]\setminus\{\sigma\} \enspace,
\end{equation}
for the values $d(s)$, $s\in[r]$, defined in \eqref{eq:def-d-s}.
The tie-breaking rule of the strategy $\PAINT()$ is only considered if the inequality in \eqref{eq:d-s-d-sigma} is tight for some color different from $\sigma$. We fix such a color $s\in[r]\setminus\{\sigma\}$ for which
\begin{equation} \label{eq:same-d-value}
  d(s)=d(\sigma) \enspace.
\end{equation}
The first part of Lemma~\ref{lemma:Dsigma-seq-Hsigma} yields with \eqref{eq:d-s-d-sigma} and \eqref{eq:same-d-value} that $\cD_s\seq \cH_s$ and $\cD_\sigma\seq\cH_\sigma$ (and that $d(s)$ and $d(\sigma)$ are finite values). Thus by the definition in \eqref{eq:def-J-s} we also have $\cJ_s\seq \cH_s$ and $\cJ_\sigma\seq\cH_\sigma$.
Fix some $(J_s,\tau_s)\in\cJ_s$ and some $(J_\sigma,\tau_\sigma)\in\cJ_\sigma$. By the definition in \eqref{eq:def-J-s} and Lemma~\ref{lemma:argmin-lambda-x-vec} we have
\begin{equation} \label{eq:lambda-Js-Jsigma}
  \lambda_\theta(J_s,w_{(J_s,\tau_s,s)}) = 1+\sum_{t\in[r]} x_{(J_s,\tau_s,s),t} \;\;\quad \text{and} \quad\;\;
  \lambda_\theta(J_\sigma,w_{(J_\sigma,\tau_\sigma,\sigma)}) = 1+\sum_{t\in[r]} x_{(J_\sigma,\tau_\sigma,\sigma),t} \enspace.
\end{equation}
Furthermore, using \eqref{eq:same-d-value} and the definitions in \eqref{eq:def-d-s} and \eqref{eq:def-J-s} shows that
\begin{equation} \label{eq:lambda-Js-eq-Jsigma}
  \lambda_\theta(J_s,w_{(J_s,\tau_s,s)})=\lambda_\theta(J_\sigma,w_{(J_\sigma,\tau_\sigma,\sigma)}) \enspace.
\end{equation}
Combining \eqref{eq:lambda-Js-Jsigma} and \eqref{eq:lambda-Js-eq-Jsigma} we obtain
\begin{equation} \label{eq:sum-x-Js-x-Jsigma}
  1+\sum_{t\in[r]} x_{(J_s,\tau_s,s),t} = 1+\sum_{t\in[r]} x_{(J_\sigma,\tau_\sigma,\sigma),t} \enspace.
\end{equation}
Note that the points $x_{(J_s,\tau_s,s)}$ and $x_{(J_\sigma,\tau_\sigma,\sigma)}$ are elements of the walk $\cW$ defined in Lemma~\ref{lemma:walk}. By Lemma~\ref{lemma:walk-order} the relation \eqref{eq:sum-x-Js-x-Jsigma} implies that $x_{(J_s,\tau_s,s)}=x_{(J_\sigma,\tau_\sigma,\sigma)}$, i.e., the graphs $(J_s,\tau_s)$ and $(J_\sigma,\tau_\sigma)$ (and all other graphs in the families $\cJ_s$ and $\cJ_\sigma$) have the same associated point on the walk $\cW$.
By property~(i) in Lemma~\ref{lemma:walk}, $x_{(J_s,\tau_s,s)}$ is contained in an $s$-segment and $x_{(J_\sigma,\tau_\sigma,\sigma)}$ in a $\sigma$-segment of $\cW$, implying that $x_{(J_s,\tau_s,s)}=x_{(J_\sigma,\tau_\sigma,\sigma)}$ must be the endpoint of some segment and the starting point of the next lower segment. As on the walk $\cW$, only pairs of consecutive segments have a point in common, this shows that the inequality \eqref{eq:d-s-d-sigma} can be tight for at most one color different from $\sigma$, proving that ties can arise only between two different colors.

Assume \wolog that for all $(J_\sigma,\tau_\sigma)\in\cJ_\sigma$, the point $x_{(J_\sigma,\tau_\sigma,\sigma)}$ is the endpoint of some $\sigma$-segment $\Gamma$ and for all $(J_s,\tau_s)\in\cJ_s$ the point $x_{(J_s,\tau_s,s)}$ is the starting point of the next lower $s$-segment $\Gamma'$. By property~(iv) from Lemma~\ref{lemma:walk} we have $\cJ_\sigma\cap\cH_\sigma'=\emptyset$, and by Lemma~\ref{lemma:argmins-Hsp} we have $\cJ_s\cap\cH_s'\neq \emptyset$. This proves the first part of the lemma and shows that our tie-breaking rule is well-defined.

Note that the segment $\Gamma$ is higher on the walk $\cW$ than $\Gamma'$, and that the tie-breaking rule decides for the color $\sigma$ corresponding to the higher of the two segments. Together with our previous observations about the location of the points $x_{(J_\sigma,\tau_\sigma,\sigma)}$ for all $(J_\sigma,\tau_\sigma)\in\cJ_\sigma$ and $x_{(J_s,\tau_s,s)}$ for all $(J_s,\tau_s)\in\cJ_s$ this proves the second part of the lemma.
\end{proof}

The following lemma will be the key to proving Lemma~\ref{lemma:witness-graphs}, our main strategy invariant based on witness graphs.

\begin{lemma}[Painter strategy ensures sufficient weight] \label{lemma:sufficient-weight}
There is a constant $\eps=\eps(F,r,\theta,\alpha)>0$ such that the following holds:
Let $\sigma\in[r]$ denote the color selected by the strategy $\PAINT()$ in a certain step of the game given the families $\cD_s$, $s\in[r]$, defined in \eqref{eq:def-D-s}. For every $s\in[r]\setminus\{\sigma\}$, let $(J_s,\tau_s)$ be an inclusion-minimal graph from the family
\begin{equation*}
  \cJ_s \eqBy{eq:def-J-s} \argmin_{(H,\pi)\in\cD_s} \lambda_\theta(H,w_{(H,\pi,s)}) \enspace,
\end{equation*}
and let $u_{s1}$ denote the youngest vertex of $(J_s,\tau_s)$.

Then for any graph $(H,\pi)\in\cD_\sigma$, $\pi=(v_1,\ldots,v_h)$, we have
\begin{equation} \label{eq:sufficient-weight}
  \sum_{s\in[r]\setminus\{\sigma\}} \big(\lambda_\theta(J_s\setminus u_{s1},w_{(J_s,\tau_s,s)})-\deg_{J_s}(u_{s1})\big) \leq w_{(H,\pi,\sigma)}(v_1) \enspace.
\end{equation}

If the inequality \eqref{eq:sufficient-weight} is strict, then the difference between the right and left hand side is at least $\eps$.

If on the other hand the inequality \eqref{eq:sufficient-weight} is tight, then for every $s\in[r]\setminus\{\sigma\}$ we have $(J_s,\tau_s)\in\cH_s'\cup\cC(\cH_s,F)$. Moreover, denoting by $\cW$ the walk defined in Lemma~\ref{lemma:walk} and by $\Gamma$ the $\sigma$-segment containing $x_{(H,\pi,\sigma)}$ on this walk, we have the following: if $(J_s,\tau_s)\in\cH_s'$, then $x_{(J_s,\tau_s,s)}$ is the starting point of the next $s$-segment on $\cW$ that is lower than\/ $\Gamma$, whereas if $(J_s,\tau_s)\in\cC(\cH_s,F)$, then there is no $s$-segment on $\cW$ lower than\/ $\Gamma$.
\end{lemma}

\begin{proof}
Let
\begin{multline} \label{eq:eps}
  \eps=\eps(F,r,\theta,\alpha):=\min\{\enspace
  |d_\theta(H,v_1,w_{(H,\pi,s)})-d_\theta(J,u_1,w_{(J,\tau,s)})| \enspace\mid\enspace
  s\in[r] \;\wedge\; \\
  (H,\pi=(v_1,\ldots,v_h)),(J,\tau=(u_1,\ldots,u_c))\in \cH_s\cup\cC(\cH_s,F) \;\wedge\; \\
  d_\theta(H,v_1,w_{(H,\pi,s)})\neq d_\theta(J,u_1,w_{(J,\tau,s)}) \enspace\}>0
\end{multline}
(recall that for all $s\in[r]$ the $d_\theta()$-value of all graphs in $\cH_s\cup\cC(\cH_s,F)$ is a finite real number).
Note that in the boundary case that for all $s\in[r]$ and all $(H,\pi)\in\cH_s\cup\cC(\cH_s,F)$, $\pi=(v_1,\ldots,v_h)$, the value $d_\theta(H,v_1,w_{(H,\pi,s)})$ is the same (i.e., the walk $\cW$ defined in Lemma~\ref{lemma:walk} degenerates to a single point), the minimum in \eqref{eq:eps} is over an empty set. We will see that in this case the inequality \eqref{eq:sufficient-weight} is never strict. Therefore we may set $\eps$ to an arbitrary positive constant in this case, $\eps:=1$, say.

Recall that the families $\cD_s$, $s\in[r]$, are nonempty and closed under taking subgraphs that contain the youngest vertex. By the second part of Lemma~\ref{lemma:Dsigma-seq-Hsigma} we have $\cD_s\seq \cH_s\cup\cC(\cH_s,F)$ for all $s\in[r]$.

We first prove that \eqref{eq:sufficient-weight} holds.
By the definition of the strategy, the selected color $\sigma\in[r]$ satisfies
\begin{equation} \label{eq:d-s-leq-d-sigma}
  d(s)\leq d(\sigma) \quad, \enspace s\in[r]\setminus\{\sigma\}
\end{equation}
for the values $d(s)$, $s\in[r]$, defined in \eqref{eq:def-d-s}.
Let $(J_\sigma,\tau_\sigma)$ be an arbitrary graph from the family
\begin{equation*}
  \cJ_\sigma \eqBy{eq:def-J-s} \argmin_{(H,\pi)\in\cD_\sigma} \lambda_\theta(H,w_{(H,\pi,\sigma)}) \enspace.
\end{equation*}
By the definition in \eqref{eq:def-d-s} and the choice of $(J_s,\tau_s)$, $s\in[r]$, we have
\begin{equation} \label{eq:lambda-J-s-d-s}
  \lambda_\theta(J_s,w_{(J_s,\tau_s,s)})=d(s) \quad, \enspace \text{$s\in[r]$} \enspace.
\end{equation}
Combining \eqref{eq:d-s-leq-d-sigma} and \eqref{eq:lambda-J-s-d-s} yields
\begin{equation} \label{eq:least-dangerous-color}
  \lambda_\theta(J_s,w_{(J_s,\tau_s,s)}) \leq \lambda_\theta(J_\sigma,w_{(J_\sigma,\tau_\sigma,\sigma)}) \quad, \enspace \text{$s\in[r]\setminus\{\sigma\}$} \enspace.
\end{equation}
By \eqref{eq:d-s-leq-d-sigma} and the first part of Lemma~\ref{lemma:Dsigma-seq-Hsigma} we have $\cD_\sigma\seq\cH_\sigma$. We fix some graph $(H,\pi)\in\cD_\sigma$, $\pi=(v_1,\ldots,v_h)$. By Lemma~\ref{lemma:lambda-x-vec} there is a subgraph $\Jhat\seq H$ with $v_1\in\Jhat$ satisfying
\begin{equation} \label{eq:lambda-Jhat-x-H}
  \lambda_\theta(\Jhat,w_{(\Jhat,\pi|_\Jhat,\sigma)}) = 1+\sum_{t\in[r]} x_{(H,\pi,\sigma),t} \enspace.
\end{equation}
By the closure property of the family $\cD_\sigma$ and the choice of $(J_\sigma,\tau_\sigma)$ we have
\begin{equation} \label{eq:lambda-J-sigma-Jhat}
  \lambda_\theta(J_\sigma,w_{(J_\sigma,\tau_\sigma,\sigma)}) \leq \lambda_\theta(\Jhat,w_{(\Jhat,\pi|_\Jhat,\sigma)}) \enspace.
\end{equation}
By Lemma~\ref{lemma:argmin-lambda-x-vec} we have for every $s\in[r]\setminus\{\sigma\}$ for which $(J_s,\tau_s)$ is contained in $\cH_s$ that
\begin{equation} \label{eq:lambda-Js-x-Js}
  \lambda_\theta(J_s,w_{(J_s,\tau_s,s)}) = 1+\sum_{t\in[r]} x_{(J_s,\tau_s,s)} \enspace.
\end{equation}
For those $s\in[r]\setminus\{\sigma\}$ we thus obtain
\begin{equation} \label{eq:sum-x-J-s-sum-x-H-s}
  1+\sum_{t\in[r]} x_{(J_s,\tau_s,s)}
  \eqBy{eq:lambda-Js-x-Js} \lambda_\theta(J_s,w_{(J_s,\tau_s,s)})
  \leBy{eq:least-dangerous-color} \lambda_\theta(J_\sigma,w_{(J_\sigma,\tau_\sigma,\sigma)})
  \leByM{\eqref{eq:lambda-Jhat-x-H},\eqref{eq:lambda-J-sigma-Jhat}} 1+\sum_{t\in[r]} x_{(H,\pi,\sigma),t} \enspace.
\end{equation}
Note that if $(J_s,\tau_s)\in\cH_s$, then $x_{(J_s,\tau_s,s)}$ is an element of the walk $\cW$ defined in Lemma~\ref{lemma:walk} (the point $x_{(H,\pi,\sigma)}$ is clearly also an element of this walk as $\cD_\sigma\seq\cH_\sigma$).
Using \eqref{eq:sum-x-J-s-sum-x-H-s} and Lemma~\ref{lemma:walk-order} yields that for every $s\in[r]\setminus\{\sigma\}$ for which $(J_s,\tau_s)$ is contained in $\cH_s$ we have 
\begin{equation} \label{eq:x-J-s-x-H-all-coords}
  x_{(J_s,\tau_s,s),t} \leq x_{(H,\pi,\sigma),t} \quad \text{for all $t\in[r]$} \enspace,
\end{equation}
from which we conclude using
\begin{equation} \label{eq:x-J-s-d-J-s}
  x_{(J_s,\tau_s,s),s} \eqBy{eq:def-x-H-pi-s} d_\theta(J_s,u_{s1},w_{(J_s,\tau_s,s)})
\end{equation}
that
\begin{equation} \label{eq:d-J-s-x-H}
  d_\theta(J_s,u_{s1},w_{(J_s,\tau_s,s)}) \leq x_{(H,\pi,\sigma),s} \enspace.
\end{equation}
For all $s\in[r]\setminus\{\sigma\}$ for which $(J_s,\tau_s)$ is not contained in $\cH_s$ but in $\cC(\cH_s,F)$, the relation \eqref{eq:d-J-s-x-H} follows from the second part of Lemma~\ref{lemma:d-CHs-dHs}.

Combining our previous observations and applying Lemma~\ref{lemma:w-sum-x} and Lemma~\ref{lemma:argmin-d-lambda}, we thus obtain
\begin{equation} \label{eq:sum-lambda-x-wv1}
\begin{split}
  \sum_{s\in[r]\setminus\{\sigma\}} \big(\lambda_\theta(J_s\setminus u_{s1},w_{(J_s,\tau_s,s)})-\deg_{J_s}(u_{s1})\big)
  &\eqBy{eq:d-J-tau}
  \sum_{s\in[r]\setminus\{\sigma\}} d_\theta(J_s,u_{s1},w_{(J_s,\tau_s,s)}) \\
  &\leBy{eq:d-J-s-x-H}
  \sum_{s\in[r]\setminus\{\sigma\}} x_{(H,\pi,\sigma),s}
  \eqBy{eq:w-sum-x}
   w_{(H,\pi,\sigma)}(v_1) \enspace,
\end{split}
\end{equation}
proving \eqref{eq:sufficient-weight}.

If the inequality \eqref{eq:sum-lambda-x-wv1} is strict, then by \eqref{eq:d-J-s-x-H} we have
\begin{equation} \label{eq:ineq-strict}
  d_\theta(J_s,u_{s1},w_{(J_s,\tau_s,s)})<x_{(H,\pi,\sigma),s}\quad \text{for some $s\in[r]\setminus\{\sigma\}$} \enspace.
\end{equation}
By the definition in \eqref{eq:def-x-H-pi-s} and the definition in line~\ref{cw:least-dangerous-threat}, the right hand side of \eqref{eq:ineq-strict} equals $d_\theta(\Jbar,\ubar_1,w_{(\Jbar,\taubar,s)})$ for some $(\Jbar,\taubar)\in\cH_s\cup\cC(\cH_s,F)$, where $\ubar_1$ denotes the youngest vertex of $(\Jbar,\taubar)$. With the definition in \eqref{eq:eps} it follows that the difference between the right and left hand side of \eqref{eq:ineq-strict} and therefore also the difference between the right and left hand side of \eqref{eq:sum-lambda-x-wv1} is at least $\eps$.

If the inequality in \eqref{eq:sum-lambda-x-wv1} is tight, then by \eqref{eq:d-J-s-x-H} we have
\begin{equation} \label{eq:x-J-s-x-H-all-equal}
  d_\theta(J_s,u_{s1},w_{(J_s,\tau_s,s)}) = x_{(H,\pi,\sigma),s} \quad \text{for all $s\in[r]\setminus\{\sigma\}$} \enspace.
\end{equation}
Let $\Gamma$ denote the $\sigma$-segment on the walk $\cW$ that contains the point $x_{(H,\pi,\sigma)}$. We fix some $s\in[r]\setminus\{\sigma\}$ and distinguish the cases whether $(J_s,\tau_s)$ is contained in $\cH_s$ or in $\cC(\cH_s,F)$.

We first consider the case that $(J_s,\tau_s)\in\cH_s$. We claim that the $s$-segment $\Gamma'$ on the walk $\cW$ that contains the point $x_{(J_s,\tau_s,s)}$ is lower on the walk $\cW$ than $\Gamma$: This is trivially true if one of the inequalities in \eqref{eq:x-J-s-x-H-all-coords} is strict. If on the other hand \eqref{eq:x-J-s-x-H-all-coords} holds with equality for all $t\in[r]$, then also all inequalities in \eqref{eq:sum-x-J-s-sum-x-H-s} are tight, from which we conclude using \eqref{eq:lambda-J-s-d-s} that $d(s)=d(\sigma)$, i.e., we have a tie between the colors $s$ and $\sigma$. In this case, by the second part of Lemma~\ref{lemma:tie-breaking}, our tie-breaking rule ensures that $\Gamma'$ is lower on the walk $\cW$ than $\Gamma$.
From \eqref{eq:x-J-s-d-J-s} and \eqref{eq:x-J-s-x-H-all-equal} it follows that $\Gamma'$ must be the next $s$-segment on $\cW$ that is lower than $\Gamma$ and that $x_{(J_s,\tau_s,s)}$ must be the starting point of $\Gamma'$. Applying Lemma~\ref{lemma:argmins-Hsp} shows that $(J_s,\tau_s)\in\cH_s'$ (note the inclusion-minimal choice of $(J_s,\tau_s)$), completing the proof in this case.

It remains to consider the case $(J_s,\tau_s)\in\cC(\cH_s,F)$. Suppose for the sake of contradiction that there was some $s$-segment $\Gamma'$ that is lower than $\Gamma$ on the walk $\cW$. By property~(v) in Lemma~\ref{lemma:walk} the segment $\Gamma'$ can not be lowest segment of $\cW$, as we would otherwise have $\cH_s=\cS(F)$ and therefore $\cC(\cH_s,F)=\emptyset$. So $\Gamma'$ has an endpoint $x\in\mathbb{R}^r$ which clearly satisfies
\begin{equation} \label{eq:x-s-x-H-pi-sigma-s}
  x_s<x_{(H,\pi,\sigma),s}
\end{equation}
and which is also the starting point of the next lower segment $\Gammabar$ (the segment $\Gammabar$ is an $\sbar$-segment for some $\sbar\in[r]\setminus\{s\}$). By property~(ii) of Lemma~\ref{lemma:walk} there is some $(\Jbar,\taubar)\in\cH_\sbar$ such that
\begin{equation} \label{eq:sp-ep}
  x_{(\Jbar,\taubar,\sbar)}=x \enspace.
\end{equation}
Applying the second part of Lemma~\ref{lemma:d-CHs-dHs} we thus obtain
\begin{equation*}
  d_\theta(J_s,u_{s1},w_{(J_s,\tau_s,s)}) \leq x_{(\Jbar,\taubar,\sbar),s} \eqBy{eq:sp-ep} x_s \lBy{eq:x-s-x-H-pi-sigma-s} x_{(H,\pi,\sigma),s} \enspace,
\end{equation*}
contradicting \eqref{eq:x-J-s-x-H-all-equal}. This completes the proof also in this case.
\end{proof}

\subsection{Proof of Lemma~\texorpdfstring{\ref{lemma:witness-graphs}}{35}}

We are now ready to prove Lemma~\ref{lemma:witness-graphs}, our main strategy invariant.

\begin{proof}[Proof of Lemma~\ref{lemma:witness-graphs}]
For the reader's convenience, Figure~\ref{fig:lbvars1} illustrates the notations used in the first part of the proof.

\begin{figure}
\centering
\PSforPDF{
 \psfrag{sinr}{$s\in[r]\setminus\{\sigma\}$}
 \psfrag{v}{$v$}
 \psfrag{vdots}{$\vdots$}
 \psfrag{EH}{$E_\sigma^v$}
 \psfrag{EJ}{$E_s^v$}
 \psfrag{H}{$(H,\pi)$}
 \psfrag{Hm}{$(H\setminus v_1,\pi\setminus v_1)$}
 \psfrag{J}{$(J_s,\tau_s)$}
 \psfrag{Jm}{$(J_s\setminus u_{s1},\tau_s\setminus u_{s1})$}
 \psfrag{Hsp}{$H_\sigma'=(V_\sigma,E_\sigma)$}
 \psfrag{Jsp}{$J_s'=(V_s,E_s)$}
 \psfrag{Kp}{$K_s'$}
 \psfrag{Hp}{\Large $H'=(V',E')$}
 \includegraphics{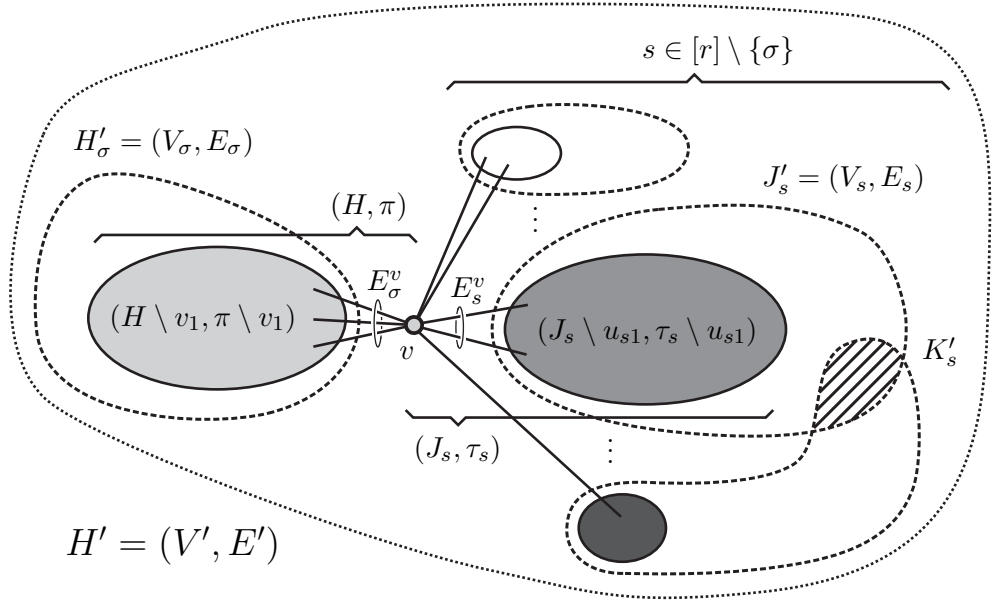}
}
\caption{Notations used in the first part of the proof of Lemma~\ref{lemma:witness-graphs}.} \label{fig:lbvars1}
\end{figure}

Let
\begin{equation} \label{eq:vmax}
  \vmax = \vmax(F,r,\theta,\alpha) := r^{v(F)/\eps+(v(F)/\eps+1)(r\cdot|\cS(F)|+1)(v(F)+1)+2}\cdot v(F)+1 \enspace,
\end{equation}
where $\eps=\eps(F,r,\theta,\alpha)$ is the constant guaranteed by Lemma~\ref{lemma:sufficient-weight} (and explicitly defined in \eqref{eq:eps}).

We argue by induction over the number of vertices of the board. For the induction base consider the board at the beginning of the game when no vertex is added yet. It is convenient for the proof to extend the statement of the lemma to $(H,\pi)$ being the null graph (the graph whose vertex set is empty). For this graph we define $H'$ to be the null graph as well. Clearly, for every $s\in[r]$, every copy of the null graph $(H,\pi)$ `in color $s$' on the board is contained in this subgraph $H'$ of the board, and we have $\mu_\theta(H')=0=\lambda_\theta(H,w_{(H,\pi,s)})$ and $v(H')=0\leq \vmax$. This shows that the second condition of the lemma holds at the beginning of the game and settles the induction base.

For the induction step, let $v$ denote the vertex added in the current step of the game, $\cD_s$, $s\in[r]$, the families defined in \eqref{eq:def-D-s}, and $\sigma$ the color the strategy $\PAINT()$ assigns to the vertex $v$. By the first part of Lemma~\ref{lemma:Dsigma-seq-Hsigma}, we have $\cD_\sigma\seq\cH_\sigma$.

For a fixed graph $(H,\pi)\in\cD_\sigma$, $\pi=(v_1,\ldots,v_h)$, we consider a fixed copy of $(H\setminus v_1,\pi\setminus v_1)$ in color $\sigma$ that is completed by $v$ to a copy of $(H,\pi)$ in this color. Denoting by $E_\sigma^v$ the corresponding set of edges incident to $v$, we clearly have
\begin{equation} \label{eq:E-sigma-v-deg}
  |E_\sigma^v|=\deg_H(v_1) \enspace.
\end{equation}
By induction, we know that this copy of $(H\setminus v_1,\pi\setminus v_1)$ is contained in a graph $H_\sigma'=(V_\sigma,E_\sigma)$ with
\begin{equation} \label{eq:mu-H-sigma'}
  \mu_\theta(H_\sigma')\leq \lambda_\theta(H\setminus v_1,w_{(H,\pi,\sigma)})
\end{equation}
(recall from~\eqref{eq:def-w} that $w_{(H,\pi,\sigma)}(u)=w_{(H\setminus v_1,\pi\setminus v_1,\sigma)}(u)$ for all $u\in H\setminus v_1$) and
\begin{equation} \label{eq:size-H-sigma'}
  v(H_\sigma')\leq \vmax \enspace.
\end{equation}
For every $s\in[r]\setminus\{\sigma\}$, let $(J_s,\tau_s)$ be an inclusion-minimal graph from the family $\cJ_s\seq\cD_s$ defined in \eqref{eq:def-J-s}, and let $u_{s1}$ denotes the youngest vertex of $(J_s,\tau_s)$. For each $s\in[r]\setminus\{\sigma\}$ we consider a fixed copy of $(J_s\setminus u_{s1},\tau_s\setminus u_{s1})$ in color $s$ that is completed by $v$ to a copy of $(J_s,\tau_s)$ (the vertex $v$ has the color $\sigma\neq s$, so the resulting copy is not monochromatic). Denoting by $E_s^v$ the corresponding set of edges incident to $v$, we clearly have
\begin{equation} \label{eq:E-s-v-deg}
  |E_s^v|=\deg_{J_s}(u_{s1}) \quad, \enspace s\in[r]\setminus\{\sigma\} \enspace.
\end{equation}
By induction, those copies of $(J_s\setminus u_{s1},\tau_s\setminus u_{s1})$ are contained in graphs $J_s'=(V_s,E_s)$ with
\begin{equation} \label{eq:mu-J-s'}
  \mu_\theta(J_s') \leq \lambda_\theta(J_s\setminus u_{s1},w_{(J_s,\tau_s,s)}) \quad, \enspace s\in[r]\setminus\{\sigma\} \enspace,
\end{equation}
and
\begin{equation} \label{eq:size-J-s'}
  v(J_s')\leq \vmax \quad, \enspace s\in[r]\setminus\{\sigma\} \enspace.
\end{equation}

Applying Lemma~\ref{lemma:sufficient-weight} shows that the graphs $(H,\pi)$ and $(J_s,\tau_s)$, $s\in[r]\setminus\{\sigma\}$, satisfy
\begin{equation} \label{eq:sufficient-weight-behind-v1}
  \sum_{s\in[r]\setminus\{\sigma\}} \big(\lambda_\theta(J_s\setminus u_{s1},w_{(J_s,\tau_s,s)})-\deg_{J_s}(u_{s1})\big) \leq w_{(H,\pi,\sigma)}(v_1) \enspace.
\end{equation}

If $\mu_\theta(H_\sigma')<0$ or $\mu_\theta(J_s')<0$ for some $s\in[r]\setminus\{\sigma\}$,  we have found a graph $K'$ with $\mu_\theta(K')<0$ and $v(K')\leq \vmax$ (see \eqref{eq:size-H-sigma'} and \eqref{eq:size-J-s'}). Otherwise we have $\mu_\theta(H_\sigma')\geq 0$ and $\mu_\theta(J_s')\geq 0$ for all $s\in[r]\setminus\{\sigma\}$. We will argue later that this implies even stronger bounds on the number of vertices of $H_\sigma'$ and $J_s'$, namely
\begin{equation}
\begin{split} \label{eq:size-H-sigma'-J-s'}
    v(H_\sigma') &\leq (\vmax-1)/r \enspace,\\
    v(J_s')      &\leq (\vmax-1)/r \enspace, \quad s\in[r]\setminus\{\sigma\} \enspace. 
\end{split}
\end{equation}

We define the graph $H'=(V',E')$ as
\begin{equation} \label{eq:def-H'}
\begin{split}
  V' &:= \{v\} \cup \bigcup_{s\in[r]} V_s  \enspace, \\
  E' &:= \bigcup_{s\in[r]} (E_s \cup E_s^v)
\end{split}
\end{equation}
(see Figure~\ref{fig:lbvars1}).
This graph clearly contains the copy of $(H,\pi)$ in color $\sigma$ we are considering.

Furthermore, we define for $2\leq s\leq r$ the graphs
\begin{equation} \label{eq:def-K'}
  K_s' := \Big(V_s\cap\bigcup_{1\leq t\leq s-1}V_t,E_s\cap\bigcup_{1\leq t\leq s-1}E_t\Big) \enspace.
\end{equation}
From \eqref{eq:size-H-sigma'-J-s'} and \eqref{eq:def-K'} we conclude that $v(K_s')\leq (\vmax-1)/r\leq \vmax$, $2\leq s\leq r$. Therefore, if $\mu_\theta(K_s')<0$ for some $2\leq s\leq r$, then we have found a graph $K'$ with $\mu_\theta(K')<0$ and $v(K')\leq \vmax$. Otherwise we have
\begin{equation} \label{eq:mu-K-s'}
  \mu_\theta(K_s') \geq 0 \quad, \enspace 2\leq s\leq r \enspace.
\end{equation}

With \eqref{eq:E-sigma-v-deg} and \eqref{eq:E-s-v-deg} we obtain from \eqref{eq:def-H'} and \eqref{eq:def-K'} that
\begin{equation} \label{eq:v-e-H'}
\begin{split}
  v(H') &= 1+v(H_\sigma')+\sum_{s\in[r]\setminus\{\sigma\}} v(J_s') - \sum_{2\leq s\leq r} v(K_s') \enspace, \\
  e(H') &= e(H_\sigma')+\deg_H(v_1) + \sum_{s\in[r]\setminus\{\sigma\}} \big(e(J_s')+\deg_{J_s}(u_{s1})\big) - \sum_{2\leq s\leq r} e(K_s') \enspace.
\end{split}
\end{equation}

Combining our previous observations yields
\begin{equation} \label{eq:mu-H'-leq-lambda-H}
\begin{split}
  \mu_\theta(H')
    &\eqByM{\eqref{eq:def-mu},\eqref{eq:v-e-H'}}
     1+\mu_\theta(H_\sigma')-\deg_H(v_1)\cdot\theta
     +\sum_{s\in[r]\setminus\{\sigma\}} \big(\mu_\theta(J_s')-\deg_{J_s}(u_{s1})\cdot\theta\big)
     -\sum_{2\leq s\leq r}  \mu_\theta(K_s') \\
    &\leByM{\mathclap{\eqref{eq:mu-H-sigma'},\eqref{eq:mu-J-s'},\eqref{eq:mu-K-s'}}} \qquad
     1+\lambda_\theta(H\setminus v_1,w_{(H,\pi,\sigma)})-\deg_H(v_1)\cdot\theta
     +\underbrace{\sum_{s\in[r]\setminus\{\sigma\}} \big(\lambda_\theta(J_s\setminus u_{s1},w_{(J_s,\tau_s,s)})-\deg_{J_s}(u_{s1})\cdot\theta\big)}_{\leBy{eq:sufficient-weight-behind-v1} w_{(H,\pi,\sigma)}(v_1)} \\
    &\leq \lambda_\theta(H\setminus v_1,w_{(H,\pi,\sigma)}) + 1 + w_{(H,\pi,\sigma)}(v_1) -\deg_H(v_1)\cdot\theta
     \eqBy{eq:lambda-recursive} \lambda_\theta(H,w_{(H,\pi,\sigma)})
\end{split}
\end{equation}
which proves \eqref{eq:mu-H'-lambda-H}.
From \eqref{eq:size-H-sigma'-J-s'} and \eqref{eq:def-H'} we conclude that $v(H')\leq \vmax$.

It remains to show \eqref{eq:size-H-sigma'-J-s'}, i.e.\ that for every graph $H'$ as defined in \eqref{eq:def-H'} with $\mu_\theta(H')\geq 0$ we have $v(H')\leq (\vmax-1)/r$. For the reader's convenience, the notations used in this part of the proof are illustrated in Figure~\ref{fig:lbvars2}.

\begin{figure}
\centering
\PSforPDF{
 \psfrag{sinr}{$s\in[r]\setminus\{\sigma\}$}
 \psfrag{vdots}{$\vdots$}
 \psfrag{ldots}{$\ldots$}
 \psfrag{ddots}{\reflectbox{$\ddots$}}
 \psfrag{H}{$(H,\pi)$}
 \psfrag{Hm}{$(H\setminus v_1,\pi\setminus v_1)$}
 \psfrag{J}{$(J_s,\tau_s)$}
 \psfrag{Jm}{$(J_s\setminus u_{s1},\tau_s\setminus u_{s1})$}
 \psfrag{Thp}{\Large $\cT(H')$}
 \psfrag{Thsp}{$\cT(H_\sigma')$}
 \psfrag{Tjsp}{$\cT(J_s')$}
 \psfrag{P}{\Large $P$}
 \psfrag{Z}{$Z$}
 \psfrag{Z1b}{$\Zbar_1$}
 \psfrag{Zbb}{$\Zbar_b$}
 \psfrag{Zh}{$\Zhat$}
 \psfrag{Jsm1}{$(J_s^{-1},\tau_s^{-1})$}
 \psfrag{Jsmb}{$(J_s^{-b},\tau_s^{-b})$}
 \psfrag{colsigma}{color $\sigma$}
 \psfrag{cols}[c][c][1][90]{color $s$}
 \psfrag{colss}{color $s$}
 \psfrag{colsp}{color $s'$}
 \psfrag{walk}{\Large $\cW$}
 \psfrag{G}{$\Gamma$}
 \psfrag{Gp}{$\Gamma'$}
 \psfrag{xH}{$x_{(H,\pi,\sigma)}$}
 \psfrag{xJ}{$x_{(J_s,\tau_s,s)}$}
 \psfrag{xJmb}{$x_{(J_s^{-b},\tau_s^{-b},s)}$}
 \psfrag{xZ}{$x(Z)$}
 \psfrag{xZb}{$x(\Zbar_b)$}
 \includegraphics{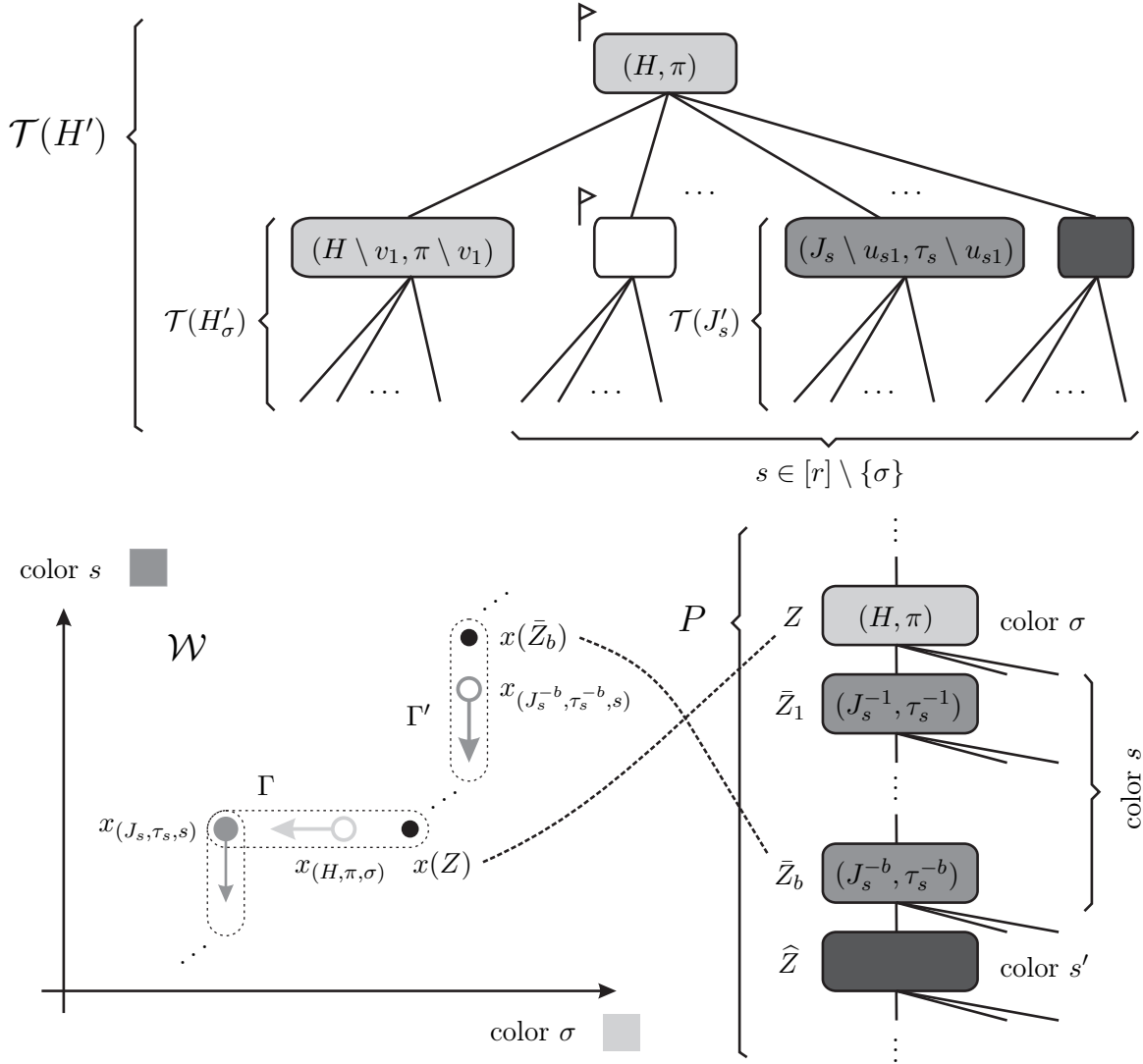}
}
\caption{Notations used in the second part of the proof of Lemma~\ref{lemma:witness-graphs}.} \label{fig:lbvars2}
\end{figure}

In the above argument we constructed the graph $H'$ containing the copy of $(H,\pi)$ in color $\sigma$ inductively from the graph $H_\sigma'$ containing the copy of $(H\setminus v_1, \pi\setminus v_1)$ in color $\sigma$ and the graphs $J_s'$ containing the copies of $(J_s\setminus u_{s1}, \tau_s\setminus u_{s1})$ in color $s$, $s\in[r]\setminus\{\sigma\}$.
We associate this inductive construction with a node-colored rooted tree $\cT(H')$, some of whose non-leaf nodes receive a special marking (we refer to it as a \emph{flag}), as follows (see the upper part of Figure~\ref{fig:lbvars2}): The nodes of $\cT(H')$ correspond to monochromatic copies of graphs from $\cS(F)$ on the board (the same copy may appear as a node multiple times). If $(H,\pi)$ is the null graph `in color $\sigma$' (recall that in this case $H'$ is the null graph as well), $\cT(H')$ consists only of this copy of $(H,\pi)$ as an isolated node which receives the color $\sigma$. 
Otherwise $\cT(H')$ consists of the copy of $(H,\pi)$ as the root node joined to $r$ subtrees, $\cT(H_\sigma')$ and $\cT(J_s')$ for all $s\in[r]\setminus\{\sigma\}$. The root node receives the color $\sigma$, and it is flagged if and only if the instance of the inequality \eqref{eq:sufficient-weight-behind-v1} corresponding to this induction step is strict. Note that the tree $\cT(H')$ captures only the logical structure of the inductive history of $H'$. Overlappings (captured by the graphs $K_s'$, $2\leq s\leq r$) are completely neglected.

Every flagged node of $\cT(H')$ corresponds to a strict inequality in \eqref{eq:sufficient-weight-behind-v1}. In this case, inequality~\eqref{eq:mu-H'-leq-lambda-H} is strict as well, with a difference of at least $\eps$ between the right and left hand side, where $\eps$ is the constant guaranteed by Lemma~\ref{lemma:sufficient-weight}. Consequently, each flagged node contributes a term of $-\eps$ to the right hand side of~\eqref{eq:mu-H'-leq-lambda-H} in the corresponding induction step. Accumulating these terms along the induction yields that
\begin{equation} \label{eq:accumulate-eps}
  \mu_\theta(H') \leq \lambda_\theta(H,w_{(H,\pi,\sigma)})-f(H')\cdot\eps\enspace,
\end{equation}
where $f(H')$ denotes the number of flagged nodes in $\cT(H')$.

By Lemma~\ref{lemma:finite-weights} we have $\lambda_\theta(H,w_{(H,\pi,\sigma)})\leq v(F)$. Thus if $\mu_\theta(H')\geq 0$, then by \eqref{eq:accumulate-eps} the tree $\cT(H')$ has at most $\lambda_\theta(H,w_{(H,\pi,\sigma)})/\eps\leq v(F)/\eps$ many flagged nodes. We will show that this bound on the number of flagged nodes of $\cT(H')$ implies the claimed bound of $(\vmax-1)/r$ on the number of vertices of $H'$. To that end, we first show that the length of any descending path in $\cT(H')$ that consists only of non-flagged nodes is bounded by a constant depending only on $F$ and $r$. We will do so by showing that any descending sequence of non-flagged nodes in $\cT(H')$ corresponds to an ascending sequence of points on the walk $\cW$ defined in Lemma~\ref{lemma:walk}.

Specifically, we assign to every non-leaf node $Z$ in $\cT(H')$ a point $x(Z)\in\RR^r$ on the walk $\cW$ as follows: Let $\sigma$ denote the color of $Z$ and $(H,\pi)$ the graph for which a copy in color $\sigma$ is represented by the node $Z$. We define $x(Z)$ as the starting point of the $\sigma$-segment that contains the point $x_{(H,\pi,\sigma)}$ on the walk $\cW$.

Consider a descending sequence $Z,\Zbar_1,\ldots,\Zbar_b,\Zhat$ of consecutive non-flagged nodes in $\cT(H')$, where $Z$ has some color $\sigma\in[r]$, all nodes $\Zbar_1,\ldots,\Zbar_b$ have the same color $s\in[r]\setminus\{\sigma\}$ and $\Zhat$ has some color $s'\in[r]\setminus\{s\}$ (see the lower right part of Figure~\ref{fig:lbvars2}). Let $(H,\pi)$ be the graph for which a copy in color $\sigma$ is represented by the node $Z$, and $(J_s,\tau_s)$, $\tau_s=(u_{s1},\ldots,u_{sc})$, the graph for which a copy of $(J_s\setminus u_{s1},\tau_s\setminus u_{s1})$ in color $s$ is represented by the node $\Zbar_1$.
Using these definitions, clearly the nodes $\Zbar_1,\ldots,\Zbar_b$ represent a sequence of nested copies of $(J_s^{-a},\tau_s^{-a})$, $a=1,\ldots,b$, in color $s$, where $(J_s^{-a},\tau_s^{-a}):=(J_s\setminus\{u_{s1},\ldots,u_{sa}\},\tau_s\setminus \{u_{s1},\ldots,u_{sa}\})$.
Let $\Gamma$ denote the $\sigma$-segment of the walk $\cW$ that contains the point $x_{(H,\pi,\sigma)}$ and $\Gamma'$ the $s$-segment that contains the point $x_{(J_s^{-b},\tau_s^{-b},s)}$ (the starting points of these segments are $x(Z)$ and $x(\Zbar_b$), respectively; see the lower left part of Figure~\ref{fig:lbvars2}). As the node $Z$ is not flagged, the corresponding instance of the inequality \eqref{eq:sufficient-weight-behind-v1} is tight. Hence, by Lemma~\ref{lemma:sufficient-weight}, we have $(J_s,\tau_s)\in\cH_s'\cup\cC(\cH_s,F)$, and if $(J_s,\tau_s)\in\cH_s'$, then $x_{(J_s,\tau_s,s)}$ is the starting point of the next $s$-segment on $\cW$ that is lower than $\Gamma$, whereas if $(J_s,\tau_s)\in\cC(\cH_s,F)$, then there is no $s$-segment on $\cW$ lower than $\Gamma$. In the first case we apply Lemma~\ref{lemma:forward-walk} to conclude that $\Gamma$ is lower than $\Gamma'$ on the walk $\cW$, in the second case this conclusion is trivially true. It follows that in any case $x(Z)$ is lower on the walk than $x(\Zbar_b)$.

Let now $P$ be a descending path in $\cT(H')$ that consists only of non-flagged nodes, and recall that our goal is to bound the length of $P$ by a constant depending only on $F$ and $r$.
We refer to a maximal sequence of consecutive nodes of the same color along $P$ as a \emph{section} in this color (in the above argument, $\Zbar_1,\ldots,\Zbar_b$ is a section in color $s$). Moreover, we call a section of $P$ \emph{internal} if it is neither the first nor the last section on $P$. By the argument above, for the last node $\Zbar$ of an internal section and the last node $Z$ of the preceding section on $P$, $x(\Zbar)$ is higher on the walk $\cW$ then $x(Z)$. As the walk $\cW$ contains at most $r\cdot |\cS(F)|$ different elements, the path $P$ can have at most $r\cdot |\cS(F)|-1$ internal sections, and at most $r\cdot |\cS(F)|+1$ sections in total (including the first and last section). As each section consists of at most $v(F)+1$ nodes, $P$ consists of at most
\begin{equation} \label{eq:length-P}
  (r\cdot |\cS(F)|+1)(v(F)+1)=:p
\end{equation}
nodes, proving that the length of $P$ is indeed bounded by a constant depending only on $F$ and $r$.

Since in total there are at most $v(F)/\eps$ many flagged nodes in $\cT(H')$, the depth of $\cT(H')$ is bounded by
\begin{equation} \label{eq:depth-T-H'}
  v(F)/\eps+(v(F)/\eps+1)p=:t \enspace,
\end{equation}
where the first term bounds the number of flagged nodes and the second term the number of non-flagged nodes.
Consequently, we have
\begin{equation*}
  v(\cT(H')) \leq 1+r+r^2+\cdots+r^t\leq r^{t+1} \enspace.
\end{equation*}
Observing that every node of $\cT(H')$ corresponds to at most $v(F)$ vertices of $H'$, we finally obtain that
\begin{equation*}
  v(H')\leq r^{t+1}\cdot v(F) \eqByM{\eqref{eq:vmax},\eqref{eq:length-P},\eqref{eq:depth-T-H'}} (\vmax-1)/r \enspace.
\end{equation*}
This justifies~\eqref{eq:size-H-sigma'-J-s'} and concludes the proof.
\end{proof}

\section{Proof of Theorem~\texorpdfstring{\ref{thm:main-2}}{4}} \label{sec:proof-thm-main-2}

We denote the board of the probabilistic process after $i$ steps by $G_i$, where $0\leq i\leq n$. We take the alternative view mentioned in Section~\ref{sec:introduction-online-setting}, in which the random edges leading from a newly added vertex to previous vertices are generated at the moment this vertex is revealed instead of at the beginning of the process. (Recall that each edge is inserted with probability $p=p(n)$ independently from all other edges.) Thus $G_i$ is an $r$-colored graph on $i$ vertices, and the underlying
uncolored graph is distributed as $G_{i,p}$.

Recall that all our asymptotic results are with respect to $n$, the number of vertices of $G_n$ or $\Gnp$. We write $f(n)\ll g(n)$ if $f(n)=o(g(n))$, $f(n)\gg g(n)$ if $f(n)=\omega(g(n))$, and $f(n)\asymp g(n)$ if $f(n)=\Theta(g(n))$.

\subsection{Lower bound} \label{sec:lower-bound-probabilistic}

The crucial ingredient for the proof of the lower bound part of Theorem~\ref{thm:main-2} is Lemma~\ref{lemma:witness-graphs} from Section~\ref{sec:lower-bound}.

\begin{proof}[Proof of Theorem~\ref{thm:main-2} (lower bound)]
Let $\theta^*=\theta^*(F,r)$ be defined as in Theorem~\ref{thm:main-1'}, and let $\alpha^*=\alpha^*(F,r)$ be a sequence from the set $[r]^{r\cdot|\cS(F)|}$ for which the minimum in the definition of $\Lambda_{\theta^*}(F,r)$ in \eqref{eq:def-Lambda} is attained. We show that the strategy $\PAINT(F,r,\theta^*,\alpha^*)$ defined in Section~\ref{sec:strategy} \aas avoids $F$ for all $n$ steps of the process if
\begin{equation} \label{eq:p-ll-p0}
  p\ll p_0(F,r,n)=n^{-1/{m_1^*(F,r)}} \eqBy{eq:m1*-via-theta*} n^{-\theta^*} \enspace.
\end{equation}
By the choice of $\alpha^*$ and the definition in \eqref{eq:def-Lambda} we have that for all colors $s\in[r]$ and all vertex orderings $\pi\in\Pi(V(F))$ there is a subgraph $H\seq F$ such that
\begin{equation} \label{eq:lb-prob-finds-unlikely-subgraph}
  \lambda_{\theta^*}(H,w_{(H,\pi|_H,s)}) \leq \Lambda_{\theta^*}(F,r) \eqBy{eq:root-of-Lambda} 0 \enspace.
\end{equation}

According to Lemma~\ref{lemma:witness-graphs} we then have for each such $(H,\pi|_H)$: if $G_n$ contains a copy of $(H,\pi|_H)$ in color $s$, then it contains a graph $K'$ with $v(K')\leq \vmax$ and $\mu_{\theta^*}(K')<0$, or a graph $H'$ with $v(H')\leq \vmax$ and
\begin{equation*}
  \mu_{\theta^*}(H')
  \leBy{eq:mu-H'-lambda-H}
  \lambda_{\theta^*}(H,w_{(H,\pi|_H,s)})
  \leBy{eq:lb-prob-finds-unlikely-subgraph}
  0 \enspace.
\end{equation*}

This yields a family $\cW=\cW(F,\pi,s,r)$ of graphs $W'$ satisfying $\mu_{\theta^*}(W')\leq 0$ and $v(W')\leq \vmax$ such that, deterministically, $G_n$ contains a graph from $\cW$ if it contains a copy of $(F,\pi)$ in color $s$.  It follows that $G_n$ contains a graph from
\begin{equation*}
  \cW^*=\cW^*(F,r):=\bigcup_{\substack{s\in[r] \\ \pi\in\Pi(V(F))}} \cW(F,\pi,s,r)
\end{equation*}
if it contains a monochromatic copy of $F$.

Moreover, since no graph in $\cW^*$ has more than $\vmax(F,r,\theta^*(F,r),\alpha^*(F,r))$ vertices, the size of $\cW^*$ is bounded by a constant only depending on $F$ and $r$.
By the definition of $\mu_{\theta^*}()$ in \eqref{eq:def-mu} and the fact that $\mu_{\theta^*}(W')\leq 0$ for all $W'\in\cW^*$, the expected number of copies of the (underlying uncolored) graphs from $\cW^*$ in $\Gnp$ is of order
\begin{equation*}
  \sum_{W'\in\cW^*} n^{v(W')} p^{e(W')}
  \llBy{eq:p-ll-p0} \sum_{W'\in\cW^*} n^{\mu_{\theta^*}(W')}
  \leq \abs{\cW^*}\cdot n^0 = \Theta(1) \enspace.
\end{equation*}
It follows with Markov's inequality that \aas $\Gnp$ contains no copy of any of the (underlying uncolored) graphs from $\cW^*$. Consequently, \aas $G_n$ contains no copy of any of the graphs from $\cW^*$ and hence no monochromatic copy of $F$. This proves the claimed lower bound on the threshold of the probabilistic process.

To prove the second part of Theorem~\ref{thm:main-2} it suffices to show that the strategy $\PAINT(F,r,\theta^*,\alpha^*)$ is an optimal strategy for Painter in the deterministic two-player game, i.e., that it is a winning strategy in the game with density restriction $d$ for any $d<m_1^*(F,r)=1/\theta^*$ (we have already argued that this strategy can be implemented as a polynomial-time algorithm in Section~\ref{sec:algorithms} and Remark~\ref{remark:priority-list-from-strategy}). Fix some $0<d<1/\theta^*$ and define $\theta:=1/d>\theta^*$. Suppose Painter plays according to the strategy $\PAINT(F,r,\theta^*,\alpha^*)$ in the game with density restriction $d$ and suppose for the sake of contradiction that the game ends with a monochromatic copy of $F$.
Then as before it follows from Lemma~\ref{lemma:witness-graphs} that the board contains a graph $K'$ with
\begin{equation} \label{eq:mu-K'-neg}
  \mu_{\theta^*}(K')<0
\end{equation}
or a graph $H'$ with
\begin{equation} \label{eq:mu-H'-nonpos}
  \mu_{\theta^*}(H')\leq 0
\end{equation}
(note that $H'$ contains at least one vertex and as a consequence of \eqref{eq:mu-H'-nonpos} and the definition in \eqref{eq:def-mu} also at least one edge; similarly, $K'$ contains at least one edge as a consequence of \eqref{eq:mu-K'-neg}). Using that $\theta>\theta^*$ it follows from \eqref{eq:mu-K'-neg}, \eqref{eq:mu-H'-nonpos} and the definition in \eqref{eq:def-mu} that in any case the board contains a graph $W'$ (with $v(W')\geq 1$) satisfying $\mu_{\theta}(W')<0$, or equivalently, $e(W')/v(W')>1/\theta=d$, violating the given density restriction.
\end{proof}

\subsection{Upper bound} \label{sec:upper-bound-probabilistic}

As in the proof of Lemma~\ref{lemma:beta*} we identify Builder's
strategies in the deterministic two-player game with $r$ colors with
finite $r$-ary rooted trees, where each node at depth $k$ of such a
tree is an $r$-colored graph on $k$ vertices, representing the board
after the $k$-th step of the
game.

Note that in this formalization, a given tree $\cT$ represents a
generic strategy for Builder (in the deterministic game with $r$
colors) that may or may not satisfy a given density restriction $d$,
and that can be thought of as a strategy for the~`$F$-avoidance' game
for any given graph~$F$. We say that $\cT$ is a \emph{winning
strategy} for Builder in a specific $F$-avoidance game if and only
if every leaf of~$\cT$ contains a monochromatic copy of $F$. We say
that a Builder strategy $\cT$ is a \emph{legal strategy} in the game
with density restriction $d$ if and only if $e(H)/v(H)\leq d$ for every
subgraph $H$ of every node $B$ in $\cT$.

When we say that $G_i$, the board of the probabilistic process after $i$ steps,
contains a copy of some $r$-colored graph $B$ (e.g.\ a node of some Builder
strategy $\cT$) we mean that there is a subgraph of $G_i$ that is
isomorphic to $B$ as a \emph{colored} graph.

The upper bound part of Theorem~\ref{thm:main-2} is an immediate
consequence of the following lemma.

\begin{lemma}[Random process reproduces Builder strategy] \label{lemma:ub-prob-induction-done} Let $r\geq 2$ be
  a fixed integer, let $d>0$ be a fixed real number, and let $\cT$ represent
  an arbitrary legal strategy for Builder in the deterministic game
  with $r$ colors and density restriction $d$.

  If\/ $p \gg n^{-1/d}$, then regardless of the online coloring strategy employed, \aas $G_n$ contains a copy of a leaf of~$\cT$.
\end{lemma}

\begin{proof}[Proof of Theorem~\ref{thm:main-2} (upper bound)]
  By Theorem~\ref{thm:main-1} there exists a legal winning strategy
  $\cT$ for Builder in the deterministic $F$-avoidance game with $r$
  colors and density restriction $d=m^*_1(F,r)$. As each leaf of $\cT$
  contains a monochromatic copy of $F$, applying
  Lemma~\ref{lemma:ub-prob-induction-done} to~$\cT$ yields that if
  $p\gg p_0(F,r,n)=n^{-1/m^*_1(F,r)}$, then \aas $G_n$ contains a monochromatic copy
  of $F$, regardless of the online coloring strategy employed, which is exactly the upper
  bound statement of Theorem~\ref{thm:main-2}.
\end{proof}

In order to prove Lemma~\ref{lemma:ub-prob-induction-done}, we shall
show the following more technical statement by induction on~$k$.

\begin{lemma}[Random process reproduces Builder strategy step by step] \label{lemma:ub-prob-induction} Let $r\geq 2$ be a fixed
  integer, let $d>0$ be a fixed real number, and let $\cT$ represent
  an arbitrary legal strategy for Builder in the deterministic game
  with $r$ colors and density restriction $d$.

  If\/ $p\gg n^{-1/d}$, then for any integer $k\geq 1$ the
  following is true. Regardless of the online coloring strategy employed, \aas one of the following two
  statements holds:
  \begin{itemize}
  \item $G_n$ contains a copy of a leaf of~$\cT$, or
  \item there is a node $B$ at depth $k$ in~$\cT$ such that $G_n$
    contains\/ $\Omega(n^{v(B)}p^{e(B)})$ many copies of $B$.
  \end{itemize}
\end{lemma}

The second property of Lemma~\ref{lemma:ub-prob-induction} is meaningful since,
due to the assumption that $\cT$ is a legal strategy for Builder in the game with
density restriction $d$, we have
\begin{equation*}
  e(B)/v(B) \leq m(B)\leq d\enspace,
\end{equation*}
which yields with $p\gg n^{-1/d}\geq n^{-v(B)/e(B)}$ that
\begin{equation*}
  n^{v(B)}p^{e(B)} \gg 1\enspace.
\end{equation*}

\begin{proof}[Proof of Lemma~\ref{lemma:ub-prob-induction-done}]
  Set $k:=\depth(\cT)+1$ in Lemma~\ref{lemma:ub-prob-induction}.
\end{proof}

It remains to prove Lemma~\ref{lemma:ub-prob-induction}.

\begin{proof} [Proof of Lemma~\ref{lemma:ub-prob-induction}]
  We proceed by induction on $k$. For the induction base $k=1$, note that each of the $r$ nodes
  $B$ at depth 1 in $\cT$ consists simply of an isolated vertex, colored in one of the
  $r$ available colors. Clearly, $G_n$ contains
  at least $n/r=\Omega(n)$ copies of one of these by the pigeonhole principle.

  For the induction step we employ a two-round approach. That is, we
  divide the process into two rounds of equal length $n/2$ (\wolog
  we assume $n$ to be even) and analyze these two rounds
  separately. Denoting the vertices added throughout the process by
  $v_1,\ldots,v_n$, the first round consists of adding the vertices
  $v_1,\ldots,v_{n/2}$ together with the corresponding random edges.
  At the end of the first round, we thus obtain
  a graph $G_{n/2}$, to which we can apply the induction hypothesis
  and some standard random graph arguments. The second round consists
  of adding the vertices $v_{n/2+1},\ldots,v_n$ (together with the
  corresponding random edges).
  Using a variance calculation, we show that conditional on a `good'
  first round, the second round turns out as claimed. (In fact, our
  argument does not make use of any edges added between vertices
  of the set $\{v_{n/2+1},\dotsc,v_n\}$.)

  By the induction hypothesis, if the graph $G_{n/2}$ does not contain
  a copy of a leaf of $\cT$ (in which case we are done), \aas it
  contains a family of
  \begin{equation} \label{eq:M} M\asymp n^{v(B^\circ)}p^{e(B^\circ)}
  \end{equation}
  copies of some graph $B^\circ$ corresponding to a non-leaf node at
  depth $k-1$ in $\cT$. We label these copies $B^\circ_i$, $1\leq i\leq
  M$. Let $B$ denote the graph obtained from $B^\circ$ by adding a new
  vertex $v$ to it together with edges connecting $v$ to $B^\circ$ as
  prescribed by Builder's next move specified by $\cT$ (so $v$ is uncolored
  in $B$, but assigning it one of the $r$ available colors yields exactly
  one of the children of $B^\circ$ in $\cT$).

  For each copy $B_i^\circ$, $1\leq i\leq M$, and each vertex $v_\ell$,
  $n/2+1\leq\ell\leq n$, we fix a set $E_{i,\ell}$ of $\deg_B(v)$ many vertex pairs
   such that if the elements of $E_{i,\ell}$ are actual edges
  generated in the second round, then $v_\ell$ together with those edges
  completes $B_i^\circ$ to a copy of $B$.
  We let $Z_{i,\ell}$ be the indicator variable for the event that the elements
  of $E_{i,\ell}$ are generated as edges in the second round. Let
  \begin{equation*}
    Z:=\sum_{i=1}^{M} \sum_{\ell=n/2+1}^{n} Z_{i,\ell}\enspace,
  \end{equation*}
  and note that by the pigeonhole principle at least $Z/r$ many
  copies of one of the children of $B^\circ$ in $\cT$ are created.
  Thus the second condition of the lemma is satisfied if we show that \aas
  \begin{equation} \label{eq:Z} Z \asymp n^{v(B)} p^{e(B)} \enspace.
  \end{equation}
  We will do so by the methods of first and second moment.

  We clearly have
  \begin{align*}
    \Pr[Z_{i,\ell}=1] =
    p^{\abs{E_{i,\ell}}} =
    p^{\deg_B(v)}\enspace,
  \end{align*}
  and, conditioning on the first round satisfying the induction
  hypothesis,
  \begin{equation} \label{eq:EZ} \EE[Z] = M\cdot n/2 \cdot p^{\deg_B(v)}
    \asBy{eq:M} n^{v(B)} p^{e(B)} \enspace.
  \end{equation}

  In the following, we slightly abuse notation and write $B$ also for the
  \emph{uncolored} graph underlying $B$.
  Let~$\cD$ denote the family of all (uncolored) graphs $D$ that can
  be constructed by considering the union of two copies of $B$ intersecting
  in at least two vertices, one of which must be the vertex $v$ (we again
  slightly abuse notation in the following and refer to the corresponding
  vertex in each such graph $D$ as $v$).
  For any $D\in\cD$, we denote by $D^\circ$ the graph obtained by removing
  $v$ from~$D$.

  To calculate the variance of $Z$, observe that the variables
  $Z_{i,\ell}$ and $Z_{j,\ell'}$ are independent whenever $\ell\neq \ell'$
  or $B^\circ_i\cap B^\circ_j=\emptyset$. Hence such pairs can
  be omitted, and we have
  \begin{equation} \label{eq:variance1}
    \begin{split}
      \var[Z] &= \sum_{i,j=1}^{M} \, \sum_{\ell,\ell'=n/2+1}^{n}(\EE[Z_{i,\ell} Z_{j,\ell'}]-\EE[Z_{i,\ell}]\EE[Z_{j,\ell'}]) \\
      &\leq \sum_{\begin{subarray}{l}i,j=1,\dotsc,M\colon\\B^\circ_i\cap B^\circ_j\neq\emptyset\end{subarray}} \sum_{\ell=n/2+1}^{n} \Pr[Z_{i,\ell}=1\wedge Z_{j,\ell}=1] \\
      &= \sum_{D\in\cD} \, \sum_{\begin{subarray}{l}i,j=1,\dotsc,M\colon\\B^\circ_i\cap B^\circ_j=D^\circ\end{subarray}} \smash[t]{\sum_{\ell=n/2+1}^{n}}\, p^{\abs{E_{i,\ell}\cup E_{j,\ell}}} \\
      &\leq \sum_{D\in\cD} M_{D^\circ} \cdot \Theta(1) \cdot np^{\deg_D(v)}
      \enspace,
    \end{split}
  \end{equation}
  where $M_{D^\circ}$ denotes the total number of copies of
  \(D^\circ\) in (the underlying uncolored graph of) $G_{n/2}$.
  By definition of $\cD$, each $D\in\cD$ satisfies
  \begin{equation} \label{eq:evD}
    \begin{split}
      v(D^\circ) &= 2 v(B) - v(J) - 1\enspace,\\
      e(D^\circ) &= 2 e(B) - e(J) - \deg_D(v)
    \end{split}
  \end{equation}
  for some subgraph $J\seq B$. Moreover, since we assumed that $\cT$
  is a legal strategy for Builder in the game with density restriction
  $d$, we have
  \begin{equation*}
    e(J)/v(J)\leq m(B)\leq d\enspace,
  \end{equation*}
  which yields with $p\gg n^{-1/d}\geq n^{-v(J)/e(J)}$ that
  \begin{equation} \label{eq:J}
    \begin{split}
      n^{v(J)}p^{e(J)} \gg 1\enspace.
    \end{split}
  \end{equation}

  Thus the expected number of copies of $D^\circ$ in (the underlying
  uncolored graph of) $G_{n/2}$ is
  \begin{equation*}
    \binom{n}{v(D^\circ)}\cdot \Theta(1)\cdot p^{e(D^\circ)}
    \asBy{eq:evD} n^{2v(B)-v(J)-1} p^{2e(B) - e(J) - \deg_D(v)} \llBy{eq:J} n^{2v(B)-1} p^{2e(B) -\deg_D(v)}\enspace.
  \end{equation*}
  and Markov's inequality implies that
  \begin{equation} \label{eq:MJ}
    M_{D^\circ}\ll n^{2v(B)-1} p^{2e(B) - \deg_D(v)}
  \end{equation}
  a.a.s. As moreover the number of graphs in $\cD$ is bounded by a constant
  depending only on \(\cT\), \aas \eqref{eq:MJ} holds for all $D\in\cD$
  simultaneously.

  Thus, conditioning on the first round satisfying the
  induction hypothesis (cf.~\eqref{eq:M}), and
  \eqref{eq:MJ} for all $D\in\cD$, we obtain from \eqref{eq:variance1} that
  \begin{align*}
    \var[Z] & \llBy{eq:MJ} \sum_{D\in\cD} \left(n^{v(B)} p^{e(B)}\right)^2
    \stackrel{\eqref{eq:EZ}}{\asymp} \EE[Z]^2\enspace.
  \end{align*}
  Chebyshev's inequality now yields that \aas the second round
  satisfies \eqref{eq:Z}. This implies that there is at least the
  claimed number of copies of one of the children of $B^\circ$ in $G_n$,
  as discussed.
\end{proof}

\section*{Acknowledgement}
We thank Michael Belfrage for many inspiring discussions about online Ramsey games.

\bibliographystyle{plain}
\bibliography{refs}

\end{document}